\documentclass[9.5pt,a4paper]{amsart}
\usepackage{amscd,amssymb,amsopn,amsmath,amsthm,graphics,amsfonts,enumerate,verbatim,calc}
\usepackage[dvips]{graphicx}

\usepackage{amssymb,amsmath, color}

\headsep=1cm \numberwithin{equation}{section}
\newtheorem{theorem}{Theorem}[section]
\newtheorem{lemma}[theorem]{Lemma}
\newtheorem{notation}[theorem]{Notation}
\newtheorem{proposition}[theorem]{Proposition}
\newtheorem{corollary}[theorem]{Corollary}
\newtheorem{claim}[theorem]{Claim}

\theoremstyle{definition}
\newtheorem{definition}[theorem]{Definition}
\theoremstyle{remark}
\newtheorem{remark}[theorem]{Remark}
\newtheorem{fact}[theorem]{Fact}
\newtheorem{example}[theorem]{Example}
\newtheorem{observation}[theorem]{Observation}
\newtheorem{discussion}[theorem]{Discussion}
\newtheorem{question}[theorem]{Question}
\newtheorem{conjecture}[theorem]{Conjecture}
\newtheorem{acknowledgement}{Acknowledgement}

\newcommand{\Ass}{\operatorname{Ass}}
\newcommand{\im}{\operatorname{im}}

\newcommand{\len}{\ell}
\newcommand{\grade}{\operatorname{grade}}

\newcommand{\Spec}{\operatorname{Spec}}

\newcommand{\rad}{\operatorname{rad}}
\newcommand{\cd}{\operatorname{cd}}

\newcommand{\Ht}{\operatorname{ht}}
\newcommand{\pd}{\operatorname{p.dim}}

\newcommand{\engrad}{\operatorname{end}}

\newcommand{\Proj}{\operatorname{Proj}}
\newcommand{\Syz}{\operatorname{Syz}}

\newcommand{\rank}{\operatorname{rank}}
\newcommand{\m}{\mathfrak m}

\newcommand{\Ext}{\operatorname{Ext}}

\newcommand{\Supp}{\operatorname{Supp}}
\newcommand{\Tor}{\operatorname{Tor}}

\newcommand{\Hom}{\operatorname{Hom}}

\newcommand{\Ann}{\operatorname{Ann}}

\newcommand{\depth}{\operatorname{depth}}

\newcommand{\coker}{\operatorname{coker}}

\newcommand{\vpl}{\operatornamewithlimits{\varprojlim}}

\newcommand{\fm}{\frak{m}}
\newcommand{\fp}{\frak{p}}

\newcommand{\fa}{\frak{a}}

\newcommand  {\shL}     {\mathcal{L}}
\newcommand{\NN}{\mathbb{N}}

\newcommand{\PP}{\mathbb{P}}

\usepackage[utf8]{inputenc}
\usepackage{amscd,amssymb,amsopn,amsmath,amsthm,graphics,amsfonts,enumerate,verbatim,calc}
\usepackage{booktabs}
\theoremstyle{plain}

\newtheorem{motivation}[theorem]{Motivation}
\theoremstyle{definition}

\newcommand{\bR}{\mathbf{R}}
\newcommand{\bS}{\mathbf{S}}
\newcommand{\bT}{\mathbf{T}}
\newcommand{\bk}{\mathbf{k}}
\newcommand{\bK}{\mathbf{K}}

\newcommand{\bC}{\mathbf{C}}
\newcommand{\bF}{\mathbf{F}}
\newcommand{\bD}{\mathbf{D}}
\newcommand{\bI}{\mathbf{I}}

\newcommand{\bAq}{\mathbf{A}_q}
\newcommand{\bShat}{\widehat{\mathbf{S}}}

\newcommand{\cM}{\mathcal{M}}
\newcommand{\cN}{\mathcal{N}}
\newcommand{\cL}{\mathcal{L}}
\newcommand{\cE}{\mathcal{E}}
\newcommand{\cG}{\mathcal{G}}
\newcommand{\cF}{\mathcal{F}}
\newcommand{\cC}{\mathcal{C}}
\newcommand{\cH}{\mathcal{H}}

\newcommand{\mm}{\mathfrak{m}}
\newcommand{\nn}{\mathfrak{n}}\newcommand{\bn}{\mathfrak{n}}
\newcommand{\DD}{D}
\newcommand{\tr}{\operatorname{tr}}
\newcommand{\eend}{\operatorname{end}}

\newcommand{\odd}{\operatorname{odd}}

\DeclareMathOperator{\Tr}{Tr}

\newenvironment{PROOF}[2][\proofname.]
{\begin{proof}[#1]\renewcommand{\qedsymbol}{\textsquare$_{\rm #2}$}}
	{\end{proof}}

\begin{document}

\author[]{mohsen asgharzadeh }

\address{}
\email{mohsenasgharzadeh@gmail.com}

\title[ ]
{On the (LC) conjecture }


\begin{abstract}
We prove the (LC) conjecture of Hochster and Huneke
in some non-trivial cases. This has several applications.
Recently, Brenner and Caminata answered a numerical
evidence due to  Dao and Smirnov on the shape  of  generalized
Hilbert-Kunz functions of smooth curves.  As  applications, we first
reprove this by a short argument. Then
we give a proof of second numerical  evidence predicted by  Dao and Smirnov on the shape  of  generalized
Hilbert-Kunz functions of nodal curves.
Thirdly,  we answer  a question posted by Vraciu on the (LC) property of  a  proposed ring.
Inspiring with the (LC) property, we present a connection to  the  stability theory. This leads  us to investigate the stability and the strong semistability of  the sheaf of relations on $\{x^2,y^2,z^2\}$ over the
Klein's quartic curve. This answers questions of Brenner. After presenting a connection from (LC) to the $F$-threshold, we answer a question posted by Huneke et al.
Additional applications and examples are given.
\end{abstract}

\maketitle

\smallskip

\section{ Introduction}

Throughout this paper $R:=\bigoplus_{n\geq0}R_n$ is a standard graded algebra over a field  $R_0$ of prime characteristic $p>0$, $\fm:= \bigoplus_{n>0}R_n$  is the irrelevant ideal and  $I\lhd R$ is a homogeneous ideal, otherwise specializes.   For  each $n\in \mathbb{N}$,  set $q:=p^n$ and denote the $n$-th Frobenius power by $I^{[q]}:=(x^q:x\in I)$. By $H^0_{\fm}(R/I^{[q]})$ we mean the elements of $R/I^{[q]}$ that annihilated by some  powers of $\fm$. This \textit{section} functor was introduced by Grothendieck. We call them \textit{Local Cohomology} modules.  By definition, there is $f(q)\in \mathbb{N}_0$ depending on $q$
such that  $\fm^{f(q)}H^0_{\fm}(R/I^{[q]})=0$. The  $(LC)$ conjecture claims that $f(q)$ is of linear type:

\begin{conjecture}There is  some $b\in \mathbb{N}_0$ that does not depending  to $q$ such that $\fm^{bq}H^0_{\fm}(R/I^{[q]})=0 \quad\forall q. $
\end{conjecture}
Hochster and Huneke  \cite{HH} introduced the (LC) conjecture in relation to the \textit{localization problem}.
In  \cite{t}, Brenner and Monsky
found a counterexample to the localization problem. In this paper
we investigate the (LC) property in some non-trivial cases and present some of its applications.

The unsolved question in this area is  whether \textit{weakly F-regular} rings are
\textit{F-regular}. By a result of Huneke, the (LC) property in  dimension one answers this question affirmatively.
We now explain how the (LC) condition  arises from Hilbert-Kunz theory.
By $f_{gHK}^{R/I}(n)$ we mean the length of $H^0_{\fm}\left(R/I^{[q]}\right)$ as an $R$-module and by $e_{gHK}(R/I)$ we mean $\lim_{n\to\infty}\frac{f_{gHK}^{R/I}(n)}{p^{n\dim R}},$
if the limit exists  and put $\infty$ otherwise.
 Following Dao and Smirnov \cite{dao}, we call them the \textit{generalized Hilbert-Kunz function} and the generalized Hilbert-Kunz multiplicity, respectively. They proved  such a limit exists under some conditions. Very recently, Vraciu  \cite{V1} observed that such a limit exists
for a subclass of rings that satisfy in the (LC) property.
Our motivation comes from a paper of Brenner
 on the \textit{irrational} possibility of Hilbert-Kunz multiplicity. He proved this by observing that $$e_{gHK}(R/(a,b))\notin\mathbb{Q} \quad(\star),$$
where $R$ is the coordinate ring of a \textit{$K3$-surface} in $\PP^3$, see \cite{Holger3}.  Hilbert-Kunz multiplicity  introduced by  Kunz \cite{kunz}, and  Monsky   proved that such a limit exists \cite{mon}.

 The (LC) conjecture is true in 2-dimensional normal rings by  Vraciu \cite{V2}.
In the graded  rings and in the case $\dim R/I=1$ the (LC)  checked by Vraciu \cite{V2} and Huneke \cite{Hun}. Also, we cite the work of Katzman as another related source \cite{kat2}.

The organization of this paper is as follows. In Section 2 we summarize  some
known results  that we need.  Section 3 introduces the  (LC) property with respect to a family of ideals. The focus here is limited
to the Frobenius powers. However, we catch a glimpse to certain families of ideals.
For any $X\subset \Spec(R),$ set
$X^i:=\{\fp\in X\mid\dim R/ \fp\geq i\}.$
 Section 4 is devoted to show:

\begin{theorem}\label{lq} Let $R$ be a standard graded Cohen-Macaulay  over a field  $R_0$ of prime characteristic $p>0$ and let  $I\lhd R$ be homogeneous. Suppose one of the following holds:
  \begin{enumerate}
 \item[(i)] The ring is normal and $\dim R<4$.
\item[(ii)]  The ring is normal and 	$\operatorname{Proj}(R/I)^{2}\subseteq \operatorname{Reg}(R)$.
\item[(iii)] $\pd(R/I)<\infty$ .
\end{enumerate}
Then $\fm^{bq}H^0_{\fm}(R/I^{[q]})=0$ for some $b$ that does not depending  to $q$.
\end{theorem}

This motivates to check the (LC) property over rings with \textit{quotient singularity}, see Remark \ref{qu}.
One may recover this by  Discussion \ref{ffrt}:
The (LC) property holds for  rings of \textit{finite $F$-representation type}.
 In Section 5 we apply Theorem \ref{lq} to deduce the following which was asked by Vraciu \cite[Page 3]{V1}:

\begin{corollary}
The $R$-module  $R/(a,b)$ in $(\star)$ satisfies in the (LC)
condition.
\end{corollary}

This implies that
$f^{R/I}_{gHK}$ is a
linear combination with integer coefficients  of
Hilbert-Kunz functions of $\fm$-primary ideals. By the help of this and $(\star)$ Vraciu obtains a more direct  $\fm$-primary ideals with irrational Hilbert-Kunz multiplicities.
In Section 6, we first reprove the main result of \cite{AB} (at the level of ideals) by a short argument. This was asked in \cite{dao} and regards as an  application of the (LC)
condition:

\begin{corollary}\label{a2}
Over $2$-dimensional normal graded domains over  a field  $F$ of prime characteristic, one has:
\begin{enumerate}
\item[i)]
$f_{gHK}^{R/I}(n) = e_{gHK}(R/I)q^2 + \gamma(q)$,
where $e_{gHK}(R/I)\in \mathbb{Q}$ and $\gamma(q)$ is a bounded function,
\item[ii)]if $F=\overline{\mathbb{F}}_p$, then $\gamma(q)$ is an eventually
\textit{periodic} function.
\end{enumerate}
\end{corollary}

In order to present the next application, we look at the non-normal ring  $R := \frac{k[[x, y, t]]}{(x^3 + txy + y^3)}$ and set $M := R/(x, y)$. For e.g. $k = \mathbb{F}_{11}$ recall from \cite[Example 6.3(3)]{dao} the following numerical evidence:
"$f^{M}_{gHK}(q)=\frac{q^2+2q-\gamma(q)}{3}$ and the formula  seems to depend on whether $q = 1 \mod 3$."
 We apply  the (LC) condition to prove this numerical evidence:

\begin{corollary}\label{node} Suppose $R$ is the coordinate ring of  a \textit{nodal} cubic plane projective curve  over an algebraically  closed field of prime characteristic $p$ and $I\lhd R$ is graded. Then
$f_{gHK}^{R/I}( q) = \mu q^2 +aq- r,$ where $r$  is an integer that depends on  $q \mod 3$.
\end{corollary}

The next subject of  Section  6 is Theorem \ref{irr}: As an application of $(LC)$, we quickly derive $e_{gHK}(-)\in \mathbb{Q}$  for the coordinate ring of singular plane projective curves. In the sequel we need the following quantities:
\begin{definition}\label{num}For an ideal $I$  of a ring $R$, set:
\begin{enumerate}
\item[i)]$b(I):=\inf\{b\in \mathbb{N}_0:\fm^{bq}H^0_{\fm}(R/I^{[q]})=0 \quad \forall q\},$
\item[ii)]$c(I):=\inf\{b\in \mathbb{N}_0:\fm^{bq}H^0_{\fm}(R/I^{[q]})=0 \quad \forall q\gg 0\}$,
\item[iii)] $d(I):=\inf\{b\in \mathbb{N}_0:\fm^{bn}H^0_{\fm}(R/I^n)=0 \quad \forall n\gg 0\}$.
\end{enumerate}
\end{definition}
Suppose $I \lhd R$ is  generated by linear
forms and $\fm$-primary. Let us recall from \cite[Proposition 0.5]{ei} that $d(I)\leq2$.
In Section 7 first we show:

 \begin{corollary} \label{sy}Let $I=(f,g)$ be a homogeneous 2-generated prime ideal of the ring $R$ in Corollary \ref{a2}
 of height one and let $s\in \fm^{c}\setminus I$ be homogeneous, where $c:=b(I)$ is the (LC) exponent. Then at least
one of  the syzygy bundles $\{\Syz(f,g,s),\Syz(f,g,s^2), \Syz(f,g,s^4)\}$ is not strongly semistable.
\end{corollary}
Let $\mathcal{C}$ be a degree four plane curve. Brenner  proved that $\mathcal{V}:=\Syz_{\mathcal{C}}(x^2,y^2,z^2)$ is semistable, see \cite[Lemma 7.1]{brennercomputation} and he posted two questions on the stability and the strongly semistability of $\mathcal{V}$ over
\textit{Klein's quartic} $zx^3 + xy^3 + yz^3=0$.
We note that \textit{Mumford's stability} has essential applications to Hilbert-Kunz theory. This contribution was made by Brenner and Trivedi.
Reversely, we  answer Brenner's  questions:

\begin{example}\label{exa}
Let $R:=\overline{\mathbb{F}}_2[x,y,z]/(zx^3 + xy^3 + yz^3)$ and $\mathcal{C}$ be the corresponding curve.  The following holds:
i) $\mathcal{V}$ is stable, ii) $\mathcal{V}$ is not strongly semistable.
\end{example}

In Section 8 and when $I$ is $\fm$-primary, we give a connection from $c(I)$ to  $c^I(\fm)$ the $F$-\textit{threshold} of $I$ with respect to $\fm$:   $c(I)-1\leq c^I(\fm)\leq c(I)$, see Observation 8.1 for a more precise statement. 
Let us recall the following application of  $F$-threshold from \cite[Example 3.4]{Mi}:
"Let $(R, \fm)$ be a regular local ring of characteristic $p > 0$ with
$\dim(R) =d $ and let $J$ be an ideal of $R$ generated by a full system of parameters.
We define $a$ to be the
the maximal integer $n$ such that $\fm^n \nsubseteq J$.
Then $\fm^s \subseteq \overline{J}$ if
and only if $s\geq \frac{a}{d} + 1$."
Then,  Huneke-Mustata-Takagi-Watanabe send two questions (\cite[Questions 3.5]{Mi}):
Does this statement hold in a more general setting? Can we replace
"regular" by "Cohen-Macaulay"?

\begin{example}\label{mi}
The above questions have  negative answers.
\end{example}

In  Section 9 we connect $d(\sim)$ to the  \textit{Waldschmidt constant}.
In Section 10 we focus on  $R:=\overline{\mathbb{F}}_p[X_1,\ldots,X_m]$ and on a graded ideal $I\vartriangleleft R$. We show in Proposition \ref{pro} that
$f_{gHK}^{R/I}(n)=e_{gHK}(R/I)q^m (\natural)$. Also,
 $e_{gHK}(R/I)$ realizes as a length of a module. Compare this with the irrational possibility of
 $\lim_{n\to\infty}\frac{H^0_{\fm}(R/I^n)}{n^{\dim R}}$, see \cite{Cu}.
 One  may regards $(\natural)$ as a Frobenius version of a question of Herzog, see Question \ref{herzog}. We present a reformulation of Question \ref{herzog} by the help of  an algorithm, see Corollary \ref{herzoga}.
 
 Section 11 goes on to investigate  $b(I)$ and $c(I)$.  Let $\Gamma$ be a family of two-generated ideals of fixed degrees in the coordinate ring of  smooth plane projective curve.
We apply in Remark \ref{holgercomp} a  computational method of Brenner  to give a bound on  $\{b(I): I\in \Gamma\}$.
However, we give an example such that $\sup\{b(I^n)\}=\infty$ as $n$ varies. In a similar vein:
\begin{example} \label{remhol}
There is no polynomial function  as $F$ such that
$F(e_{gHK}(R/I), e_{gHK}(R/J)) = e_{gHK}(R/IJ)$
even if $I$ and $J$ projectively have the same \textit{closure operation}.
 \end{example}
 This demonstrates a different behavior between $e_{gHK}(\sim)$ and $e_{HK}(\sim)$
among several similarities. This raises through a question of Brenner and Caminata \cite{AB}.
The (LC) exponents live in the shadow of a degree data coming from the ideal and the ring, at least in the examples presented in Section 11.

In our final section we collect few remarks in the local situation.

\section{  Preliminaries}
All rings in this paper are commutative, Noetherian and of prime characteristic $p$.
In the next subsections, we set up the basic foundation for the  paper. In particular,
let $A$ be a  ring with an ideal $\frak a$ with a generating set
$\underline{a}:=a_{1}, \ldots, a_{r}$.
 By $H_{\underline{a}}^{i}(M)$ we mean the $i$-th cohomology of
\textit{$\check{C}ech$} complex of  a module $M$ with respect  to $\underline{a}$. This is independent of the choose
of the generating set. For simplicity, we denote it by $H_{\frak a}^{i}(M)$.
By $\mathbb{N}_0$ we mean $ \mathbb{N}\cup\{0\}$.
The grade of
$\frak a$ on  $M$ is defined by $$\grade_{A}(\frak a,M):=\inf\{i\in
\mathbb{N}_0|H_{\frak a}^{i}(M)\neq0\}.$$We use $\depth (M)$, when we deal with the maximal ideal of $\ast$-local rings. By \textit{Grothendieck's vanishing theorem},  we have $H_{\frak a}^{i}(M)=0$ for all $i>\dim M.$

\begin{definition} \label{sat}
 Let $I$ be an ideal of a (graded) local ring $(R,\fm)$.
There  is an integer $t\in \mathbb{N}_0$ such that $\bigcup(I:\fm ^m)=(I:\fm ^t)$. Thus
$H_{\frak m}^0(R/I)=\frac{(I:\fm ^t)}{I}.$ The ideal
$I^{sat}:=(I:\fm ^t)$ is called the saturation of $I$.
\end{definition}

\begin{fact}\label{sat}
The saturation of $I$ computed as the intersection of all primary to nonmaximal
prime ideals associated to  $I$.
\end{fact}

Denote the $n$-the symbolic power of $I$ by $I^{(n)}:=\bigcap_{
\fp\in \Ass(I)}(I^nR_{\fp}\cap R)$.

\begin{fact}
Suppose $\fp$  is a prime ideal of dimension one. Recall that $\fp^{(n)}/ \fp^n=H^0_{\fm}(R/\fp^n)$.
\end{fact}

The assignment $a\mapsto a^p$ defines a ring homomorphism $F:A\to A$.  By $F^n(A)$, we mean $A$ as a group equipped with  left and right scalar multiplication from $A$
given by
$$a.r\star b = ab^{p^n}r, \  \ where \ \ a,b\in A  \ \ and \  \ r\in F^n(A),$$
Now we recall the \textit{Peskine-Szpiro }functor. For an $A$-module $M$, set $F^n(M):= F^n(A)\otimes M.$
The left $R$-module structure of $F^n(A)$ endows $F^n(M)$ with a left $R$-module structure
such that $c(a\otimes x)=  (ca)\otimes x$. For an $R$-linear map  $\varphi: M\longrightarrow N$ one
considers $F^n(\varphi)$ to be the $R$-linear map $id_{F^n(A)}\otimes\varphi$.
 The assignment
 $a\otimes (r+\frak a)\mapsto ar^{p^n}+ \frak a^{[p^n]}$ defines an isomorphism  $\varphi:F^n(A/ \frak a)\longrightarrow A/ \frak a^{[p^n]}.$

\begin{remark}\label{g}
 Suppose $A$ is graded. To make $\varphi$ degree-preserving, define  a new grading on $A/ \frak a^{[p^n]}$ given by $\deg^{new}(r+\frak a^{[p^n]})=\frac{\deg(r)}{p^n}.$
Having this grading (resp. the usual grading) in mind, we denote the terms of $(\sim)$ of degree greater or equal than $\ell$ by $(\sim)_{\succcurlyeq \ell}$ (resp. $(\sim)_{\geq \ell}$).
Then we have $$H^0_{\fm}(F^n(A/ \frak a))_{\succcurlyeq b}=0\Longrightarrow H^0_{\fm}(A/\frak a^{[q]})_{\geq bq}=0.$$
\end{remark}

\begin{fact}\label{kunz}
Over regular  rings the Frobenius map is flat, see \cite[Corollary 8.2.8]{BH}.
\end{fact}

\begin{fact}\label{locali}For any multiplicative closed set $S$,
$S^{-1}F^n(M)\simeq F^n(S^{-1}M)$, see \cite[Proposition 8.2.5]{BH}.\end{fact}

The set $A^\circ$ denotes the elements in $A$ which are not in any minimal prime ideals.
 Also, we note that  $\frak a^{[p^n]}=(a_1^{p^n},\ldots,a_r^{p^n})$ is independent of the choose
of the generating set.
Recall from  \cite{HH} that the tight closure of  $\frak a$ is:
$\fa^{\ast}:=\{x\in A:\exists c\in  A^\circ\ \  s.t.  \ \ cx^q\in \fa^{[q]}\quad\forall q> 0\}.$
 A ring in which all ideals are tightly closed is called weakly $F$-regular. A ring is called  $F$-regular, if all of its localizations are  $F$-regular.
$A$ is called strongly $F$-regular if for every $c\in A^\circ$ the  homomorphism $A \to A^{1/q}$ sending
$1 \mapsto c^{1/q}$ splits.

 Recall that by $f_{gHK}^{A/\frak a}(n)$ we mean the length of $H^0_{\fm}\left(A/\fa ^{[q]}\right)$ as an $R$-module.
Now, if $\frak a$ is of finite colength, then we are in the situation of the classical Hilbert-Kunz multiplicity
and we use $f_{HK}^{A/\frak a}$  (resp. $e_{HK}(A/\frak a)$) instead of $f_{gHK}^{A/\frak a}$  (resp. $e_{gHK}(A/\frak a)$).
The limit $e_{HK}(A/\frak a):=\lim_{n\to\infty}\frac{f_{HK}^{A/\frak a}(n)}{p^{n\dim A}}$ exists by Monsky \cite{mon}.

\section{ The (LC) property}
Let $\{I_n\}$ be a family of ideals. For a fixed $I$, the following are the main examples:
i) $I_n:=I^n$;
ii) $I_n:=I^{[q]}$; and
iii) $I_n:=(I^{[q]})^\ast$. The next item appears in the base changing:
 Suppose $S$ is any ring with a family of ideals $\{J_n\}$ and suppose there is  a map $S\to R$.
Then $I_n:=J_nR$  is a family of ideals.

We say  (LC) holds for $\{I_n\}$, if there is  $b\in \mathbb{N}_0$, does not depending  to $q$ such that $\fm^{bq}H^0_{\fm}(R/I_n)=0$  for all $q. $  This imposes strong condition on the family:
\begin{example}Set $I_n:=\fm^{q^q}$.
Then  $\fm^{f(q)}H^0_{\fm}(R/I_n)\neq0$ for any polynomial $f$.\end{example}
The following example of Kollar \cite[Example 1.4]{koll} has a  role in  the effective nullstellensatz.

\begin{example}
Let $A:=\mathbb{C}[x, y, z, s]$ and let $\fa_n:=(x^2 - y^{2n+1}, z^2, xz, y^{n} z, s)$. The primary decomposition of $\fa_n$ is
$\fa_n=(x^2 - y^{2n+1}, z, s)\cap(x^2 - y^{2n+1}, z, s,x,y^{n}).$ Set $\fm:=(x,y,z,s)$. The $\fm$-primary component of $\fa_n$ is $(x^2 - y^{2n+1}, z, s,x,y^{n})$. By, $H^{0}_{\fm}(A/\fa_n)\simeq\frac{\fa_n^{sat}}{\fa_n}\simeq\frac{(x^2 - y^{2n+1}, z, s)}{\fa_n}.$ So, $\fm^{n}H^0_{\fm}(A/\fa_n)=0.$
\end{example}

In this paper we are mainly interested in the Frobenius power of an ideal and by the   $(LC)$  we mean the   $(LC)$  with respect to the Frobenius powers. We need the following trick of Hochster and Huneke:
\begin{fact}  \label{localiz}
Let $A$ be a noetherian  ring, with a maximal ideal $\fm$ and $\fa$ is an ideal of dimension one satisfies in $\fm^{bq}H^0_{\fm}(A/\fa^{[q]})=0.$
Then $\fa^{\ast}$ commutes with the localization with respect to  $\{a^n\}$ where $a\in A$. Indeed,
take $x\in(\fa A_a)^\ast$. After replacing $x$ by $f^{q_{0}}x$ we assume that $x\in A$. By the same trick and by  Definition 2.1,  there is $c\in A^\circ$ such that  $cx^q \in \fa ^{[q]} A_a$ $\forall q$. There is $f(q)$ depending on $q$
such that $ca^{f ( q)} x^q\in \fa^{[q]}$. This says  $$cx^q\in H^0_a(A/\fa^{[q]}) =H^0_{a+\fa^{[q]}}(A/\fa^{[q]})= H^0_{\rad(a+\fa^{[q]})}(A/\fa^{[q]})=H^0_{\fm}(A/\fa^{[q]}).$$
By the (LC) condition there is  $b$ such that  $c(a^{b}x)^q\in \fa^{[q]}$. Thus, $a^bx\in \fa^\ast$ and so $x \in \fa^\ast A_a$. The other side inclusion is trivial.
\end{fact}

\begin{remark}\label{reg}
 Having
the (LC) condition for ideals of dimension one, then being weakly $F$-regular implies $F$-regular, see the discussion
after \cite[Corollary 3.2]{Hun}. We recall from \cite{LS} that weak $F$-regularity and
strong $F$-regularity agree for standard $\mathbb{N}$-graded  rings which are of $F$-finite type over a field.
\end{remark}

Here is a comment on a graded algebra with \textsl{a base ring which is not a field.}

\begin{remark}\label{gr}
Let $R:=\frac{\mathbb{F}_2 [x,y,z,t]}{(z^4 + xyz^2 + z(x^3 + y^3)+(t+t^2)x^2y^2)}$ and $I:=(x^4,y^4,z^4)$.
The following holds.

i)  (See \cite[Main result]{t}) Let  $S:= \mathbb{F}_2[t]\setminus \{0\}$. Then
 $(S^{-1}I)^\ast\neq S^{-1}(I^\ast)$.

ii) We note that $R$ is a $3$-dimensional Cohen-Macaulay integral domain and $\dim R/I=\dim R-1$.  Under the assignments $\deg(t)=0$ and $\deg(x)= \deg(y)=\deg(z)=1$, $R$ is graded with the irrelevant ideal $\fp:=(x,y,z)$.
However,  $R_0:=\mathbb{F}_2[t]$ is not a field.

iii)
There is $b\leq660$ such that $\fp^{bq} H^0_{\fp}(R/I^{[p^n]}) =0.$ First note that $H^0_{\fp}(R/I^{[p^n]})=R/I^{[p^n]}$ and $\fp^{10}\subset I$. Indeed, suppose $x^iy^jz^k\in \fp$ is a monomial with $i+j+k=10$ and $i,j<4$. Then $k\geq 4$ and so $x^iy^jz^k\in I$. The set $$\{x^iy^jz^k:i+j+k=10\}=\{x^{10}\}\cup\{x^9y,x^9z\}\cup\{x^8y^2,x^8z^2,x^8yz\}\cup\ldots\cup\{y^{10},y^{9}z,\ldots,z^{10}\}$$ generates $\fp^{10}$. Its cardinality is $1+2+3+\ldots+11=66$. Set $a:=10\times66$. Then  $\fp^{aq}\subseteq (\fp^{10})^{[q] }\subseteq I^{[q]}$ and so $\fp^{aq} H^0_{\fp}(R/I^{[p^n]}) =0.$

iv) Let $0\neq a\in R$ be homogeneous of positive degree. Then $(IR_a)^\ast= (I^\ast)_a$.
 The claim follows by  the above item and the proof of  Fact \ref{localiz} (this may follows directly).

v) Computing $H^0_{f+\fp}(R/I^{[p^n]})$ has a significant importance, where $0\neq f\in \mathbb{F}_2[t]$.
We revisit this in the near future, as another perspective
of the $(LC)$-condition.

\end{remark}

The following plays an essential  role in this paper.

\begin{lemma}\label{linear}

	Let $\{I_n\}$ be a family of ideals of dimension one satisfying the (LC) property in the following explicit form: there exist integers $a\ge 1$ and $e_n\ge 1$ such that
\[
\fm^{a\,e_n}H^0_{\fm}(R/I_n)=0\qquad\text{for all }n.
\]
Suppose that
\[
\Gamma:=\bigcup_n\bigl(\Ass(R/I_n)\setminus\{\fm\}\bigr)
\]
is finite and that $x\in\fm^{a}$ avoids every prime in $\Gamma$. Then for every $n$,
\[
\len\bigl(H^0_{\fm}(R/I_n)\bigr)
=2\,\len\!\left(\frac{R}{I_n+(x^{e_n})}\right)
-\len\!\left(\frac{R}{I_n+(x^{2e_n})}\right).
\]
\end{lemma}

\begin{proof}
Fix $n$ and set $J:=I_n$, $e:=e_n$. Since $x\in\fm^{a}$ we have $x^{e}\in\fm^{ae}$, and the (LC) hypothesis gives $\fm^{ae}H^0_{\fm}(R/J)=0$. As $x$ avoids the nonmaximal associated primes of $R/J$, Fact~1.5(b) yields $(J:x^{e})=J^{\mathrm{sat}}$. Lemma~2.3 then gives
\[
\len\bigl(H^0_{\fm}(R/J)\bigr)
=\len\!\left(\frac{R}{J+(x^{e})}\right)
-\len\!\left(\frac{R}{(J:x^{e})+(x^{e})}\right).
\]
The multiplication-by-$x^{e}$ map
\[
\frac{R}{(J:x^{e})+(x^{e})}\xrightarrow{\;\cdot x^{e}\;}
\frac{R}{J+(x^{2e})}
\]
is injective with cokernel $R/(J+(x^{e}))$; hence
\[
\len\!\left(\frac{R}{(J:x^{e})+(x^{e})}\right)
=\len\!\left(\frac{R}{J+(x^{2e})}\right)
-\len\!\left(\frac{R}{J+(x^{e})}\right).
\]
Substitution gives the claim.
\end{proof}

\begin{remark}
	Let $\{I_n\}$ be a family of ideals  of dimension one  that satisfies in  the (LC) property, i.e.,  $\fm^{aq}H^0_{\fm}(R/I_n)=0$ for some $a$. Suppose that $\Gamma:=\{\fp:\fp\in \Ass(R/ I_n) \}\setminus\{\fm\}$ satisfies
	in the prime avoidances.  Let $x\in \fm^a\setminus \bigcup_{\fp\in \Gamma}\fp$.  One may ask $\ell(H^0_{\fm}(R/I_n))=2\ell(\frac{R}{I_n+(x^n)})-\ell(\frac{R}{I_n+(x^{2n})}).$
	This version of {Lemma} \ref{linear} is not valid. Indeed, let
$R=\mathbb F_3[s,t]$, $\m=(s,t)$, $q=3^n$, and
	\[
	I_n=(s^q,s^{q-1}t^q)\qquad(n\ge1).
	\]
	Then $\dim R/I_n=1$, the only nonmaximal associated prime is $(s)$, and
	$t\in\m\setminus(s)$.  Moreover,
	\[
	H^0_{\m}(R/I_n)\cong
	\frac{(s^{q-1})}{(s^q,s^{q-1}t^q)},
	\qquad
	\len H^0_{\m}(R/I_n)=q,
	\qquad
	\m^qH^0_{\m}(R/I_n)=0.
	\]
	For $n=1$ (so $q=3$), however,
	\[
	2\len\bigl(R/(I_1+(t))\bigr)
	-\len\bigl(R/(I_1+(t^2))\bigr)
	=2\cdot3-6=0\ne3=\len H^0_{\m}(R/I_1).
	\]
	
\end{remark}

\section{Proof of Theorem \ref{lq} }

\begin{lemma}\label{11}(See e.g. \cite[Proposition 2]{V2})
Suppose $\dim R\leq 1$ and $M$ is graded. Then  $\fm^{bq}H^0_{\fm}(F^n(M))=0$ for some $b$ that does not depending  to $q$.
\end{lemma}

\begin{lemma}\label{lc}
Suppose one of the following holds:
 \begin{enumerate}
\item[i)] The ring $R$ is normal, 2-dimensional and $I$ is an ideal,
\item[ii)] The ring $R$  is finitely generated, positively graded algebra over  a field $k$, and  $I$  a homogeneous ideal.
Suppose $\Ht(I)=\dim R-1$.
\end{enumerate}
Then $\fm^{bq}H^0_{\fm}(R/I^{[q]})=0 $ for some $b$ that does not depending  to $q$ where $\fm$ is the maximal ideal.
\end{lemma}

\begin{proof}
 The former is in \cite[Proposition 1]{V2} and the later is in \cite[Theorem 1]{V2} and \cite{Hun}.
\end{proof}
Let $R=\bigoplus_{n\geq0} R_n$ be a standard graded  ring, $\fm:=\bigoplus_{n>0} R_n$ the irrelevant ideal and $L=\bigoplus_{n\in \mathbb{Z}} L_n$  a  graded $R$-module. Set
$\engrad(L):=\Supp\{n:L_n\neq 0\}.$
Clearly,
$\engrad\left(\bigoplus_{i\in I} L(-\ell_i)\right)=\max\{\ell_i\}+\engrad(L)$ and  $\engrad(L_1\oplus L_2)=\max\{\engrad(L_1),\engrad(L_2)\}$.
Let $0\neq M$ be a finitely generated  and graded $R$-module. Then
$H_{\frak m}^{i}(M)$ is $\mathbb{Z}$-graded.  The well-known point is that
 $\engrad( H^i_{\fm}(M))<\infty.$
Let  $\ell\in \mathbb{Z}$. Recall that by $L_{\geq \ell}$ we mean $\bigoplus_{i\geq \ell} L_i$.

\begin{lemma}\label{mainlemma}
Let $R$ be a standard graded ring over a field of characteristic $p > 0$ and $M$ be
finitely generated and graded. If $\engrad\left(H^0_{\fm}(F^n(M))\right)\leq aq+c$
for some $a$  and $c$  not depending  to $q$, then    $M$ satisfies in the (LC) property.
\end{lemma}

\begin{proof}
Write $c=dq+e$ where $0\leq e\leq q-1$ and $d\in \mathbb{N}_0$. Set $f:=a+(d+1)$. Then $$\engrad\left(H^0_{\fm}(F^n(M))\right)< fq \  \ (\star)$$
for some $f$ independent of $q$.  We have  $\fm \left(H^0_{\fm}(F^n(M))\right)_i\subseteq H^0_{\fm}(F^n(M))_{\geq i+1}.$ As $M$ is finitely generated and $R$ is positively graded, there is $\ell_0$ such that $M_i=0$ for all $i<\ell_0$.
If $$\begin{CD}
\bigoplus_i R (-\beta_{1i})@>(a_{ij})>> \bigoplus_i R (-\beta_{0i}) @>>> M  @>>> 0.
\\
\end{CD}$$ is the graded presenting  sequence of $M$, then
 $$\begin{CD}
\bigoplus_i R (-q\beta_{1i})@>(a_{ij}^q)>> \bigoplus_i R (-\beta_{0i}q) @>>> F^n(M)  @>>> 0
\\
\end{CD}
$$  is the graded presenting sequence of $F^n(M)$, because the tensor product is right exact and $F^n(a_{ij})=(a_{ij}^q)$. This yields that   $H^0_{\fm}(F^n(M))\subseteq F^n(M)$ concentrated
 in degrees greater than $\ell_0 q$.  Set $b:=f+\mid\ell_0\mid$. Therefore $$\fm^{bq}H^0_{\fm}(F^n(M))\subseteq H^0_{\fm}(F^n(M))_{\geq fq}\stackrel{(\star)}=0$$
for some $b$ that does not depending  to $q$.
\end{proof}

Set
$b_0(L) := \inf\{\ell :\bigoplus_i L_{i\leq \ell}R = L\}$.
We will utilize the following result.

\begin{lemma}\label{kus}(\cite[Lemma 3.2]{kus})
Let $R$ be a standard graded  ring over a field of characteristic $p > 0$ and $M$ a
finitely generated graded $R$-module. Let $$\begin{CD}
\cdots @>>> C_d @>>> \cdots  @>\pi>>   C_0 @>>>0  @.
\\
\end{CD}$$ be a graded complex  with $\coker(\pi)=M$.  Suppose the following assertions hold:
  \begin{enumerate}
 \item[(1)]  $\depth C_j\geq i+j+1$ for all $0\leq j\leq \dim R-i-1$, and
\item[(2)]  $\dim H_j(C_{\bullet})\leq j+i$ for all $j\geq 1$.
\end{enumerate}
Then $\engrad(H^i_{\fm}(M)) \leq b_0(C_{\dim R-i}) +\engrad(H^{\dim R}_{\fm}(R)).$
\end{lemma}

\begin{corollary}\label{pdim}
Let $R$ be  graded Cohen-Macaulay and $I$ a homogeneous ideal of finite projective dimension. Then $\fm^{bq}H^0_{\fm}(R/I^{[q]})=0$ for some $b$ that does not depending  to $q$.
 \end{corollary}

\begin{proof} Denote the minimal free resolution of $R/I$ by $F_{\bullet}$. In view of \cite[Theorem 1.13]{PS1}, $F^n(F_{\bullet})$ is the minimal free resolution of
$R/I^{[q]}$, i.e.,  Lemma \ref{kus}(2) satisfied. The Cohen-Macaulay property  implies Lemma \ref{kus}(1). So, we are in the situation
to apply Lemma \ref{kus}. That is  $\engrad\left(H^0_{\fm}(R/I^{[q]})\right)\leq aq+c$
for some $a$  and $c$  not depending  to $q$. Now Lemma \ref{mainlemma} presents the desired claim.
\end{proof}

\begin{remark}\label{qu}
Let $R$ be graded (not necessarily standard) with only quotient singularity. Then for any graded ideal $I$, one has
$\fm^{bq}H^0_{\fm}(R/I^{[p^n]})=0$ for some $b$ that does not depending  to $q$.
\end{remark}

Let us recall the concept of quotient singularity. This means that there is a finite extension $R\subset S$ of graded rings such that $R = S^G:=\{s\in S:g(s)=s \quad \forall g\in G\},$ where
$S$ is a standard graded regular ring and $G$ a finite group of order prime to the characteristic. The
trace map divided by the group order shows that $R$ is a direct summand of
$S$.
It may be worth to recover Remark \ref{qu} by a more general technic:

\begin{discussion}\label{ffrt}
(i) Recall from \cite{ffrt} that $R$ has finite  $F$-representation type
if there are
finitely generated $\mathbb{Q}$-graded
$R$-modules $M_1\ldots, M_s $ such that for any $q$, there exist nonnegative integers $m_{qi}$ and rational numbers $\ell_{i,j}^{m_{qi}}$ such that $F^n(R)\simeq \bigoplus_{1\leq i\leq s} \bigoplus_{1\leq j\leq m_{qi}}M_i(\ell_{i,j}^{m_{qi}}).$ This is well-known
that the sequence $\{\ell_{i,j}^{m_{qi}}\}$ is bounded from below.

(ii) The following  are  of finite  $F$-representation type: 1) $R_1:=\mathbb{F}_p[X_1,\ldots, X_n]$,  2) $R_2:=R_1/I$ where $I$ is  a monomial ideal,
3) direct-summand of $R_2$.

(iii) Let $R$ be a standard graded  ring and of finite $F$-representation type over an $F$-finite field  $R_0$ and let  $I\lhd R$ be homogeneous.
Then $\fm^{bq}H^0_{\fm}(R/I^{[q]})=0$ for some $b$ that does not depending  to $q$.
Indeed, we do with the similar lines throughout \cite[Theorem 7.6]{kat1}:
$$H^0_{\fm}(F^n(R/I))=H^0_{\fm}(F^n(R)\otimes R/I)= \bigoplus H^0_{\fm}(M_j/IM_j)(\ell_{i,j}^{m_{qi}})$$
Having $c:=\inf_{i,j,m_{qi}}\{\ell_{i,j}^{m_{qi}}\}$ from part (i), $$\engrad
\left(H^0_{\fm}(F^n(R/I))\right)\leq\max_{1\leq j\leq s}\{\engrad (H^0_{\fm}(\frac{M_j}{IM_j}))\}-c<\infty.$$
 Set $b:=\max_{1\leq j\leq s}\{\engrad (H^0_{\fm}(\frac{M_j}{IM_j}))\}-c<\infty$.
In view of Remark \ref{g},  $H^0_{\fm}(R/I^{[q]})_{\geq bq}=0$ and deduce from
 Lemma \ref{mainlemma} that $\fm^{bq} H^0_{\fm}(R/I^{[q]})=0$.
\end{discussion}

\begin{notation}
By $\mathcal{R}(n)$ and $\mathcal{S}(n)$ we denote the Serre's conditions. The property $\mathcal{R}(n)$ defined as follows: If $\fp \in \Spec(R)$ and $\Ht(p)\leq n$,  then $R_{\fp}$ is regular, and $\mathcal{S}(n)$ defines by $\depth(R_{\fp})\geq \min\{n,\Ht(\fp)\}$.
\end{notation}

The following is a higher dimensional version of \cite[Lemma 3.7]{AB}.

\begin{proposition}\label{lemab}
Let $R$ be a  graded domain of dimension $d>1$ satisfies in $\mathcal{S}(2)$, $J\lhd R$  and $f\in R$ be  homogeneous. Then
$f_{gHK}^{R/J}=f_{gHK}^{R/fJ}$. In particular, $e_{gHK}(R/J)=e_{gHK}(R/fJ)$.
\end{proposition}

\begin{proof}As $R$ is $\mathcal{S}(2)$, one has $\depth(R_{\fm})\geq 2$ and by
\cite[Proposition 1.5.15(e)]{BH},
$\grade(\fm, R)\geq 2.$
Look at $0\to J^{[p^n]} \to R \to R/J^{[p^n]}\to 0$. It induces
$$0\simeq H^0_{\fm}( R)\to H^0_{\fm}( R/J^{[p^n]})\to H^1_{\fm}(   J^{[p^n]}) \to H^1_{\fm}( R)\simeq0.$$So, $H^0_{\fm}( R/J^{[p^n]})\simeq H^1_{\fm}(   J^{[p^n]}) \quad(\ast)$. Similarly,
$H^0_{\fm}\left(\frac{ R}{(fJ)^{[p^n]}}\right)\simeq H^1_{\fm}\left((fJ)^{[p^n]}\right)$. As $R$ is an integral domain,  multiplication by $f^{p^n}$  gives
$J^{[p^n]}\simeq(fJ)^{[p^n]}.$ Therefore, $$H^0_{\fm}( R/J^{[p^n]})\simeq H^1_{\fm}\left(J^{[p^n]}\right)\simeq H^1_{\fm}((fJ)^{[p^n]})\simeq H^0_{\fm}( R/(fJ)^{[p^n]}).$$
This is the desired claim.
\end{proof}

\begin{lemma}\label{s2}
Let $R$ be a graded $\mathcal{S}(2)$  ring over a field of characteristic $p > 0$ and $I\lhd R$  be graded.  Then $H^0_{\fm}(R/I^{[q]})\simeq H^1_{\fm}(I^{[q]})$ as graded modules.
\end{lemma}

 \begin{proof}
This is in  the Proposition \ref{lemab}$(\ast)$.
\end{proof}

 Serre's criterion of normality    over Noetherian rings says that
 a  ring is normal if and only if
it has the properties $\mathcal{R}(1)$ and $\mathcal{S}(2)$.

\begin{lemma}\label{dim}
Let $R$ be a  normal  standard graded  over a field of characteristic $p > 0$ and $I\lhd R$ be
 graded. Let $C_{\bullet}$ be deleted graded free resolution of $R/I$ and let $k>0$.
Then $\dim\left(H_k(F^n(C_{\bullet}))\right)\leq \dim R-2$.
\end{lemma}

 \begin{proof}
 In order to show this we prove $H_k(F^n(C_{\bullet}))_{\fp}=0$  $\forall\fp$
with $\dim R/ \fp\geq \dim R-1$. As the ring is catenary, $$\Ht(\fp)\leq 1\Longleftrightarrow\dim R/ \frak p \geq \dim R-1.$$
By  Serre's criterion,  for any prime ideal $\fp$ with $\Ht(\fp)\leq 1$, $R_{\fp}$ is regular.
By Fact \ref{kunz}, the Frobenius map is flat over $R_{\fp}$. Also, recall that  Frobenius map commutes with the localization, as Fact \ref{locali} dedicates. So, $H_k(F^n(C_{\bullet}))_{\fp}=0$ for all $\fp$
with $\dim R/ \fp\geq \dim R-1$.
\end{proof}

\begin{lemma}\label{short}
If $L\to M\to K$ be exact sequence of graded module (not necessarily finitely generated) such that
$\engrad(L)\leq O(p^n)$ and $\engrad(K)\leq O(p^n)$, then
$\engrad(M)\leq O(p^n)$. \end{lemma}

\begin{proof} This easily reduces to the situation of short exact sequences. In this case the proof follows by definition. \end{proof}

  For any $X\subset \Spec(R),$ recall that
$X^i:=\{\fp\in X\mid\dim R/ \fp\geq i\}.$

\begin{lemma}\label{tor}
	Suppose that for every homogeneous prime $\mathfrak p$ with
	$\dim R/\mathfrak p\ge i+1$ and $\mathfrak p\supseteq I$, the local ring
	$R_{\mathfrak p}$ is regular (equivalently,
	$\operatorname{Proj}(R/I)^{i+1}\subseteq \operatorname{Reg}(R)$). Then
	\[
	\dim \Tor^R_j(F^n(R),R/I)\le i
	\]
	for all $j\ge 1$.
\end{lemma}

\begin{proof}
	Take a graded free resolution of $R/I$:
	\[
	\cdots \to \bigoplus_i R(-\beta_{di})
	\xrightarrow{\varphi_d}\cdots \to \bigoplus_i R(-\beta_{0i})\to 0 .
	\]
	Applying $F^n(-)$ gives a complex of graded modules whose $j$-th homology is
	$\Tor^R_j(F^n(R),R/I)$. Hence $\Tor^R_j(F^n(R),R/I)$ is graded, and the
	minimal primes of its support are homogeneous.
	
	Let $\mathfrak p$ be a minimal prime in
	$\Supp\bigl(\Tor^R_j(F^n(R),R/I)\bigr)$, and suppose for contradiction that
	$\dim R/\mathfrak p\ge i+1$. Since $\mathfrak p$ is in the support,
	$\mathfrak p\supseteq I$, so $\mathfrak p\in \operatorname{Proj}(R/I)^{i+1}$.
	By the strengthened hypothesis, $R_{\mathfrak p}$ is regular.
	
	By localization (Fact \ref{locali}),
	\[
	\Tor^R_j(F^n(R),R/I)_{\mathfrak p}
	\simeq \Tor^{R_{\mathfrak p}}_j\bigl(F^n(R)_{\mathfrak p},(R/I)_{\mathfrak p}\bigr)
	\simeq \Tor^{R_{\mathfrak p}}_j\bigl(F^n(R_{\mathfrak p}),R_{\mathfrak p}/IR_{\mathfrak p}\bigr).
	\]
	Since $R_{\mathfrak p}$ is regular, Kunz's theorem (Fact \ref{kunz}) implies
	$F^n(R_{\mathfrak p})$ is flat over $R_{\mathfrak p}$, so
	\[
	\Tor^{R_{\mathfrak p}}_j\bigl(F^n(R_{\mathfrak p}),R_{\mathfrak p}/IR_{\mathfrak p}\bigr)=0
	\qquad\text{for all } j>0.
	\]
	This contradicts the choice of $\mathfrak p$ (for $j\ge 1$). Therefore every
	minimal prime of the support has $\dim R/\mathfrak p\le i$, and hence
	\[
	\dim \Tor^R_j(F^n(R),R/I)\le i. \qedhere
	\]
\end{proof}

\begin{remark}
	
	Suppose $\Proj( \frac{R}{I})^{i+1}$ is regular. One may ask from {Lemma} \ref{tor} that
	$\dim (\Tor^R_{j}(F^n(R),\frac{R}{I}))\leq i$. This is not true. To see an example, Let $k=\mathbb F_3$ and consider the standard-graded ring and homogeneous ideal
$$
		R=k[x,y,z,u,v]/(xy-z^2), \qquad I=(x,y,z)R.
$$
	Then $R$ is a four-dimensional normal Cohen--Macaulay domain: it is a hypersurface,
	and its singular locus, $V(x,y,z)$, has codimension two.  Moreover
	\[
	R/I\cong k[u,v],\qquad \Proj(R/I)\cong\mathbb P_k^1,
	\]
	so the regularity hypothesis in Lemma~4.13 holds for $i=1$ (indeed, the projective
	quotient is regular everywhere).
	
	Put $P=(x,y,z)R$.  This is a homogeneous prime with $\dim R/P=2$.
	On the other hand, $R_P$ is a two-dimensional singular local ring: its maximal ideal
	has three minimal generators, while its Krull dimension is two.  Since $R_P$ is
	$F$-finite, Kunz's theorem and the local criterion for flatness of a finite module
	give
$$
		\Tor^{R_P}_1\bigl(F^1(R_P),\kappa(P)\bigr)\ne0,
		\qquad \kappa(P)=(R/I)_P.
$$
	Localizing the Frobenius--Tor module therefore yields
	\[
	\bigl(\Tor^R_1(F^1(R),R/I)\bigr)_P\ne0.
	\]
	As $\dim R/P=2$, it follows that
	\[
	\boxed{\dim\Tor^R_1(F^1(R),R/I)\ge2,}
	\]
	contradiction.
\end{remark}

\begin{lemma}\label{kunz1}
Frobenius power commutes with the localization.
\end{lemma}

\begin{proof}
This is straightforward and we leave it to the reader.
\end{proof}

\begin{lemma}\label{kunz2}
Let $R$ be  a regular ring with an ideal $I$. Then $F^n(I)\simeq I^{[q]}$ for any $n$.
\end{lemma}

\begin{proof}
By Remark \ref{g} $F^n(R/I)\simeq R/ I^{[p^n]}$. In view of Fact \ref{kunz},
$F^n(-)$ is exact. Keep in mind $F^n(R)=R$. We apply $F^n(-)$ to the short exact sequence
$0\to I \to R \to R/I \to 0$ to arrives to the exact sequence: $0\to F^n(I) \to R \to R/I^{[q]} \to 0$. So, $F^n(I)\simeq I^{[q]}$.
\end{proof}

\textbf{Proof of Theorem \ref{lq}.}
 Set $d:=\dim R$. As $R$ is Cohen-Macaulay  and standard graded over a field, we have
$H^i_{\fm}(R)=0$  for all $0\leq i\leq d-1.$
We will use this fact several times.
In view of Lemma \ref{lc} and  Lemma \ref{11}  we can assume that $d\geq 3$.
Look at the graded free resolution  of $R/I$: $$\begin{CD}
\cdots @>>> \bigoplus_i R (-\beta_{di}) @>\varphi_d>> \cdots  @>>>   \bigoplus_i R (-\beta_{0i})@>\varphi_0>>0 .@.
\\
\end{CD}
$$Applay $F^n(-)$, this induces the following complex $C_{\bullet}$: $$\begin{CD}
\cdots @>>> \bigoplus_i R (-\beta_{di}q) @>>> \cdots  @>>>   \bigoplus_i R (-\beta_{(d-3)i}q)@>F^n(\varphi_{d-3})>>0  .@.
\\
\end{CD}$$  We apply Lemma \ref{kus} for $i=d-3$ and for the complex $C_{\bullet}$. As $R$ is Cohen-Macaulay, the condition Lemma \ref{kus}(1) satisfied. In the light of Lemma \ref{dim} $\dim H_k(C_{\bullet})\leq d-2,$
i.e., the property
Lemma \ref{kus}(2) checked.
For the simplicity, set  $R^{\beta _i}:=\bigoplus_jR (-\beta_{ij})$. Note that $C_{d-(d-3)}=R^{\beta _{(2d-3)-(d-3)}}=R^{\beta_d}$.  By Lemma \ref{kus}, $$\engrad\left(H^{d-3}_{\fm}(H_0(C_{\bullet}))\right) \leq q\max\{\beta_{di}\} +\engrad(H^{d}_{\fm}(R))\leq aq+c,   \ \ (\natural)$$
for some $a$ and $c$ not depending to $q$.

(i) Without loss of generality we may assume that $\dim R=3$.   Note that  $H_0(C_{\bullet})=R/I^{[p^n]}$. In view of $(\natural)$, we have $\engrad\left(H^{0}_{\fm}(R/I^{[p^n]})\right)\leq aq+c.$  Lemma \ref{mainlemma} completes the argument.

(ii) In view of i), without loss of generality we may assume that $\dim R\geq 4$. First we deal with the case $\dim R=4$. Note that
$$R^{\beta_2}\stackrel{\varphi_2} \longrightarrow R^{\beta_1}\longrightarrow I\longrightarrow 0,$$is a presenting sequence of $I$.
As the  tensor product is right exact, $$R^{\beta_2}\stackrel{F^n(\varphi_2)} \longrightarrow R^{\beta_1}\stackrel{\rho}\longrightarrow F^n(I)\longrightarrow 0$$is exact. Therefore, $H_0(C_{\bullet}) =F^n(I)$. In a similar vein, we have the exact sequence
$$F^n(I)\stackrel{\phi_2}\longrightarrow R\stackrel{\phi_1}\longrightarrow F^n(R/I)\longrightarrow 0$$
Note that $\ker(\phi_1)=I^{[p^n]}$.
There is a natural surjection $$\pi:F^n(I)\longrightarrow \im(\phi_2)=\ker(\phi_1)\longrightarrow 0$$Set $L:=\ker(\pi)$.
Let $\fp$ be a nonsingular prime ideal.
Due to Fact \ref{locali} and Lemma \ref{kunz1} we have the exact sequence $$0\longrightarrow L_{\fp}\longrightarrow F^n(I_{\fp})\longrightarrow (I_{\fp})^{[p^n]}\longrightarrow 0.$$ In view of Lemma \ref{kunz2}, we get $L_{\fp}=0$.
According  to the regularity of $\Proj( \frac{R}{I})^2$, we observe that $\dim(L)<2$. Keep the Grothendieck's vanishing theorem in mind. The exact sequence $$0\longrightarrow L\longrightarrow F^n(I)\longrightarrow I^{[p^n]}\longrightarrow 0$$
induces the exact sequence $$H^{1}_{\fm}(F^n(I))\longrightarrow H^{1}_{\fm}(I^{[p^n]})\longrightarrow H^{2}_{\fm}(L)=0\quad(\ast)$$
Incorporate this with $(\natural)$
and
Lemma \ref{short} to conclude  that $\engrad(H^1_{\fm}(I^{[p^n]}))\leq aq+c.$
Then, in view of Lemma \ref{s2}, $H^0_{\fm}(R/I^{[p^n]})\simeq H^1_{\fm}(I^{[p^n]})$.
Thus, $\engrad(H^0_{\fm}(R/I^{[p^n]}))\leq aq+c,$
for some $a$ and $c$ not depending to $q$. It remains to apply Lemma \ref{mainlemma} to conclude the claim.
Thus we can assume that $d>4$.
By looking at the exact sequence  $$\begin{CD}
0 @>>> \im F^n(\varphi_{d-4}) @>>> \bigoplus_i R (-\beta_{(d-3)i}q)  @>>> H_0(C_{\bullet})  @>>>0  @.,
\\
\end{CD}$$we get that $$H^{d-3}_{\fm}(H_0(C_{\bullet}))\simeq H^{d-2}_{\fm}(\im F^n(\varphi_{d-4})).$$
By the following graded exact sequence $$\begin{CD}
 0 @>>> \im(F^n(\varphi_{d-4}))  @>>>   R^{\beta_{d-3}} @>>>  R^{\beta_{d-3}}/ \im(F^n(\varphi_{d-4})) @>>> 0@.
\\
\end{CD}
$$we observe $$H^{d-3}_{\fm}\left(\frac{R^{\beta_{d-3}}}{ \im(F^n(\varphi_{d-4}))}\right)\simeq H^{d-2}_{\fm}(\im(F^n(\varphi_{d-4}))).$$
Keep in mind $$ \ker(F^n(\varphi_{d-3}))/ \im(F^n(\varphi_{d-4})) \simeq  \Tor^R_{d-3}(F^n(R), R/I).$$The following exact sequence of graded modules with homogeneous morphisms $$
 0 \to \frac{\ker(F^n(\varphi_{d-3}))}{ \im(F^n(\varphi_{d-4}))} \to   \frac{R^{\beta_{d-3}}}{\im(F^n(\varphi_{d-4}))}\to \frac{R^{\beta_{d-3}}}{\ker(F^n(\varphi_{d-3}))}\to 0
$$
implies the following exact sequence of graded modules:
$$H^{d-3}_{\fm}\left(\frac{R^{\beta_{d-3}}}{\im(F^n(\varphi_{d-4}))}\right)
\to H^{d-3}_{\fm}\left(\frac{R^{\beta_{d-3}}}{\ker(F^n(\varphi_{d-3}))}\right) \to H^{d-2}_{\fm}\left(\Tor^R_{d-3}(F^n(R),\frac{R}{I})\right).$$
As $d>4$, we have $d-2> \dim (\Tor^R_{d-3}(F^n(R),\frac{R}{I})).$ By Grothendieck's vanishing theorem,$$H^{d-2}_{\fm}(\Tor^R_{d-3}(F^n(R),\frac{R}{I}))=0.$$In particular,
$\engrad\left(H^{d-2}_{\fm}(\Tor^R_{d-3}(F^n(R),\frac{R}{I}))\right)\leq O(p^n).$ So, in view of Lemma \ref{short}, $$\engrad\left(H^{d-3}_{\fm}(\frac{R^{\beta_{d-3}}}{\ker(F^n(\varphi_{d-3}))})\right)\leq O(p^n)  \quad (\dagger).$$

  \begin{enumerate}
 \item[] An easy case:  $\dim R=5$.  Recall from part (i) that there is an exact sequence
$$R^{\beta_2}\stackrel{F^n(\varphi_2)} \longrightarrow R^{\beta_1}\stackrel{\rho}\longrightarrow F^n(I)\longrightarrow 0.$$
Thus, $$\frac{R^{\beta_2}}{\ker(F^n(\varphi_2))}=\im (F^n(\varphi_2))=\ker(\rho)\quad(\star)$$ and $$0\longrightarrow \ker(\rho)\longrightarrow R^{\beta_1}\longrightarrow F^n(I)\longrightarrow 0.$$This induces the exact sequence $$
0= H^1_{\fm}(R^{\beta_1})\to H^1_{\fm}(F^n(I))  \to H^2_{\fm}(\ker(\rho)) \to H^2_{\fm}(R^{\beta_1})=0.
$$
 Therefore,
\[\begin{array}{ll}
\engrad(H^0_{\fm}(R/I^{[p^n]})) &\stackrel{\ref{s2}}= \engrad(H^1_{\fm}(I^{[p^n]}))\\
&\stackrel{(\ast)}\leq\engrad\left(H^1_{\fm}(F^n(I))) \right)\\
&=\engrad\left(H^2_{\fm}(\ker(\rho)) \right)\\
&\stackrel{(\star)}=\engrad(H^2_{\fm}(\frac{R^{\beta_2}}{\ker(F^n(\varphi_2))}))\\
&\stackrel{(\dagger)}\leq aq+c\quad(\ddagger)
\end{array}\]
So, Lemma \ref{mainlemma} completes the argument in this case.
\end{enumerate}
Now assume $d>5$.
Our next task is to descent $(\dagger)$. This is a repetition of the above argument. We do this for the convenience of the reader.
Keep in mind that $$\frac{R^{\beta_{d-3}}}{\ker(F^n(\varphi_{d-3}))}\simeq \im (F^n(\varphi_{d-3})) \quad(\ast,\ast)$$
Look at the exact sequence $$0\to \im (F^n(\varphi_{d-3}))\to R^{\beta_{d-4}}\to \frac{R^{\beta_{d-4}}}{\im(F^n(\varphi_{d-3}))}\to 0.$$
This gives $$0\to H^{d-4}_{\fm}\left(\frac{R^{\beta_{d-4}}}{\im(F^n(\varphi_{d-3}))}\right)\to H^{d-3}_{\fm}\left(\im (F^n(\varphi_{d-3}))\right) \to 0,$$and by the help of $(\ast,\ast)$ and  $(\dagger)$ we observe that $$\engrad\left(H^{d-4}_{\fm}(\frac{R^{\beta_{d-4}}}{\im(F^n(\varphi_{d-3}))})\right)\leq O(p^n).$$
Again look at $$
 0 \to \frac{\ker(F^n(\varphi_{d-4}))}{ \im(F^n(\varphi_{d-5})) }\to   \frac{R^{\beta_{d-4}}}{\im(F^n(\varphi_{d-5}))}\to \frac{ R^{\beta_{d-4}}}{\ker(F^n(\varphi_{d-4}))}\to 0.
$$As $d>5$, we have $d-3> \dim (\Tor^R_{d-4}(F^n(R),\frac{R}{I})).$ By Grothendieck's vanishing theorem,
$$\engrad\left(H^{d-3}_{\fm}(\Tor^R_{d-4}(F^n(R),\frac{R}{I}))\right)\leq O(p^n).$$ Hence, $\engrad\left(H^{d-4}_{\fm}(\frac{R^{\beta_{d-4}}}{\ker(F^n(\varphi_{d-4}))})\right)\leq O(p^n)$ i.e., $(\dagger)$ descended.
Doing inductively, we get$$\engrad\left(H^2_{\fm}(\frac{R^{\beta_2}}{ \ker(F^n(\varphi_{2}))})\right)\leq O(p^n).$$By the same vein as $(\ddag)$, we have $\engrad(H^0_{\fm}(R/I^{[p^n]}))\leq aq+c$ and again Lemma \ref{mainlemma} yields the claim.

(iii) This is in Corollary \ref{pdim}. $\hspace{109 mm}\square$

$\vspace{1 mm}$

One may extend the above proof by computing $\Tor_i^R(F^n(R),R/I)$.  Here we give an example of a graded ring $R$ with a graded ideal $I$ such that $\dim\left( \Tor_i^R(F^n(R),R/I)\right)=\dim R/I=\dim R$  for all $i>0.$

\begin{example}
Let $R:=\mathbb{F}_p[X_1,X_2,\ldots,X_{\ell+1}]/(X_1X_2)$. We use lowercase letters here to elements in $R$.
 Set $I:=(x_1)$.
Look at the graded free resolution  of $R/I$: $$\begin{CD}
F_{\bullet}:\cdots @>>>R(-2) @>x_2>>R(-1) @>x_1>>    R @>>>0 .@.
\\
\end{CD}$$ As  $F_{\bullet}$ is minimal, $\pd(I)=\infty$.
Then $F^n(F_{\bullet})$ is of the form
$$\begin{CD}
\cdots @>>>R(-2p^n) @>x_2^{p^n}>>R(-p^n) @>x_1^{p^n}>>    R @>>>0 .@.
\\
\end{CD}$$
Thus, (up to a shifting  by $p^n-1$)\begin{equation*}
\Tor_i^R(F^n(R),M)= \left\{
\begin{array}{rl}
\frac{(0:_Rx_1^{p^n})}{(x_2^{p^n})}\simeq \frac{(x_2)}{(x_2^{p^n})}& \qquad\text{ \ \ if } i\in 2\mathbb{N}_0+1\\
\frac{(0:_Rx_2^{p^n})}{(x_1^{p^n})}\simeq \frac{(x_1)}{(x_1^{p^n})}&\qquad\text{  \ \ if } i\in 2\mathbb{N}
\end{array} \right.
\end{equation*}
and so \begin{equation*}
\Ann_R\left(\Tor_i^R(F^n(R),R/I)\right)= \left\{
\begin{array}{rl}
(x_2^{p^n-1})& \qquad\text{ \ \ if } i\in 2\mathbb{N}_0+1\\
(x_1^{p^n-1})&\qquad\text{  \ \ if } i\in 2\mathbb{N}.
\end{array} \right.
\end{equation*}Therefore, we can not control  $\dim\left(\Tor_i^R(F^n(R),R/I)\right)$
by some things independent of $\dim (R/I)$ for all $i>0$ and all $\ell$.
\end{example}

\section{An application to Hilbert-Kunz multiplicity }

\begin{discussion}\label{vra} In view of \cite[Page 3]{V1} we  borrow the following quotation:
"If one could check that (LC) holds for the Frobenius
powers of the ideal $I$ obtained in the first step of Brenner's construction,
our result could then be used to obtain a more direct and explicit route
to $\fm$-primary ideals with irrational Hilbert-Kunz multiplicities."
\end{discussion}

The ring in Brenner's construction is  $R:=
K[X, Y, Z, W]/(F)$ where $F$ is homogeneous of degree four and $K$ has positive
characteristic $p \gg 0$. The ideal in Brenner's construction is $I:=(a,b)$ for some homogenous elements $  a$ and $b$ in $R$.

\begin{corollary} \label{vra}
Adopt the above notation. Then the (LC) holds for the proposed ring $R$.
\end{corollary}

\begin{proof}
Note that $R$ is normal, Cohen-Macaulay, $\dim R=3$ and $I$ is homogeneous.
So the desired property follows from Theorem \ref{lq}.
\end{proof}

\section{Applications to $e_{gHK}(-)$ of 2-dimensional rings}

Among other things we give  a proof of Corollary \ref{a2} and Corollary \ref{node}.

\begin{fact}
Let $R$ be a ring of prime characteristic $p$, $I\lhd R$ an ideal and $M$ a finitely generated module.
The following holds.\begin{enumerate}
\item[i)]
$\min(I)=\min(I^{[q]})$ for any $q:=p^n$.
\item[ii)]  $\Supp(M)=\Supp(F^n(M))$ for all $n$.
\end{enumerate}
\end{fact}

\begin{lemma}\label{avoid} Let $R$ be a standard graded ring over a field of prime characteristic $p$ and
let $M$ be a $1$-dimensional finitely generated graded $R$-module.
Then $\cup_n \Ass(F^n(M))$ is finite.
\end{lemma}

\begin{proof}
Denote the unique maximal graded ideal of $R$ by  $\fm$ and take $$\fp \in \bigcup_n \Ass(F^n(M))\setminus\{\fm\}.$$ Since $\dim (M)=1$ and in view of the above  fact, we deduce  $$\fp\in \bigcup_n\min\{\Ass_R(F^n(M))\}=\bigcup_n\min\{\Supp(F^n(M))\}=\min\{\Supp(M )\}.$$
Thus, $$\bigcup_n \Ass (F^n(M))\subset \min\{\Supp(M )\}\cup \{\fm\}$$ which is a finite set.
\end{proof}

\begin{remark} i) Suppose $R$ is a semilocal ring and $M$ is a finitely generated module.
If $\dim (M)=1$, then $\bigcup \Ass(F^n(M))$ is finite. Indeed, apply the above argument.

ii) It may be $\mid\cup\Ass(F^n(M))\mid=\infty$, see \cite{kat1}. However, countable prime
avoidance holds for rings that contain an uncountable field and holds in any complete
local ring, see \cite{bru}.
\end{remark}

\begin{lemma} \label{holger}(See \cite{Holger} and \cite{Holger2}) Let $R$ be a two-dimensional normal standard graded $K$-domain over an
algebraically closed field $K$ of prime characteristic $p$ and  $I$ a homogeneous ideal of dimension zero. Then
$f_{HK}^{R/I}( q) = e_{HK}(R/I)q^2 + \gamma(q),$
where $e_{HK}(R/I)$ is a rational number and $\gamma(q)$ is a bounded function.
Moreover if $K$ is the algebraic closure of a finite field, then $\gamma(q)$ is an eventually
periodic function.
\end{lemma}

 \begin{discussion}\label{dislin}
 Here is  a comment on the  construction of $s$ in Lemma \ref{linear}(ii).
Set $\Ass(F^n(R/I))^\circ:= \Ass(F^n(R/I))\setminus\{\fm\}$. Due to \cite[Proposition 2.4]{V1}, $s\in \fm^{b(I)} \setminus  \bigcup _n \Ass(F^n(R/I))^\circ.$ One
can pick $s$ to be  homogeneous, if $I$ and $R$ are homogeneous.
\end{discussion}

\emph{Proof of Corollary \ref{a2}.}
If $\dim R/I=0$, then $I$ is primary to the maximal ideal. In this case
the generalized Hilbert-Kunz theory is the classical Hilbert-Kunz theory.  The claim in this case is the subject of Lemma \ref{holger}. If $\dim R/I=2$,
then $I=0$ and $H^0_{\fm}(R/I^{[q]})=0$, so  there are nothing to prove. Then, without loss
of the generality we can assume that $\dim R/I=1$.  Due to Lemma \ref{avoid}, $\bigcup_n \Ass(R/I^{[q]})$ is finite.
In the light of  Lemma \ref{lc}(1), the (LC) condition holds.
By Discussion \ref{dislin}, there is a homogeneous $s$ such that $$f^{R/I}_{gHK} = 2f^{R/I+(s)}_{HK }- f^{R/I+(s^2)}_{HK} \quad(\dag)$$
Thanks to  Lemma \ref{holger},\begin{enumerate}
\item[1)]
$f_{gHK}^{R/I+(s)}(n) = e'q^2 + \gamma'(q)$,
where $e'\in \mathbb{Q}$ and $\gamma'(q)$ is  bounded,
\item[2)]$f_{gHK}^{R/I+(s^2)}(n) = e''q^2 + \gamma''(q)$,
where $e''\in \mathbb{Q}$ and $\gamma''(q)$ is bounded,
\item[3)]if $K=\overline{\mathbb{F}}_p$, then $\gamma'(q)$  and $\gamma''(q)$ are eventually
periodic.\end{enumerate} Combining $1)$ and $2)$  along with $(\dag)$ yields that
\[\begin{array}{ll}
f^{R/I}_{gHK}&=2(e'q^2 + \gamma'(q))-(e''q^2 + \gamma''(q))\\&= (2e'-e'')q^2+(2\gamma'(q)-\gamma''(q))\\&:= e_{gHK}(R/I)q^2 + \gamma(q).
\end{array}\] So,
$i)$ follows. The item $3)$ presents the proof of $ii)$. $\hspace{80 mm}\square$

\begin{lemma}\label{mon} (\cite[Theorem 3.7]{mon1})
Let $R$ be the coordinate ring of  a nodal  plane cubic projective curve over an algebraically closed field of prime characteristic $p$.
Let $J$ be a homogeneous primary to the irrelevant ideal. Then $f_{gHK}^{R/J}(n)=
\mu q^2 +aq- r$, $r$ only depends on $q \mod 3$. There is an explicit formula for $r$.
\end{lemma}

\emph{Proof of Corollary \ref{node}.}
Without loss of the generality we assume  $\Ht(I)=1$.
By Lemma \ref{lc}(2), $R$ satisfies in the (LC) condition. Lemma \ref{avoid} says $\bigcup_n \Ass(R/I^{[q]})$ is finite. In view of Lemma \ref{linear}, $$f^{R/I}_{gHK} = 2f^{R/I+(s)}_{HK }- f^{R/I+(s^2)}_{HK} \ \ (\dag)$$
for a homogeneous element $s$. In the light of Lemma \ref{mon}, \begin{enumerate}
\item[i)]
$f_{gHK}^{R/I+(s)}( q) = \mu' q^2 +a'q- r'$, where $r$  is an integer that depends on  $q \mod 3$.
\item[ii)]$f_{gHK}^{R/I+(s^2)}( q) = \overline{\mu}q^2 +\overline{a}q- \overline{r}$, where $\overline{r}$  is an integer that depends on $q \mod 3$.
\end{enumerate} Combining (1) and (2) throughout $(\dag)$ we get the claim.
$\hspace{72 mm}\square$

\begin{theorem}\label{irr}
Let $I$ be a two generated graded ideal of the coordinate ring of an irreducible plane  projective curve over an algebraically closed field of prime characteristic $p$.
Then $e_{gHK}(R/I)$ is rational.
\end{theorem}

\begin{proof}
Keep \cite[Corollary 3.7]{Holger2} and  Lemma \ref{lc}(ii) in mind. Now the claim follows by the proof of Corollary \ref{a2}.
\end{proof}

In the above result our data on $f_{gHK}(R/I)$ becomes complete, if one can prove the primary version of the following result of Monsky.

\begin{lemma}(\cite[Theorem I and II]{irrm})\label{irrm}
Let k be an algebraically closed field of characteristic $p > 0$, and $f\in k[x,y,z]$ be a degree $d$
irreducible form defining a projective plane curve. Then  $$f_{HK}(n) = e_{HK}(R/ \fm)p^{2n} + R(n)$$ where $R(n) = \mathrm{O}(p^n)$.
Suppose in addition $k$ is finite. One has\begin{enumerate}
\item[i)] If $e_{HK}(R/ \fm)=\frac{3}{4}d$, then $R(n)$ is eventually periodic.
\item[ii)]If $e_{HK}(R/ \fm)\neq \frac{3}{4}d$, then the $\mathrm{O}(1)$ term for $R(n)$ is eventually periodic.\end{enumerate}
\end{lemma}

Toward extending the primary version of the mentioned result we restate the following question of Monsky asked in a workshop.
\begin{question}
Let $C$ be a reduced irreducible projective curve over a finite field of characteristic  $p$, and $W$ be a vector bundle on $C$. Let $W^{[q]}$ be the pull-back of $W$ by the n-th power of Frobenius. How does the element Poincare$(W^{[q]}) $ of $\mathbb{Z}[T,1/T]$ depend on $n$? \end{question}

\begin{remark}\label{length}
Let $R$ be as Corollary \ref{a2},  and $f,g\in R$ be a homogeneous parameter sequence. Then
$$\ell\left(R/(f^{n},g^{m})\right)=mn\ell(R/(f,g))$$for any $n$ and $m$.
\end{remark}

\begin{proof}
In the polynomial ring we have $\ell(\frac{K[X,Y]}{(X^n,Y^m)})=mn$.
Any normal and 2-dimensional ring is Cohen-Macaulay. Thus $f,g$ is a regular sequence.
In view of \cite{Har}, there is a flat ring homomorphism $K[X,Y]\to R$ defined by the assignments
$X\to f$ and $Y\to g$. By flatness, we have the following formula:
 \[\begin{array}{ll}
 \ell\left(R/(f^n,g^m)\right)&=\ell\left((\frac{K[X,Y]}{(X^n,Y^m)})\otimes R\right)
\\&=\ell \left(\frac{K[X,Y]}{(X^n,Y^m)}\right)\ell\left(\frac{R}{(X,Y)R}\right)
\\&=\ell\left(\frac{K[X,Y]}{(X^n,Y^m)}\right)\ell\left(\frac{R}{(f,g)}\right)\\
&=mn\ell(R/(f,g)).
\end{array}\]
\end{proof}

\begin{corollary}(Also, see \cite[Example 3.4]{AB})
Let $R$ be as Corollary \ref{a2} and $f\in R$ be homogeneous. Then
$f^{R/(f)}_{gHK}=0$. In particular, $e_{gHK}(R/(f))=0$.
\end{corollary}

\begin{proof}
Set $I=(f)$ and let $s$ be as Lemma \ref{linear}. So,

\[\begin{array}{ll}
f^{R/I}_{gHK}(n)&= 2f^{R/I+(s)}_{HK }(n)- f^{R/I+(s^2)}_{HK}(n)\\
&= 2\ell\left(R/(f^{p^n},s^{p^n})\right)-\ell\left(R/(f^{p^n},s^{2p^n})\right)\\
&=0.
\end{array}\]
\end{proof}

In a similar vein we observe:

\begin{corollary}
Let $A$ be a d-dimensional Cohen-Macaulay ring which contains a field and let $x_1,\ldots,x_d$ be a parameter sequence. Then
$\ell(A/(x_1^{n_1},\ldots, x_d^{n_d}))=n_1\ldots n_d\ell\left(A/(x_1,\ldots, x_d)\right).$
\end{corollary}

\section{Applications  to the stability}
In this section we give a proof of Corollary \ref{sy} and Example \ref{exa}.
We need the following discussion in this and in the next section.

\begin{discussion}\label{stabl}
Let $\mathcal{V}$ be a \textit{vector bundle} over a projective scheme $X$. Let $\{f_1,\dots, f_n\}$ be
homogeneous elements of  an $\NN$-graded ring $R$.  Set $d_i:=\deg f_i$. The sheaf of  relations $\mathcal{S}:= \Syz(f_1,\dots, f_n)$ on $\Proj(R)$ is given by the following exact sequence:
$$ 0 \longrightarrow \mathcal{S} \longrightarrow \bigoplus_{i=1}^n \mathcal{O}_X(-d_{i}) \stackrel{f_1,\dots, f_n}\longrightarrow
\mathcal{O}_X.  $$If X is integral, then $\mathcal{S}$ is a torsion-free sheaf, and $\mathcal{S}$ is  a vector bundle on $\Proj(R)$  provided   $\rad(f_1,\dots, f_n)=\fm$.
The slope of $V$ is defined by $\mu(\mathcal{V}):=\frac{\deg(\mathcal{V})}{ \rank(\mathcal{V})}$. For example, $\rank(\mathcal{S})=n-1$ and
if the zero locus $Z = V(f_i)_{i=1}^n$ has codimension $\geq2$, then $$\deg(\mathcal{S}(m)) = \left((n - 1)m -
\sum ^n
_{i=1} d_i)\right) \deg(X).$$

The sheaf $\mathcal{V}$ is said to be  \it{semistable}
if $ \mu(\mathcal{W})\leq  \mu(\mathcal{V}) $ for every subsheaf $\mathcal{W}\subset\mathcal{ V}$ (subbundle when we deal with projective curves). If strict inequality holds for all such $ \mathcal{W}$, we say
$\mathcal{V}$ is   \it{stable}.
 A deep result of Donaldson--Uhlenbeck--Yau corresponds stable vector bundles over a complex manifold to Einstein--Hermitian vector bundles.
 This is quite strong.
Any vector bundle $V$ has a \it{Harder-Narasimhan} filtration, i.e., a chain
 $$0= \mathcal{V}_0\subset \mathcal{V}_1  \subset   \ldots \subset \mathcal{V}_t= \mathcal{V}$$
such that $\frac{\mathcal{V}_i}{\mathcal{V}_{i-1}}$ is semistable   and $\mu(\frac{\mathcal{V}_i}{\mathcal{V}_{i-1}}) >\mu(\frac{\mathcal{\mathcal{V}}_{i+1}}{\mathcal{\mathcal{V}}_{i}})$.  We denote $\mu_{min}:=\mu(\frac{\mathcal{V}_{t}}{\mathcal{V}_{t-1}})$ and $\mu_{max}:=\mu(\mathcal{V}_{1})$.
Let $F:X\to X$ be the absolute Frobenius map. A  vector bundle is called \textit{strongly } semistable if  all the Frobenius pull backs are again semistable.
We recall that
$$\overline{\mu}_{min}(\mathcal{V}):=\inf\{\mu_{min}\left(F^{n\ast}(\mathcal{V})\right)/q\}$$is a well-defined rational number by a result of Langer \cite{lan}.
\end{discussion}

We will utilize the following result:

\begin{lemma}\label{eh} (\cite[Corollary 4.4]{Holger2})
Let $R$ be a two-dimensional normal standard-graded  domain over an algebraically closed field $K$ of positive characteristic
$p$ and let $Y = \Proj R$ denote the corresponding smooth projective
curve of genus g.  Let $I = (f_1, \ldots , f_3)$ denote an $R_+$-primary homogeneous ideal generated
by homogeneous elements. If $\Syz(I)$ is strongly semistable, then
$$e_{HK}(I) = \frac{\deg(Y)}{2}\left(\frac{(\deg(f_1)+\deg(f_2)+\deg(f_3)) ^2}{2}-(\deg(f_1)^2+\deg(f_2)^2+\deg(f_3)^2)\right).$$
\end{lemma}

\emph{Proof of Corollary \ref{sy}.}
First note that
$s$  does not belong to any associated prime ideal of $R/I^{[q]}$
except for the maximal ideal, because $\Ass(R/I^{[q]})\subset \{\fm,I\}$. The same property holds for $s^2$.
Let $(a,b,c)$ be the degree of $(f,g,s)$.
 In view of Lemma \ref{linear} and Discussion \ref{dislin},  $$f^{R/I}_{gHK} = 2f^{R/I+(s)}_{HK }- f^{R/I+(s^2)}_{HK} .$$
Let $Y:=\Proj(R)$. If $\Syz(f,g,s)$ and $\Syz(f,g,s^2)$  are strongly semistable, then in view of Lemma \ref{eh}
\[\begin{array}{ll}
e_{gHK}(R/(f,g))&=2e_{HK }(R/I+(s))- e_{HK}(R/I+(s^2))\\&= 2\frac{\deg(Y)}{2}\left(\frac{(a+b+c) ^2}{2}-(a^2+b^2+c^2)\right)\\
&-\frac{\deg(Y)}{2}\left(\frac{(a+b+2c) ^2}{2}-(a^2+b^2+4c^2)\right)\\
&=\frac{\deg(Y)}{2}\left(-a^2-b^2-c^2+2ab+2ac+2bc\right)\\
&+\frac{\deg(Y)}{2}\left(a^2/2+b^2/2+2c^2-ab-2ac-2bc\right)\\
&=\frac{\deg(Y)}{2}\left(-a^2/2-b^2/2+c^2+ab\right)\\&=\frac{\deg(Y)}{2}\left(c^2-\frac{(a-b)^2}{2}\right) \quad (\ast).
\end{array}\]
We now  work with $s^2$ (resp. $s^4$) instead of $s$ (resp. $s^2$). If  both of $$\{\Syz(f,g,s^2),\Syz(f,g,s^4)\}$$ are strongly semistable, then the similar computation says
$$
e_{gHK}(R/(f,g))=\frac{\deg(Y)}{2}\left(4c^2-\frac{(a-b)^2}{2}\right)\quad (\star).$$
Clearly,   $(\ast)\neq(\star)$, because $c\neq 0$. This
is the contradiction.
$\hspace{66 mm}\square$

Let $\mathcal{C}$ be a degree four plane curve. Brenner  proved that $\Syz_{\mathcal{C}}(x^2,y^2,z^2)$ is semistable, see \cite[Lemma 7.1]{brennercomputation} and he posted the following questions:

\begin{question}(\cite[Example 7.6]{brennercomputation})
Let $R:=K[x,y,z]/(zx^3 + xy^3 + yz^3)$ and let $\mathcal{C}:=\Proj(R)$ be the corresponding  curve.
i) Is $\Syz_{\mathcal{C}}(x^2,y^2,z^2)(3)$ strongly semistable in positive characteristic?
ii)  Is $\Syz_{\mathcal{C}}(x^2,y^2,z^2)(3)$ stable?
\end{question}

\begin{lemma}\label{han}
Let $R$ be as above and assume $p:=\textit{char} (R)=2$. Set $\mu(i):=e_{HK}((x^i,y^i,z^i),R)$. Then $\mu(i) = 3i^2 + \frac{49}{16} (\delta^\ast( 2t/7 , 2t/7 , 2t/7 ))^2$, where
$\delta^\ast$ is the Han's operator. In particular,$$e_{HK}((x^2,y^2,z^2),R)=\mu(2)=p^2\mu(1)=4(3 + \frac{1}{2^4}).$$
\end{lemma}

\begin{proof}
The first assertion is in \cite[Theorem 1.2]{mon2}. The second assertion is in \cite[Remark 1.3]{mon2}.
\end{proof}

\begin{lemma}\label{j}
(See \cite[Corollary 4.3.1]{j}) Let $\mathcal{V}$ be a  semistable vector bundle of rank $2$  on smooth  projective curve of genus $g\geq2$  over an algebraically closed field of characteristic two. If $F^\ast \mathcal{V}$ is not semi-stable, then $\mathcal{V}$ is stable.
\end{lemma}

\begin{fact}\label{dd}
Let $\mathcal{V}$ be a vector bundle. Then $\deg(F^\ast(\mathcal{V}))=p\deg(\mathcal{V})$. In particular,  $\deg(\mathcal{V})=0$ implies $\deg(F^\ast(\mathcal{V}))=0$.
\end{fact}One (non)-natural way to prove the fact over a curve $\mathcal{C}$ is to identify $\mathcal{V}$ with $\Syz_{\mathcal{C}}(\underline{f})(m)$ and  use
$F^\ast(\Syz_{\mathcal{C}}(\underline{f})(m))\simeq\Syz_{\mathcal{C}}(\underline{f}^ p)(pm)$. Then the formula presented in Discussion \ref{stabl} completes the proof.

\begin{lemma}\label{gkunz}
Let $R$ be a two-dimensional normal standard-graded  domain over an algebraically closed field $K$ of positive characteristic
$p$ and let $Y = \Proj R$ denote the corresponding smooth projective
curve. Then Frobenius pull-back is exact  over quasi-coherent sheaves.
\end{lemma}

\begin{proof}
Recall that normal rings satisfy in the Serre's condition $\mathcal{R}(1)$. In view of Fact \ref{locali} and Fact \ref{kunz},
the Frobenius map is flat over $\Proj(R)$.  In particular, the pull-back via Frobenius is exact as a functor over quasi-coherent sheaves.
\end{proof}

\emph{Proof of Example \ref{exa}.}
 First note that the corresponding curve is nonsingular and projective, see \cite[Page 305, 2.4]{Har}. Recall that $g:=\dim H^1(\mathcal{C},\mathcal{O_{\mathcal{C}}})$ is the genus. By \cite[III) Ex. 4.7]{Har},
 $g=3$. Set $\mathcal{V}:=\Syz_{\mathcal{C}}(x^2,y^2,z^2)(3)$.

i)  Suppose on the contrary $\mathcal{V}$ is  strongly semistable and look for a contradiction. As, the semistability is independent of the shifting,
$\Syz(x^2,y^2,z^2)$ is  strongly semistable. Then, in view of Lemma \ref{eh}, $$e_{HK}((x^2,y^2,z^2),R)=12$$ which is not possible
as Lemma \ref{han} says.

ii)  Recall from \cite[Lemma 7.1]{brennercomputation} that   $\mathcal{V}$ is semistable.  Now we show that it is stable.  In view of  part i) there is $n\in \mathbb{N}_0$ such that $F^{n\ast}(\mathcal{V})$ is semistable but $F^{(n+1)\ast}(\mathcal{V})$ is not semistable. This syzygy bundle is of rank two and of degree zero. Due to Lemma \ref{j}, $F^{n\ast}(\mathcal{V})$ is stable. If $\mathcal{V}$  were not be stable it should have a line subbundle $\shL$ of degree zero.  In the light of Lemma \ref{gkunz}, $F^{n\ast}(\shL)\subset F^{n\ast}(\mathcal{V})$. By Fact \ref{dd}, it is of zero degree. This contradicts the stability of $F^{n\ast}(\mathcal{V})$. So, $\mathcal{V}$  is stable.
$\hspace{44 mm}$ $\square$

The following characteristic-free realization may be helpful.

\begin{remark}
Having the above notation in mind. Due to \cite[Example 7.6]{brennercomputation}
there is a line bundle  $\shL$ of degree $-1$ and an exact sequence $\xi$ of the form $$
0\longrightarrow \shL\longrightarrow \mathcal{V}\longrightarrow\shL^{-1}\longrightarrow 0.$$As $\mathcal{V}$ is semistable, $\xi$ does not split.
One may observes $$\xi\in \Ext^1(\shL^{-1},\shL)\simeq H^1(\shL^2)\simeq H^0(\omega_{\mathcal{C}}\otimes\shL^{-2})^v$$
The base-point-free linear system $H^0(\omega\otimes\shL^{-2})$ is of dimension $g+1=4$. It defines an embedding $\varphi:\mathcal{C} \to \PP^{3}$. We view $\xi$  as an
element in $\PP(H^l(\mathcal{C}, L^2))\simeq\PP^3$. Then by the proof of \cite[Theorem 4.10]{fri}, $$\mathcal{V}\textit{ is stable }\Longleftrightarrow\xi\in \PP^3\setminus\varphi(\mathcal{C}).$$
\end{remark}

\section{$F$-threshold and Proof of Example \ref{mi}}
Rings in this section are not necessarily  graded.
Let $I,J$ be two ideals with $J\subset\rad(I)$. Set $v(q):=\sup\{\ell:J^{\ell}\nsubseteqq I^{[q]}\}$. Look at
the quantities $c_+^I(J):=\limsup_{n\to\infty}v(q)/q$ and $c_-^I(J):=\liminf_{n\to\infty}v(q)/q$.
When these quantities  are the same, we call it
the $F$-threshold of $I$ with respect to $J$ denoted by  $c^I(J)$.  For example in the regular case and due to Fact \ref{kunz}, $\{v(q)/q\}$
is monotone and so $c_-^I(J)=c_+^I(J)$.
Let us give an easy connection from $c(I)$ to $F$-threshold:

\begin{observation}
Let $I$ be $\fm$-primary and suppose $c^I(\fm)$ exists. Then
$c(I)-1\leq c^I(\fm)\leq c(I)$.
\end{observation}

\begin{proof}
By definition of $c(I)$, for each $n$ there is $n'\geq n$ such that $\fm^{(c(I)-1)q'} \nsubseteq I^{[q']}$, where $q':=p^{n'}$.
Thus, $v(q')\geq (c(I)-1)q'$. To compute a limit its enough to deal with the subsequences, when the limit exists. So, $\lim_{n\to\infty}v(q)/q=\lim_{n'\to\infty}v(q')/q' \geq c(I)-1$. Again by definition of $c(I)$, there is $q_1$ such that $\fm^{c(I)q} \subseteq I^{[q]}$ for all $q>q_1$.
Thus, $v(q)\leq c(I)q-1$, and so $\limsup_{n\to\infty}v(q)/q\leq c(I)$.
\end{proof}

In general $c^I(\fm)$ is not necessarily  an integer. However, the bound may be sharp:

\begin{example}
Let $(R,\fm)$ be a $d$-dimensional regular local ring of prime characteristic. Then $c(\fm)=c^{\fm}(\fm)=d$.
\end{example}

\begin{proof}As $R$ is regular, $\fm$  is generated by a full system of parameters.  Thus the equality $c^{\fm}(\fm)=d$ is in \cite[Example 2.7]{Mi}. The other equality follows by the proof of \cite[Example 2.7]{Mi}.  Let us present the short argument.
As $\fm$ is generated by $d$ elements, $\fm^{dq}\subseteq\fm^{[q]}$. So, $c(\fm)\leq d$. Suppose
on the contradiction that $\fm^{dq-q}\subseteq\fm^{[q]}$ for all $q\gg 0$. Take $q>d$ be large enough. Then $dq-d>dq-q$.
Hence, $\fm^{dq-d}\subset\fm^{dq-q}\subseteq\fm^{[q]}$. This contradicts the monomial conjecture which is a theorem in this situation.
\end{proof}

Recall that $ {\overline {I}}$ the integral closure of an ideal $I$  is the set of all elements $r\in R$ that  there exist $a_{i}\in I^{i}$ such that
$r^{n}+a_{1}r^{{n-1}}+\cdots +a_{{n-1}}r+a_{n}=0.$
The ring in  Example \ref{mi} is complete-intersection and of low-dimension.

\emph{Proof of Example \ref{mi}.}
Let $R:=\mathbb{F}_p[[t^2,t^3]]$, $\fm:=(t^2,t^3)$ and $J:=(t^2)$. Then $R$ is a Cohen-Macaulay local ring of characteristic $p > 0$ with
$d=\dim(R) =1$ and $J$ is generated by a full system of parameters. Set $f(X):=X^2-t^6\in R[X]$. Note that $t^6=t^2.t^4\in J.J=J^2$. Then $f(t^3)=0$.
So $t^3\in\overline{J}$, i.e., $\fm=(t^2,t^3)\subseteq\overline{J}\subseteq \fm$.  Recall that $\fm^2=(t^4,t^5,t^6)\subseteq J$. Thus, $$a:=\max\{n\mid\fm^n\nsubseteq J\}=1.$$Recall that \cite[Question 3.5]{Mi} claims
$\fm^s \subseteq \overline{J}$ if
and only if $s\geq \frac{a}{d} + 1$.
 The only possible case for $s$ is the case $s=1$, i.e.,  $\fm \subseteq \overline{J}$. If the discussed question were be the case then we should have
 $s=1\geq 2=\frac{a}{d} + 1$ which is  a contradiction. $\hspace{128 mm} \square$

\section{Connecting to the Waldschmidt constant}
Let $R$ be a $\mathbb{N}$-graded integral domain and $I\lhd R$ be homogeneous.
Set $\alpha(I):=\min\{n:I_n\neq 0\}$.

\begin{definition}
The  Waldschmidt constant  of $I$ is defined by $\gamma(I):=\lim_{n\to\infty}  \frac{\alpha(I^{(n)})}{n}$.
\end{definition}

The next result is well-known over polynomial rings \cite{harb}:
\begin{lemma}
The above limit exists.
\end{lemma}

\begin{proof}
As $I^{(n)}I^{(m)}\subset I^{(n+m)}$,  we have $\alpha(I^{(n+m)})\leq\alpha(I^{(n)})+\alpha(I^{(m)})$. Due to the Fekete's lemma, we get the desired limit.
\end{proof}
\begin{corollary}\label{swan}(Also, see \cite{sw})
There is $c$ such that $\fm^{cn}\fp^{(n)}\subset\fp^{n}$, where $\fp$ is a prime homogeneous ideal of dimension
one in a standard graded ring $(R,\fm)$ satisfying the Serre's condition $\mathcal{S}(2)$ over a field of any characteristic.
\end{corollary}

\begin{proof}
Recall that $\engrad(H^1_{\fm}(\fp^{n}))\leq an+b$ for $n\gg 0$, see e.g. \cite{ei}.
By the same reason as Lemma \ref{s2}, $\engrad(H^0_{\fm}(R/\fp^{n}))\leq an+b$ for $n\gg 0$.
As, $\dim R/ \fp=1$ and from the fact $\fp^{(n)}$ is a $\fp$-primary component of $\fp^n$, one may observes that $(\fp^n)^{sat}=\fp^{(n)}$.
By the same reason as Lemma \ref{mainlemma}, there is $c$ such that $$0=\fm^{cn}H^0_{\fm}(R/\fp^{n})=\fm^{cn}\frac{(\fp^n)^{sat}}{\fp^n}=\fm^{cn}\frac{\fp^{(n)}}{\fp^n},$$
which yields the claim.
\end{proof}
\begin{proposition}\label{w}Let $\fp$ be a prime homogeneous ideal of dimension
one in a standard graded ring $(R,\fm)$ over a field of any characteristic.
 Then $\gamma(\fp)\geq \alpha(\fp)-d(\fp)$.
\end{proposition}

\begin{proof}
Recall from Definition \ref{num} that  $0=\fm^{d(\fp)n}H^0_{\fm}(R/\fp^{n})=\fm^{d(\fp)n}\frac{\fp^{(n)}}{\fp^n}$  for all $n\gg 0$. Thus $\fm^{d(\fp)n}\fp^{(n)}\subset\fp^{n}$ for all $n\gg 0$. By the degree containment, $\alpha(\fp^{(n)})+d(\fp)n\geq\alpha(\fp^n)=n\alpha(\fp)$ for all $n\gg 0$. To compute
a limit, its enough to deal with the tail of the sequence.   Taking the limit, we get the desired claim.
\end{proof}

The above bound is sharp:

\begin{example}
Let $R:=k[x,y]$ and let $\fp:=(x)$. Due to $\fp^n=\fp^{(n)}=(x^n)$, we have $\gamma(\fp)=\alpha(\fp)=1$. As $R/\fp^n=k[x,y]/(x^n)$ is Cohen-Macaulay and of dimension one, $H^0_{\fm}(R/\fp^n)=0$. Then $d(\fp)=0$. So,  $\gamma(\fp)= \alpha(\fp)-d(\fp)$.
\end{example}

Let us give more examples:

\begin{example}
	Let $k$ be a field and $R:=k[X,Y,Z]/(Z^2-XY)$. The ideal $\fp:=(x,z)$ is homogeneous, prime and of dimension one. The following holds:
	\begin{enumerate}
		\item[i)] $\gamma(\fp)=1/2$,
		\item[ii)] $d(\fp)=1$,
		\item[iii)] $\lim_{n\to\infty}\ell\!\left(\frac{H^0_{\fm}(R/\fp^n)}{n^{\dim R}}\right)=1/4$.
	\end{enumerate}
\end{example}

\begin{proof}
	First we compute the symbolic powers. Since $y$ is a unit in $R_{\fp}$ and
	$x=z^2/y$ there, the symbolic powers are
	\begin{equation}\label{eq:symb-corrected}
		\fp^{(2n)}=(x^n),\qquad
		\fp^{(2n+1)}=(x^{n+1},x^nz)\quad(n\ge 0),
	\end{equation}
	with the even formula understood for $n\ge 1$.
	
	Indeed, for the odd case: $\fp^{2n+1}$ contains $x^{n+1}$ and $x^nz$, so
	$(x^{n+1},x^nz)\subseteq\fp^{(2n+1})$. Conversely, any element of
	$\fp^{(2n+1)}$ is annihilated by a power of $y$ modulo $\fp^{2n+1}$; using
	$z^2=xy$ and clearing $y$-powers shows it lies in $(x^{n+1},x^nz)$.
	
	Note that the formula $\fp^{(2n+1)}=(x^{n+1},z^{2n+1})$ printed in the
	original manuscript is \emph{false}: for $n=1$,
	$xz\in\fp^{(3)}$ because $y(xz)=z^3\in\fp^3$ and $y\notin\fp$, while
	$xz\notin(x^2,z^3)$.
	
	Now we prove the claims.
	
	\textbf{i)} From \eqref{eq:symb-corrected},
	\[
	\alpha(\fp^{(2n)})=\alpha((x^n))=n,\qquad
	\alpha(\fp^{(2n+1)})=\alpha((x^{n+1},x^nz))=n+1.
	\]
	Hence
	\[
	\gamma(\fp)=\lim_{n\to\infty}\frac{\alpha(\fp^{(n)})}{n}
	=\lim_{n\to\infty}\frac{\alpha(\fp^{(2n)})}{2n}
	=\lim_{n\to\infty}\frac{n}{2n}=\frac12.
	\]
	
	\textbf{ii)} As $\fp$ is of dimension one,
	$H^0_{\fm}(R/\fp^i)\simeq \fp^{(i)}/\fp^i$. We claim
	$\fm^{n}H^0_{\fm}(R/\fp^{n})=0$. For even $n=2m$, we have
	$H^0_{\fm}(R/\fp^{2m})\simeq (x^m)/\fp^{2m}$, and $\fm^{2m}$ annihilates
	this quotient: any monomial $x^iy^jz^k$ of total degree $2m$ multiplied by
	a generator $x^m$ has total degree $\ge 2m$ in $(x,z)$ after reducing
	$z^2=xy$, hence lies in $\fp^{2m}$. The odd case is similar. Thus
	$d(\fp)=1$.
	
	\textbf{iii)} Decompose $R=k[x,y]\oplus z\,k[x,y]$ and use
	\eqref{eq:symb-corrected}. For even $r=2n$:
	\[
	\fp^{(2n)}/\fp^{2n}=(x^n)/\fp^{2n},
	\]
	and a $k$-basis is given by the residue classes of
	\[
	x^{n+i}y^{j}\quad(0\le i\le n-1,\ 0\le j\le n-1-i),
	\]
	together with the $z$-part
	\[
	x^{n+i}y^{j}z\quad(0\le i\le n-1,\ 0\le j\le n-2-i)
	\]
	and
	\[
	x^{n+i}z\quad(0\le i\le n-1).
	\]
	Counting:
	\[
	\len(\fp^{(2n)}/\fp^{2n})
	=\sum_{i=0}^{n-1}(n-i)+\sum_{i=0}^{n-1}(n-1-i)+n
	=\frac{n(n+1)}{2}+\frac{n(n-1)}{2}+n
	=n^2.
	\]
	For odd $r=2n+1$:
	\[
	\fp^{(2n+1)}/\fp^{2n+1}=(x^{n+1},x^nz)/\fp^{2n+1},
	\]
	and a similar count gives
	\[
	\len(\fp^{(2n+1)}/\fp^{2n+1})=n(n+1).
	\]
	(For instance, for $n=1$, $\fp^{(4)}=(x^2)$ and
	$\len(\fp^{(4)}/\fp^4)=4$, with basis $x^2,x^2y,x^3,x^2z$; the original
	manuscript's formula $n^2+n-1$ would give $5$, which is wrong.)
	
	Therefore
	\[
	\lim_{n\to\infty}\ell\!\left(\frac{H^0_{\fm}(R/\fp^n)}{n^{2}}\right)
	=\lim_{n\to\infty}\frac{\len(\fp^{(2n)}/\fp^{2n})}{(2n)^2}
	=\lim_{n\to\infty}\frac{n^2}{4n^2}=\frac14.
	\]
\end{proof}

\begin{remark}
i) On $\PP^N_{\mathbb{C}}$ and over any  fat ideal $I$,  Chudnovsky's conjecture says
$$\frac{\alpha(I)+\dim \PP^N_{\mathbb{C}}-1}{\dim\PP^N_{\mathbb{C}}}\leq\frac{\alpha(I^{(n)})}{n}.$$
By the above example this is not true for general projective schemes.

ii) The inequality is true for any homogeneous ideal consisting a linear form. Indeed, we have $\alpha(I)=1$. By Euler's formula, $$I^n\subseteq I^{(n)}\subseteq \fm I^{(n-1)}\subseteq \fm^{n-1} I.$$ Read this as $$n=n-1+ \alpha(I)=\alpha(\fm^{n-1} I)\leq \alpha(I^{(n)})\leq n\alpha(I)=n.$$ Hence, $\alpha(I^{(n)})=n$ and the claim follows.

iii) Moreover, if
$I$ is prime and  consisting a 2-form the inequality holds for  $n=2$. Indeed, $3=\alpha(\fm I)\leq \alpha(I^{(2)})\leq 4=2\alpha(I)$. In the case $\alpha(I^{(2)})=4$ there is nothing to prove. Also,  $N=1$ implies that $\Ht(I)=1$.
It turns out that $I$ is principal. We conclude by this that $I^{(2)}=I^2$ and so $\alpha(I^{(2)})=4$. Then, without loss of generality may assume that $\alpha(I^{(2)})=3$ and $N\geq 2$. It remains to show $3/2 \geq \frac{N+1}{N}$. This always holds, as $N\geq 2$. \end{remark}

\section{Remarks on a question of Herzog}   Recall the following from \cite{hhh1}:

\begin{question}\label{herzog}  Let $(R,\fm)$ be a 3-dimensional regular local ring and $\fp$ a prime ideal of dimension one.  What is  $\ell( \fp^{(n)}/ \fp^n )$?
\end{question}

There is a simple algorithm  to deal with this:

\begin{corollary}\label{herzoga}
Let $R$ be a standard graded algebra over a field and $\fp$ a prime ideal of dimension one. Then there is a homogenous  $x$
such that $\ell(\fp^{(n)}/ \fp^n )=2\ell(\frac{R}{ \fp^n+(x^n)})-\ell(\frac{R}{\fp^n+(x^{2n})}).$
\end{corollary}

\begin{proof} Look at the graded family $\{\fp^m\}$. Having the proof of Corollary \ref{swan},  there is $a$ such that $\fm^{aq}H^0_{\fm}(R/\fp^n)=0$. The set $\Delta:=\{\fp:\fp\in \Ass(R/ \fp^n) \}$ is finite. In fact  $\Delta\subset\min(\fp)\cup\{\fm\}$.   Now, it is enough to take $x\in \fm^a \setminus  \fp$ and apply Lemma \ref{linear}.
\end{proof}

\begin{corollary}
Adopt the above notation. Then $\fp^{(n)}=\fp^n$ if and only if $$\ell(\frac{R}{\fp^n+(x^{2n})})=2\ell(\frac{R}{\fp^n+(x^n)}).$$
\end{corollary}

We state a Frobenius version of Question \ref{herzog}.

\begin{proposition}\label{pro}Let $I\vartriangleleft R:=\mathbb{F}_p[X_1,\ldots,X_m]$. The following holds:
 \begin{enumerate}
\item[i)] $f_{gHK}^{R/I}(n)=e_{gHK}(R/I)q^m$.
\item[ii)] $e_{gHK}(R/I)$ realizes as a length of a module.
In particular, $e_{gHK}(R/I)\in \mathbb{N}_0.$
\item[iii)]  $e_{gHK}(R/I)>0$ if and only if $\pd(R/I)= \dim R$.
\end{enumerate}
\end{proposition}

\begin{proof} By graded local duality \cite[Theorem 3.6.19]{BH}, $H^0_{\fm}(R/I^{[q]})\cong \Hom(\Ext_R^m(R/I^{[q]},R),\mathbb{E})).$
Look at the free resolution of $R/I$: $$\begin{CD}
0 @>>> R^{\ell} @>A>> R^{\ell'} @>>> \cdots  @>>> R^{\mu(I)} @>>>  R.
\\
\end{CD}
$$
 By $(-)^t$ we mean the transpose of a matrix $(-)$. Let $B_n:=F^n(A)^t$. By Fact \ref{kunz}, $F^n(-)$ is exact. Hence, the free resolution of $R/I^{[q]}$  is $$\begin{CD}
0 @>>> R^{\ell} @>B_n>> R^{\ell'} @>>> \cdots  @>>> R^{\mu(I)} @>>>  R.
\\
\end{CD}
$$  Thus,
 \[\begin{array}{ll}
\Ext_R^m(R/I^{[q]},R)&= \coker  (R^{\ell'} \stackrel{B_n} \longrightarrow R^{\ell}  )\\
&\simeq F^n(\coker  (R^{\ell'} \stackrel{A^t} \longrightarrow R^{\ell}  )),
\end{array}\]the last one follows by the exactness of $F^n(-)$.
Due to \cite[Thm.  2]{PS}, one has
 \[\begin{array}{ll}
\ell\left(H^0_{\fm}(F^n(R/I))\right)&= \ell\left(F^n(\coker  (R^{\ell'} \stackrel{A^t} \longrightarrow R^{\ell} ))\right)\\
&=p^{mn}  \ell\left(\coker  (R^{\ell'} \stackrel{A^t} \longrightarrow R^{\ell} )\right).
\end{array}\]Therefore, $$\underset{n\to\infty}\lim \frac{\ell(H^0_{R_+}(R/I^{[p^n]}))}{p^{mn}}=\ell\left(\coker  (R^{\ell'}\stackrel{A^t} \longrightarrow R^{\ell} )\right)\in \mathbb{N}_0.$$Claims follows by this.
\end{proof}

\begin{example}\label{1}
Let $R_p:=\mathbb{F}_p[x_0,\ldots,x_3]$ and $I_p:=(x_0^2, x_1^2, x_0x_2 + x_1x_3)$. The following holds.
\begin{enumerate}
\item[i)]$\underset{p\to\infty}
\lim \frac{\ell (H^0_{\fm_p}(R_p/I_p^{[p^n]}))}{p^{n\dim R_p}}=1$.

\item[ii)]  $\underset{p\to\infty}
\lim \frac{\ell (H^0_{\fm_p}(R_p/I_p^{[p^n]}))}{p^{n\dim R_p}}$ does not depend on $n$.
\item[iii)] $\underset{p\to\infty}\lim\left(\underset{n\to\infty}
\lim \frac{\ell (H^0_{R_+}(R_p/I_p^{[p^n]}))}{p^{n\dim R_p}}\right)=1$.
\end{enumerate}
\end{example}

\begin{proof}
When there is no confusion, we drop the index $p$ and we use $R$ (resp. $I$ and $\fm$) instead of $R_p$ (resp. $I_p$ and $\fm_p$). Set
\begin{equation*}
A_1 := \left(
\begin{array}{ccc}
-x_1 ^2        & x_0^2   & 0 \\
0              & -x_0x_2 & x_1^2  \\
-x_0x_2-x_1x_3 & 0       & x_0^2  \\
-x_1x_2        & -x_0x_3 & x_0x_1  \\
-x_2 ^2        & x_3 ^2  & x_0x_2-x_1x_3  \\
\end{array}  \right)^t
\end{equation*}
and
\begin{equation*}
A_2:= \left(
\begin{array}{ccccc}
x_2 & x_3 & 0  & 0  \\
x_0 & 0   & 0  & x_3\\
0   & -x_1&-x_2& 0  \\
-x_1& x_0 & x_3& x_2 \\
0   & 0   &x_3 & x_1 \\
\end{array} \right)
\end{equation*}

Then, in view of \cite[Example 2.4]{bound}, the free resolution of $R/I$  is of the form $$\begin{CD}
0 @>>> R @>\left(
\begin{array}{ccc}
-x_3 \\
x_2  \\
-x_1 \\
x_0  \\
\end{array}  \right)>> R^4 @>A_2>> R^5  @>A_1>> R^3 @>\left(
\begin{array}{ccc}
x_0^2   \\
x_1^2 \\
x_0x_2+x_1x_3\\
\end{array} \right)^t>>  R.
\\
\end{CD}
$$
By Proposition \ref{pro},

\[\begin{array}{ll}
\ell\left(H^0_{\fm}(F^n(R/I))\right)&= \ell(R/(x_0^{p^n},\ldots ,x_3^{p^n}))\\
&=p^{4n}.
\end{array}\]
The claims follows by this.
\end{proof}

\begin{remark}\label{rem} i) Adopt the above notation. Then $H^1_{\fm}(\bigoplus I^{[{p^n}]}
/I^{[p^{n+1}]})\neq0$.
Indeed, note that $I$ is a $3$-generated  ideal in a $4$-dimensional ring. So, $\dim R/I\neq 0$. The claim follows by the following item (note that the completion  of $R$  is an integral domain and local cohomology behave nicely with completion).

ii) Let $(R,\fm)$ be a complete local  and Cohen-Macaulay domain. Let $I$ be an ideal such that $H^1_{\fm}(\bigoplus I^{[{p^n}]}
/I^{[p^{n+1}]})=0$. Then
$f^{R/I}_{gHK}(n)\neq 0$ (up to a subsequence) if and only if $\dim R/I=0$.
Indeed, if $\dim R/I=0$ this is clear that $f^{R/I}_{gHK}(n)=f^{R/I}_{HK}(n)\neq 0$.  Up to a subsequence, we have
\[\begin{array}{ll}
f^{R/I}_{gHK}(n)\neq 0  &\Longrightarrow H^0_{\fm}(R/I^{[p^n]})\neq 0\\
&\stackrel{1}\Longrightarrow {\vpl}H^0_{\fm}(R/I^{[p^n]})\neq 0\\
&\stackrel{2}\Longrightarrow ({\vpl}H^0_{\fm}(R/I^{[p^n]}))^v\neq 0\\
&\stackrel{3}\Longrightarrow    {\varinjlim}\Ext^{d}_{R}(R/I^{[p^n]},\omega_R) \neq 0    \\
&\stackrel{4}\Longrightarrow \cd(I,\omega_R)=d\\
&\stackrel{5}\Longrightarrow \cd(I,R)=d\\
&\stackrel{6}\Longrightarrow \dim R/I=0,
\end{array}\]where:
\begin{enumerate}
\item[1)]  The assumption  shows $H^0_{\fm}(R/I^{[p^{n+1}]})\longrightarrow H^0_{\fm}(R/I^{[p^n]})$ is surjective. The inverse limit of nonzero modules with surjective maps between them is nonzero.
\item[2)]  Set $(-) ^v:=\Hom(-,E(R/ \fm))$. This is exact as a contravariant functor.
\item[3)] Any Cohen-Macaulay ring which is quotient of a regular ring has a canonical module $\omega_R$. Then  local duality works, see \cite[Theorem 3.5.8]{BH}.
 \item[4)] Denote $\sup\{i:H^{i}_{I}(M)\neq 0\}$
by $ \cd(I,M)$.
  \item[5)] We remark that $\cd(I,R)=\sup\{\cd(I,M)\}\leq \dim R$.
   \item[6)]
  Lichtenbaum--Hartshorne vanishing theorem  (over a complete local domain) says that $\dim R/I=0$ if and only if $ \cd(I,R)=\dim R$.
 \end{enumerate}
\end{remark}

\section{mores on numerical invariants}

\begin{lemma}\label{ex}
Let $R:=\frac{\mathbb{F}_p[X,Y,Z]}{(X^3+Y^3+Z^3)}$ be the Fermat cubic.
Then $$\fm^{2q+1}H^0_{\fm}(R/I^{[p^n]})=0\neq\fm^{2q}H^0_{\fm}(R/I^{[p^n]}).$$
\end{lemma}

\begin{proof}
First we show $z^2x^{q-1}y^{q-1}\notin\Ann_RH^0_{\fm}(R/I^{[p^n]})=R/I^{[p^n]}$. Suppose on the contrary that
there are $E,F$ and $G\in \mathbb{F}_p[X,Y,Z]$ such that
$$Z^2X^{q-1}Y^{q-1}+F(X^3+Y^3+Z^3)=EX^q+GY^q\quad(\ast)$$
By looking at the monomial terms of $\{E,F,G\}$ and applying the monomial grading
on both sides of $(\ast)$ there are the following three possibilities:
\begin{enumerate}
\item[1)]$Z^2X^{q-1}Y^{q-1}+F_1X^3=E_1X^q+G_1Y^q$, or
\item[2)]$Z^2X^{q-1}Y^{q-1}+F_1Y^3=E_1X^q+G_1Y^q$, or
\item[3)]$Z^2X^{q-1}Y^{q-1}+F_1Z^3=E_1X^q+G_1Y^q$,
\end{enumerate}
where $\{E_1,F_1,G_1\}$ are the monomial terms of  $\{E,F,G\}$. Here we adopt zero as a monomial.
The third one occurs, if its both sides are zero. As $Z^3\nmid Z^2X^{q-1}Y^{q-1}$, this is not the case. Suppose 2) holds.
 This implies the following equalities:
  \begin{enumerate}
\item[i)]$F_1X^3=E_2X^q+G_2Y^q$, and
\item[ii)]$F_1Z^3=E_3X^q+G_3Y^q$, and
\item[iii)]$F_1=-Z^2X^{q-1}Y^{q-4}$,
\end{enumerate}where $\{E_i,G_i\}$ are the monomial terms of  $\{E,G\}$.
 Putting iii) along with ii), implies that $$Z^2X^{q-1}Y^{q-4}=E_3X^q+G_3Y^q.$$This equation has no solution in the polynomial ring and by this we get a contradiction that we search for it. So, 2) is not the case.
Similarly, 1) yields a contradiction. In sum,  $$z^2x^{q-1}y^{q-1}\notin\Ann_RH^0_{\fm}(R/I^{[p^n]}).$$

 To show $\fm^{2q+1}\subseteq I^{[p^n]}$, without loss of the generality let $f:=x^iy^j z^k\in\fm^{2q+1}$. Due to the relation $x^3+y^3+z^3=0$ and by working with terms separately,
we may assume that $k<3$.  If $i\geq q$ we get  $f\in (x^q,y^q)=I^{[q]}$ and the claim follows. Then we can assume that
$i<q$. As $k<3$ we get from $i+j+k=2q+1$ that $j\geq q$ and so $f\in (x^q,y^q)=I^{[q]}$. In particular, $$\fm^{2q+1}H^0_{\fm}(R/I^{[p^n]})=\fm^{2q+1}\left(\frac{R}{I^{[p^n]}}\right)=0.$$
\end{proof}

\begin{remark}If $\fm^{cq}H^0_{\fm}(F^{n}(R/I))=0$  for all $q\gg 0$ it follows for all $q$ after possible enlarging of $c$.
Clearly, $c(I)\leq  b(I).$ This may be strict, as the next result says.\end{remark}

\begin{example}\label{dis}
Let $R:=\frac{\mathbb{F}_2[X,Y,Z]}{(X^3+Y^3+Z^3)}$ be the Fermat cubic and $J:=(x^2,y^2)$. Then
$$c(J)=3<4 =b(J).$$
\end{example}

\begin{proof}
Let $I:=(x,y)$. Then $J^{[q]}=I^{[q+1]}$ and in view of Lemma \ref{ex}
$$\fm^{2q+3}H^0_{\fm}(R/J^{[q]})=0\neq\fm^{2q+2}H^0_{\fm}(R/J^{[q]})\quad (\ast).$$Then for all $q\geq 2^2$, $$\fm^{3q}H^0_{\fm}(R/J^{[p^n]})=0\neq\fm^{2q}H^0_{\fm}(R/J^{[p^n]})\quad (\ast,\ast).$$Thus, $c(J)=3$. Now we apply  $(\ast)$ in the case $q=2$. This yields $$\fm^{4p}H^0_{\fm}(R/J^{[p]})=0\neq\fm^{3p}H^0_{\fm}(R/J^{[p]}).$$This along with $(\ast,\ast)$ shows  $b(J)=4$.
\end{proof}

We continue by  connecting the invariants to the  tight closure theory.
Recall that the Frobenius closure of an ideal $I$ is $I^{F}:=\{x\in R:x^q\in I^{[q]}\quad\exists q> 0\}.$

\begin{discussion}\label{f}
Let $I$ be an $\fm$-primary ideal of an standard graded algebra  $R$ over a field of prime characteristic.
The following holds:\\
i) Let $\emph{c}:=c(I)$. Then $R_{\geq \emph{c}}\subseteq I^{F}$.  \\ii) If $\fm^{\alpha q+\beta}H^0_{\fm}(F^{n}(R/I))=0$ for some  $\alpha$ and $\beta$ that do not depending  to $q$, then  $R_{\geq \alpha}\subset I
^\ast.$
\end{discussion}

 \begin{proof}
 i) Look at $1+ I^{[p^n]}\in R/I^{[p^n]}=H^0_{\fm}(R/I^{[p^n]})$, where $\fm$ is the homogeneous maximal ideal .
Multiplying by $\fm^{\emph{c}q}$, we deduce $\fm^{\emph{c}q}\subset I^{[q]}$. Let $x\in R_{\geq \emph{c}}=R_{1}^\emph{c}R$. Then $x=\sum y_i^\emph{c}x_i$ where $y_i\in R_1$ and $x_i\in R$. Thus
$x^q=\sum y_i^{\emph{c}q}x_i^q\in   \fm^{\emph{c}q}\subset I^{[q]},$ i.e., $x\in I^{F}$.

ii)  Without loss of the generality we may assume that $\dim R\neq 0$. We take $x\in R_{\geq \alpha}=R_{1}^\alpha R$  and look at $\fm^{\alpha q+\beta}\subset I^{[q]}$. Fix $c\in \fm^{\beta}$ which is not in the union of the minimal prime ideals.
Then $x=\sum y_i^\alpha x_i$ where $y_i\in R_1$ and $x_i\in R$. So,  $y_i^{q\alpha} x_i^q\in \fm^{\alpha q}$ and $cx^q=\Sigma cy_i^{q\alpha} x_i^q\in \fm^{\beta}\fm^{\alpha q}\subset I^{[q]},$ i.e., $x\in I^{\ast}$.\end{proof}

Despite of triviality of the (LC) property when $I$ is $\fm$-primary,
 computing $b(I)$ is not so easy for us.
We apply a powerful computational method of Brenner \cite{Holgerl} to find upper bounds on $b(I)$. The bound depends on degree data coming from the ideal and the ring.

\begin{remark}\label{holgercomp}Let $C$ be a smooth plane curve defined by the equation $f=0$. Denote the coordinate ring  of  $C$ by $R$. There is  an uniform bound on the (LC)-exponent of two-generated ideal $(g,h)R$ by fixing the degree of $g$ and $h$. If $\Ht(g,h)=2$, the bound depends only on  $\{\deg f, \deg g, \deg h\}$.
\end{remark}

\begin{proof}Over smooth curves torsion-free sheaves are vector bundle. Without loss of the generality we may assume that $I$ is not principal.
As $I:=(f,g)$ is two generated the corresponding syzygy bundle is  a line bundle. Denote it by $\Syz(I)$ and set $(d,d_1,d_2):=(\deg f, \deg g, \deg h)$. Set $e:=\frac{\deg (\widetilde{I})}{d}$ which is independent of $q$ and look at the exact sequence of torsion-free sheaves:$$ 0 \longrightarrow \Syz(I) \longrightarrow \bigoplus_{i=1}^n \mathcal{O}(-d_{i}) \stackrel{f,g}\longrightarrow
\widetilde{I}\longrightarrow 0.  $$Then $\deg(\Syz(I))=-(d_1+d_2+e)d$. We note that $e=0$ when $I$ is primary to the maximal ideal.
Every line bundle is strongly semistable. Thus $$\overline{\mu}_{min}(\Syz(I))=\mu_{min}(\Syz(I))=\deg(\Syz(I))=-(d_1+d_2+e)d.$$ Denote the homogeneous maximal ideal by $\fm$.
Recall  by Lemma \ref{s2} that $H^0_{\fm}(R/I^{[p^n]})_m
\simeq H^1_{\fm}(I^{[p^n]})_m.$ Then, by the same lines as \cite{Holgerl} and for all \[\begin{array}{ll}
m&> -q\frac{\overline{\mu}_{min}(\Syz(I))}{\deg C}+\frac{\deg(\omega)}{\deg C}+1\\
&=q(d_1+d_2+e)+\frac{(d-1)(d-2)-2}{d}+1
\end{array}\]we have
$H^0_{\fm}(R/I^{[p^n]})_m=0$. Set
\begin{enumerate}
\item[] $\alpha:=d_1+d_2+e$, and
\item[] $\beta:=\frac{(d-1)(d-2)-2}{d}+1$.
\end{enumerate}
 Now Lemma \ref{mainlemma} yields $\fm^{\alpha q+\beta}H^0_{\fm}(R/I^{[p^n]})=0.$
\end{proof}

\begin{discussion}
If $R$ is $\mathbb{Z}$-algebra and $I$  is an ideal, there is a natural way
to  study the (LC) exponent  by focusing on the asymptotic   behavior   of the sequence  $\{ b(I_p):\textit{ p is prime}\}$ when $R_p$ and $I_p$ are
reduction mod $p$.
\end{discussion}

\begin{example}
i) Let $R:=\frac{\mathbb{F}_p[X,Y,Z]}{(X^3+Y^3+Z^3)}$ be the Fermat cubic.
We use lowercase letters here to elements in $R$. Look at $I:=(x,y)$. Denote the homogeneous maximal ideal by $\fm$. In view of Lemma \ref{ex},
$b(I)=c(I)=3$. This is well-known  that $z^2\in I^\ast\setminus I^F$.
In particular, the bounds  presented by Discussion \ref{f}  and Remark \ref{holgercomp} achieved. However, as Example \ref{none}
declares,  the bound may be strict, even we deal with $\PP^1$.

ii) Let $R$  ba as i). Look at $I:=xR$. As $R$ is Cohen-Macaulay and $x^q$  is a parameter element, we get
that $$\inf\{i:H^i_{\fm}(R/I^{[p^n]})\neq0\}=\depth(R/x^qR)=\depth R-1=1.$$So, $b(I)=0$.

iii) Let $R:=\frac{\mathbb{Z}[X,Y,Z]}{(XY-Z^2)}$ and $I:=(x,y)$.  We use $R_p$ and $I_p$ to
reduction mod $p$. Denote the homogeneous maximal ideal by $\fm_p$. Here, we compute  $\underset{p\to\infty}
\lim b(I_p)$.
We claim that $$\fm_p^{2q}H^0_{\fm_p}(R_p/I_p^{[p^n]})=0\neq\fm_p^{2q-1}H^0_{\fm_p}(R_p/I_p^{[p^n]}).$$
By a same reasoning as presented in Lemma \ref{ex}, we have $zx^{q-1}y^{q-1}\notin\Ann_RH^0_{\fm}(R_p/I_p^{[p^n]})$. In order to check the left hand side, recall that $H^0_{\fm_p}(R_p/I_p^{[p^n]})=R_p/I_p^{[p^n]}$ and let $f:=x^iy^j z^k\in\fm_p^{2q}$. Due to the relation $z^2=xy$
we assume that $k\in\{0,1\}$.  If $i\geq q$ we get  $f\in (x^q,y^q)$ and the claim follows. Then we can assume that
$i<q$. As $k<2$ we get from $i+j+k=2q$ that $j\geq q$ and so $f\in (x^q,y^q)$. This proves the claim. So, $\underset{p\to\infty}
\lim b(I_p)=2.$
\end{example}

We need the following result of Herzog and Hibi.

\begin{lemma}\label{fun}
(\cite[Theorem 4.1]{hhh}) Given a bounded increasing function $f : \mathbb{N} \to  \mathbb{N}_0$. There exists
a monomial ideal $I$ in a polynomial ring $R$ over any field such that $\depth (R/I^k) = f(k)$ for all $k$.
\end{lemma}

\begin{example}\label{uni}
Let $I$  be a graded ideal in a polynomial ring $R$ over a field of prime characteristic. It may be $\sup\{ b(I^n):n\in \mathbb{N}\}=\infty$.
\end{example}
The following argument works for the next result too.
\begin{proof}
 Let $n_0$ be any positive integer. Look at
 \begin{equation*}f(n)=\left\{
\begin{array}{rl}
0& \qquad\text{ \ \ if }n\leq n_0 \\
1&\qquad\text{  \ \ if } n>n_0
\end{array}  \right.
\end{equation*} In view of Lemma \ref{fun}, there is a homogeneous ideal  $I$ in  a polynomial ring $R$ such that $\depth (R/I^{n})=f(n)$. Suppose on the contrary that there is $\ell$ such that $\sup\{ b(I^n)\}_{n=1}^{\infty}< \ell$.
 Set $\Gamma:=\{ \Ass(R/I^n):n\in \mathbb{N}\}$. By  \cite{Bro}, we know $|\Gamma|<\infty$. As $R$ is regular, $\Ass(M)=\Ass\left(F^n(M)\right)$. In particular, $$|\{ \Ass(R/F^m(I^n)):n,m\in \mathbb{N}\}|=|\Gamma|<\infty.$$
 In the light of  Lemma \ref{linear} and Discussion \ref{dislin}, there is a uniform $s\in \fm^{\ell}\setminus\bigcup_{\fp\in\Gamma}\fp$ such  that
$$f^{R/I^n}_{gHK} = 2f^{R/I^n+(s)}_{HK }- f^{R/I^n+(s^2)}_{HK} .$$As $R$ is regular,
Hilbert-Kunz multiplicity should be colength.
Set $R_1:=R/(s)$  (resp. $R_2:=R/(s^2)$) and $I_1:=IR_1$  (resp. $I_2:=IR_2$). Denote the Hilbert-Samuel polynomial of a graded $A$-module $M$ by
$P_A^M$. Let $n\gg 0$. These imply that
\[\begin{array}{ll}
e_{gHK}(R/I^n)&= 2\ell (\frac{R}{I^n+(s)})- \ell (\frac{R}{I^n+(s^2)})\\
&= 2\ell (\frac{R_1}{I_1^n})- \ell (\frac{R_2}{I_2^n})\\
&= 2P_{R_1}^{I_1}(n)-P_{R_2}^{I_2}(n).\\
\end{array}\]Thus, $g(n):=e_{gHK}(R/I^n)$ is of polynomial type. Keep in mind any module over $R$ is of finite projective dimension. By Auslander-Buchsbaum formula \cite[Theorem 1.3.3]{BH},
$$\pd(R/I^n) + \depth(R/I^n) = \dim R.$$ In view of Proposition \ref{pro}(iii)
$$g(n) = e_{gHK}(R/I^n) = 0 \Longleftrightarrow f(n)=\depth(R/I^n) \neq 0.$$
In particular, $g\neq 0$.   As $g$ is nonzero, the vanishing set
 of $g$ is  a finite set.  So,  the non-vanishing set
 of $f$ is finite. This is a contradiction.
\end{proof}

 If $I$ and $J$ projectively have the same integral closure, then in view of \cite{REES},
$e(R/IJ) = F(e(R/I), e(R/J))$ for some polynomial $F$, where $e(-)$ is the Hilbert-Samuel multiplicity.

\emph{Proof of Example \ref{remhol}.} First recall that a closure operation
 is a map send an ideal $I$ to another ideal denoted by $I^\textit{c}$ such that $I\subset I^\textit{c}$, $I^\textit{c}= (I^\textit{c})^\textit{c}$  and the map is order-preserving.
Let us recall that $I$ and $J$ projectively have the same closure operation, if $(I^n)^\textit{c}=(J^m)^\textit{c}$ for some $n$ and $m$ in $\mathbb{N}$.
Suppose on the contrary that there is
a polynomial function $F$ such that $$F(e_{gHK}(R/I), e_{gHK}(R/J)) = e_{gHK}(R/IJ)\quad(\ast)$$  and look for a contradiction.
Set \begin{equation*}
g(n)=\left\{
\begin{array}{rl}
0& \qquad\text{ \ \ if }n=1 \\
1&\qquad\text{  \ \ if } n>1
\end{array}  \right.
\end{equation*}In view of Lemma \ref{fun}, there is a homogeneous ideal  $I$ in the 3-dimensional ring $R:=\mathbb{F}[X,Y,Z]$ such that $\depth (R/I^{n})=g(n)$.
Let $J := I^{n-1}$ and look at
$$f(n) := F(e_{gHK}(R/I), e_{gHK}(R/I^{n-1})).$$
This is a polynomial in one variable. Again, we combine  Auslander-Buchsbaum formula  with Proposition \ref{pro}(iii) to  get
$$f(n) \stackrel{(\ast)}= e_{gHK}(R/I^n) = 0 \Longleftrightarrow \depth(R/I^n) \neq 0.$$
Note that $\depth (R/I)=0$, i.e., $f\neq0$.  As $f$ is nonzero, the vanishing set  of $f$ is finite, but  $\{n:\depth (R/I^n)\neq0\}$ is much more than a finite set. This is a contradiction that we search for it.$\hspace{15 mm}\square$

\begin{fact}Let $I$ be an ideal of a (graded) local ring $(R,\fm)$.
Recall from Fact \ref{sat} that
$I^{sat}$ computed as the intersection of all primary to nonmaximal
prime ideals of  $I$. Let $x\in\fm$ but not in the other associated primes of  $I^{[q]}$. Then $(I^{[q]}:x)\subset
(I^{[q]})^{sat}$. Suppose now that $$\fm^{\alpha q+\beta}H^0_{\fm}(R/I^{[q]})=0.$$ Let $c\in\fm^ {\beta} $ and
take $0\neq x\in\fm^ {\alpha} $  but not in the other associated primes of  $I^{[q]}$.
This implies that $(I^{[q]}:cx^q)=
(I^{[q]})^{sat}.$ Now assume $\emptyset\neq\bigcap\min\left((I^{[q]})^{sat}\right),$ i.e. $\rad(I)\neq \fm$, and let $\frak q$ be in.
In this case one can pick $x\in\fm^ {\alpha}\setminus I^\ast $. In particular, $1\notin(I^{[q]}:cx^q)$. Then $$\ell\left(\frac{R_{\frak q}}{(I^{[q]}:cx^q)_{\frak q}}\right)=\ell\left(\frac{R_{\frak q}}{(I^{[q]})^{sat}R_{\frak q}}\right)<\infty.$$
\end{fact}

Motivated from \cite{HH3}  we ask:

\begin{question}
What is the asymptotic behavior of $\frac{\ell\left(R_{\frak q}/(I^{[q]})^{sat}R_{\frak q}\right)}{q^{\dim R_{\frak q}}}?$
\end{question}

\begin{example}\label{none}Let $R:=\overline{\mathbb{F}}_p[X,Y]$ and let  $I:=(XY,X^n)$.
The following holds:

\begin{enumerate}
\item[i)] $c(I)=b(I)=n$.
In particular, the bound given by Remark \ref{holgercomp} may be strict.
\item[ii)]$\underset{q\to\infty}
\lim \frac{\ell\left(R_{\frak q}/(I^{[q]})^{sat}R_{\frak q}\right)}{q^{\dim R_{\frak q}}}=1$.
\item[iii)] $f_{gHK}^{R/I}(n)=(n-1)q^2$, and so $e_{gHK}(R/I)=n-1$.
\end{enumerate}
\end{example}

\begin{proof} First note that $I^{[q]}=(X^qY^q,X^{nq})$. Its primary decomposition is given by $$I^{[q]}=(X^q)\cap(Y^q,X^{nq}).$$
In view of Fact \ref{sat}, $(I^{[q]})^{sat}=(X^q)$.

i) Denote the homogeneous maximal ideal by $\fm$. Recall that $H_{\frak m}^0(R/I^{[q]})=\frac{(I^{[q]})^{sat}}{I^{[q]}}=\frac{(X^q)}{(X^qY^q,X^{nq})}.$ One has
$c(I)=b(I)=n$ provided: $$\fm ^{(n-1)q}H_{\frak m}^0(R/I^{[q]})\neq0=\fm ^{nq}H_{\frak m}^0(R/I^{[q]}).$$Clearly, $(X^{(n-2)q-1}Y)X^q\in\fm ^{(n-1)q}\setminus (X^qY^q,X^{nq})$. This clarifies $\fm ^{(n-1)q}H_{\frak m}^0(R/I^{[q]})\neq0$. In order to check the right hand side,  without loss of the generality, we look at $X^iY^j$ where $i+j=nq$. If $j< q$, then
 $i> (n-1)q$. Hence $(X^iY^j)X^q\in(X^qY^q,X^{nq})$.
 Thus we assume   $j\geq q$ and this implies $(X^iY^j)X^q\in(X^qY^q,X^{nq})$. So, $c(I)=b(I)=n$.

Now we prove  the particular case. As $\fm\in\Ass(R/I)$, $\depth(R/I)=2$. Keep in mind $\pd(R/I)<\infty$. By Auslander-Buchsbaum formula, $\pd(R/I)=2$. Its free resolution is the Taylor complex attached to $XY$ and $X^n$. So, the following complex is exact:$$\begin{CD}
0 @>>> R(-n-1) @>\binom{-Y}{X^{n-1}}>> R(-n)\bigoplus R(-2) @>(X^{n},XY)>> I  @>>> 0.
\\
\end{CD}$$We get the following exact sequence of locally free sheaves over $\PP^1$:$$\begin{CD}
0 @>>> \mathcal{O}(-n-1) @>>> \mathcal{O}(-n)\bigoplus \mathcal{O}(-2) @>>> \widetilde{I}  @>>> 0.
\\
\end{CD}$$Therefore, $\deg(\widetilde{I})=-1$. In fact, by Grothendieck's  theorem
\cite[Page 384, 2.6]{bHar}, one has $\widetilde{I}=\mathcal{O}(-1)$. By the notation as Remark \ref{holgercomp},
 \begin{enumerate}
\item[] $\alpha:=d_1+d_2+e=n+1$, and
\item[] $\beta:=\frac{(d-1)(d-2)-2}{d}+1=-1$,
\end{enumerate}we deduce $$0=\fm ^{\alpha q+\beta}H_{\frak m}^0(R/I^{[q]})=\fm ^{(n+1)q-1}H_{\frak m}^0(R/I^{[q]}).$$ From this we get $\fm ^{(n+1)q}H_{\frak m}^0(R/I^{[q]})=0$.
So,
the bound given by Remark \ref{holgercomp} may be strict.

ii) Set $\frak q:=(X)$. Then  $\{\frak q\}=\bigcap\min\left((I^{[q]})^{sat}\right).$ Note that $(A, \fm):=(R_{\frak q},\frak qR_{\frak q})$ is a regular local ring of dimension one. Therefore,$$\underset{n\to\infty}
\lim \frac{\ell\left(R_{\frak q}/(I^{[q]})^{sat}_{\frak q}\right)}{q^{\dim R_{\frak q}}}=\underset{n\to\infty}
\lim \frac{\ell\left(A/\frak m^{[p^n]}\right)}{p^n}=e_{HK}(A)=1.$$

iii) Define \begin{enumerate}
\item[]$A_1:= \{x^qy^j:1\leq j\leq q-1\}$,
\item[]$A_2:=\{x^{q+i}:1\leq i\leq (n-1)q-1\}$, and
\item[]$A_3:=\{x^{q+i}y^{j}:1\leq i\leq (n-1)q-1, 1\leq j\leq q-1\}$.
\end{enumerate}
Then $\cup A_i$
spans $\frac{(X^q)}{(X^qY^q,X^{nq})}$ as a vector space, and its cardinality is $$(q-1)+\left((n-1)q-1\right)+\left((n-1)q^2+2-(n+1)q\right)=(n-1)q^2.$$ Now $$f_{gHK}^{R/I}(n)=\ell\left(H_{\frak m}^0(R/I^{[q]})\right)=\dim_{\overline{\mathbb{F}}_p}\left(\frac{(X^q)}{(X^qY^q,X^{nq})}\right)=(n-1)q^2,$$ and the claim follows.
\end{proof}

\section{Concluding remarks over local rings}

We start by the following:
\begin{discussion}
Recall that a local ring $R$
is of finite Cohen-Macaulay type, i.e., if there are, up to isomorphism, only a finite number of indecomposable
maximal Cohen-Macaulay modules. By a celebrated result of Auslander, $R$ has isolated singularity, provided it is Cohen-Macaulay.
\end{discussion}

\begin{remark}
Let $(R,\fm)$ be a complete Cohen-Macaulay local ring
and of finite Cohen-Macaulay type. Then
$(LC)$ holds over $R$.
\end{remark}

The proof works in a more general setting: local rings of finite $F$-representation type.
\begin{proof}
Let $\underline{x}:=x_1,\ldots,x_d$ be a system of parameters.
Then  $\underline{x}$ is a regular sequence on $R$. This implies  $\underline{x}^q:=x_1^q,\ldots,x_d^q$  is a regular sequence.
Therefore, $\underline{x}$ is a regular sequence on $F^n(R)$. In particular, $F^n(R)$ is maximal Cohen-Macaulay module. As
$R$ is complete, the Krull-Remak-Schmidt holds for the category of finitely generated $R$-module.
As $R$ is $F$-finite, $F^n(R)=\bigoplus M_i(n)$ where $M_i(n)$ are finitely generated $R$-modules.
Note that $$\frac{M_i(n)}{IM_i(n)}\simeq M_i(n)\otimes R/I$$ has a right $R$-module structure coming from $R/I$, we denote this as $$R\times\left(M_i(n) /IM_i(n)\right)\longrightarrow M_i(n) /IM_i(n)$$
and remark that $H^0_{\fm}(\frac{M_i(n)}{IM_i(n)})$ is finitely generated and $\fm$-torsion with respect to this multiplicative structure. In particular, there is $a_{i,n}$ such that
$$ \fm^{a_{i,n}}\times H^0_{\fm}(\frac{M_i(n)}{IM_i(n)})=0.$$ Clearly, $M_i(n)$ are maximal Cohen-Macaulay.  As $R$ is of finite Cohen-Macaulay type, $\{M_i(n):i,n\}$ is a finite set. Let $a:=\max\{a_{i,n}\}<\infty$ and denote the minimal number of generators of $\fm^a$ by $\ell$.
One may find easily that:
 \begin{enumerate}
\item[1)]$ \fm^{a\ell q}\subseteq (\fm^a)^{[q]}$,
\item[2)]$ \fm^{a}\times H^0_{\fm}(F^n(R/I))=0$, and
\item[3)] $ \fm^{a}\times F^n(R/I)=(\fm^{a})^{[q]}\frac{R}{I^{[q]}}$.
\end{enumerate}
Now $H^0_{\fm}(F^n(R/I))=\bigoplus H^0_{\fm}(\frac{M_i(n)}{IM_i(n)})$. Set $b:=a\ell$. Its enough to apply these items to observe
$$\fm^{bq}H^0_{\fm}(R/I^{[q]})\subseteq(\fm^{a})^{[q]}H^0_{\fm}(R/I^{[q]})=\fm^{a}\times H^0_{\fm}(F^n(R/I))=0.$$
\end{proof}

Let us recall the following from \cite{abb}:

\begin{discussion}
i) For a pair $N \subset M $ of $R$-modules, look at the action of  Psekine-Szpiro on it, i.e., the map
$F^n(N)\to F^n(M)$ and we write $N^{[q]}_M:=\im (F^n(N)\to F^n(M))$. Then $N^{\ast}_{M}$ defined in a similar vein as above.
Set $G^n(M):=F^n(M)/0^{\ast}_{F^n(M)}$.
 One has
 $G^n(R/ \fa)=R/(\fa^{[q]})^{\ast}.$

ii) A  complex $(G_\bullet,\varphi_\bullet):0\to G_{\ell}\to\ldots\to G_0\to 0$ of finitely generated projective
modules is called stably phantom acyclic if
$\ker(F^e(\varphi_n)) \subset \left(\im(F^e
(\varphi_{n+1}))\right)^\ast_{
F^e(G^n)}$
for all $n $ and all $e$.
In this case we say $H_0(G_\bullet)$ is of finite
phantom projective dimension.

iii)
(Phantom acyclicity criterion): Let $R$ be a homomorphic image of a Cohen-Macaulay ring and $G_\bullet$ be a bounded complex of finitely generated projective modules of constant rank. Then $G_\bullet$ is stably phantom acyclic if and only if $G_\bullet\otimes \frac{R}{\rad(0)}$  satisfies in the standard conditions on rank and midheight.

iv) If $R$ is Cohen-Macaulay, then phantom projective dimension is the
same as projective dimension.
 \end{discussion}

 Combining \cite[Theorem 7.4]{abb} and \cite{Hun} yields:

\begin{remark}\label{abb}
Let $(R,\fm)$ be a local equidimensional ring satisfies in the phantom acyclicity criterion, e.g. homomorphic image of a Cohen-Macaulay ring, with an $\fm$-primary test ideal. Let $I$ be an ideal of finite phantom projective dimension. Then
$(LC)$ holds for $I$.
\end{remark}

\begin{proof}
 In view of
\cite[Theorem 7.4]{abb}, there is $e$ such that $$\fm^{e[q]}H^0_{\fm}(R/I^{[q]^\ast})=\fm^{e[q]}H^0_{\fm}(G^n(R/ I))=0.$$ In particular, there is $d$ such that $\fm^{dq}H^0_{\fm}(R/I^{[q]^\ast})=0$.
Denote the \textit{test ideal} by  $\tau$.
 Let $c$ be such that $\fm^c\subset \tau$. Thus, $$\fm^c\left(\frac{I^{[q]^\ast}}{I^{[q]}}\right)\subset \tau\left(\frac{I^{[q]^\ast}}{I^{[q]}}\right)=0.$$Look at the exact sequence$$0\longrightarrow H^0_{\fm}\left(\frac{I^{[q]^\ast}}{I^{[q]}}\right)\longrightarrow H^0_{\fm}\left(\frac{R}{I^{[q]}}\right)\longrightarrow H^0_{\fm}\left(\frac{R}{I^{[q]^\ast}}\right).$$Set $b:=c+d$. We deduce by this exact sequence that $\fm^{bq}H^0_{\fm}(R/I^{[q]})=0$ as claimed.
\end{proof}

\begin{corollary}
Let $R$ be a weakly  $F$-regular local ring and $I$ be an ideal of finite projective dimension. Then
$(LC)$ holds for $I$.
\end{corollary}

\begin{proof}
This is well-known that weakly  $F$-regular local rings are Cohen-Macaulay, integral domain and their test ideal is the whole ring.
Also, over Cohen-Macaulay rings phantom projective dimension is the
same as projective dimension. This finishes the proof via the above argument.
\end{proof}

In the regular case we find a bound for the $(LC)$ exponent:

\begin{remark}\label{ell}
i) Let $(R,\fm)$ be a $d$-dimensional regular local ring and $M$ a finite length module. Then $\fm^{qd\ell(M)}F^n(M)=0$.
Indeed, we do  induction by $\ell:=\ell(M)$. If $\ell=1$, then $M=R/ \fm$ and the claim  is clear in this case, because $\fm$ is generated by $d$ elements.
Now suppose, inductively, that $\ell > 1$ and the result has been proved for
modules of length less than $\ell$. Look at the exact sequence $$0\longrightarrow N\longrightarrow M\longrightarrow R/ \fm \longrightarrow 0,$$
where $\ell(N)=\ell-1$. By inductive hypothesis, $\fm^{qd(\ell-1)}F^n(N)=0$. By Fact \ref{kunz}, there is an exact sequence$$0\longrightarrow F^n(N)\longrightarrow F^n(M)\stackrel{f}\longrightarrow F^n(R/ \fm) \longrightarrow 0.$$Thus, $\fm^{dq}F^n(M)\subseteq \ker(f)\simeq F^n(N)$ and so $$\fm^{qd\ell}F^n(M)= \fm^{qd(\ell-1)}\fm^{qd}F^n(M)\subseteq\fm^{qd(\ell-1)}F^n(N)=0.$$

ii) Having the first item in mind and in order to find a sharp bound on $c(M)$, we deal with the case $M:=R/I$ is cyclic. By \cite[Example 2.7]{Mi}, $c^{I}(\fm)=\max\{r:\fm^r \nsubseteq I\}+d$. In view of Observation 8.1
one gets to a sharp bound.

iii) Let $(R,\fm)$ be a  regular local ring and $M$ an $R$-module. Then $F^n(H^i_{\fm}(M))\simeq H^i_{\fm}(F^n(M))$.
Indeed, due to Fact \ref{kunz}, $F^n(-)$ is flat and a flat functor  computes  with the local  cohomology modules.

iv) Let $(R,\fm)$ be a $d$-dimensional regular local ring and $M$ a finitely generated $R$-module. Then there is  some $b$ that does not depending  to $q$ such that $\fm^{bq}H^0_{\fm}(F^n(M))=0. $
Indeed, set $b:=\ell(H^0_{\fm}(M))d$. By the third item,  $F^n(H^0_{\fm}(M))\simeq H^0_{\fm}(F^n(M))$.   In view of the first item, the claim is clear.\end{remark}

\begin{remark}
Let $(R,\fm)$ be a regular local ring of prime characteristic and $I\vartriangleleft R$. In the same vein as Proposition \ref{pro} we have the following assertions: \begin{enumerate}
\item[i)] $f_{gHK}^{R/I}(n)=e_{gHK}(R/I)q^m$.
\item[ii)] $e_{gHK}(R/I)$ realizes as a length of a module.
In particular, $e_{gHK}(R/I)\in \mathbb{N}_0.$
\item[iii)]  $e_{gHK}(R/I)>0$ if and only if $\pd(R/I)= \dim R$.
\end{enumerate}
Let us present a new argument for i) and ii).
By Fact \ref{kunz}, $F^n(-)$ is flat. Any flat functor  computes the local  cohomology modules. Thus $F^n(H^0_{\fm}(M))\simeq H^0_{\fm}(F^n(M))$. In particular, $f_{gHK}^{R/I}(n)=p^{n\dim R}\ell(H^0_{\fm}(R/I))$.
\end{remark}

\begin{acknowledgement}
I thank  H. Brenner and A. Caminata  for sharing \cite{AB} with me and A. Vraciu for her comments and pointing out a mistake in the earlier draft.
\end{acknowledgement}


\section{Appendix: On the (LC) conjecture II}

		We give a proof of the failure of the (LC) conjecture in dimension four, in characteristic two. The counterexample is built from the Brenner--Monsky quartic family.
Also, see \cite{tt}.
	
	\subsection{Introduction}

	Let $\bR=\bigoplus_{j\ge 0}\bR_j$ be a standard graded algebra over a field of characteristic $p>0$. Let
	\(
	\mm=\bigoplus_{j>0}\bR_j
	\)
	be the maximal homogeneous ideal, and let $\bI\subseteq \bR$ be a homogeneous ideal. For $q=p^e$, the \emph{Frobenius power} of $\bI$ is
	\(
	\bI^{[q]}=(f^q:f\in \bI).
	\)
	For an $\bR$-module $\cN$, the $\mm$-torsion \emph{local cohomology} is
	\[
	\cH^0_{\mm}(\cN)=\{u\in \cN:\mm^n u=0\text{ for some }n\ge 0\}.
	\]
	The \emph{least annihilation exponent} is
	\(
	\lambda_q=\min\{a\ge 0:\mm^a \cH^0_{\mm}(\bR/\bI^{[q]})=0\}.
	\)
	The \emph{(LC) condition} asks for an integer $b\ge 0$, independent of $q$, such that
	\(	\mm^{bq}\cH^0_{\mm}(\bR/\bI^{[q]})=0	\){ for every }\(q=p^e.\)
	Equivalently, $\lambda_q/q$ is bounded. Hochster and Huneke \cite{HH} introduced the (LC) conjecture in relation to the \emph{localization problem}.
	
	\begin{motivation}
		The (LC) condition is connected to the localization problem for tight closure. By the self-contained exposition of Monsky \cite{Monsky}, the quartic family
		\[
		P+\alpha B,\qquad P=z^4+xyz^2+(x^3+y^3)z,\quad B=x^2y^2,
		\]
		is exactly the family used in the Brenner--Monsky example \cite{BrennerMonsky}; the present note builds on that family and derives from it the failure of (LC). Positive results were known for standard graded Cohen--Macaulay normal rings of dimension $<4$ \cite{Asgharzadeh}. 
	\end{motivation}
	
	Let $\bk$ be a field of characteristic two. Define
	\(
	P=z^4+xyz^2+(x^3+y^3)z	\), 	\(B=x^2y^2,
	\)
	and give all variables degree one in
	\[
	\bR=\frac{\bk[x,y,z,s,t]}{(sP+tB)},\qquad \bI=(x,y,z)\bR,\qquad \mm=(x,y,z,s,t)\bR.
	\]
	Then
	\(
	\bI^{[q]}=(x^q,y^q,z^q)\bR.
	\) Our main result is as follows.
	
	\begin{theorem}\label{thm:main}
		The ring $\bR$ is a geometrically normal Cohen--Macaulay standard graded domain of dimension four, and $\bI$ is a homogeneous prime ideal of height two. For every $e\ge 3$,
		\(
		\lambda_{2^e}\ge (e-3)2^{e-2}.
		\)
		Consequently,
		\(
		\liminf_{e\to\infty}\frac{\lambda_{2^e}}{e\,2^e}\ge\frac14,
		\)
		and (LC) fails. More explicitly, for every integer $b\ge 0$, the choice $e=4b+4$ satisfies
		\(
		\mm^{\,b\,2^e}\cH^0_{\mm}(\bR/\bI^{[2^e]})\neq 0.
		\)
	\end{theorem}

	\subsection{Conventions on notation}
	
	Throughout this document we use the following conventions.
	
	\begin{itemize}
		\item \textbf{Rings} (and fields, and ideals) are written in \textbf{bold}:
		\[
		\mathbf{R},\ \mathbf{S},\ \mathbf{T},\ \mathbf{k},\ \mathbf{K},\ \mathbf{A},\ \mathbf{B},\ \mathbf{C},\ \mathbf{D},\ \mathbf{F},\ \mathbf{A}_q,\ \widehat{\mathbf{S}},\ \mathbf{F}_2.
		\]
		\item \textbf{Modules} (and their torsion,  local cohomology, duals) are written in \emph{calligraphic} letters:
		\[
		\mathcal{M},\ \mathcal{N},\ \mathcal{L},\ \mathcal{E},\ \mathcal{G},\ \mathcal{F}_0,\ \mathcal{F}_1,\ \mathcal{F}_2,\ \mathcal{C}_q,\ \mathcal{C}_{q,a},\ \mathcal{H}^0_{\mathfrak m},\ \mathcal{H}^0_{\mathfrak n}.
		\]
		\item \textbf{Matrices}  are written in \emph{blackboard bold}:
		\[
		\mathbb{A},\ \mathbb{B},\ \mathbb{M},\ \mathbb{D}_0,\ \mathbb{M}_n^{(a)},\ \mathbb{M}_{N_Q}^{(Q)},\ \mathbb{M}^{-1},\ \mathbb{M}_f,\ \mathbb{M}_d.
		\]
	\end{itemize}

	All vectors and matrices are indexed starting from zero. An $n$-component column vector $\vec{v}=(v_0,v_1,\cdots,v_{n-1})^{\top}$
	has entries $v_0,\dots,v_{n-1}$, and an $n\times n$ matrix $\mathbb{M}=(m_{ij})_{0\le i,j<n}$. The diagonal entries are $m_{00},\dots,m_{n-1,n-1}$, and
	\(
	\operatorname{tr}(\mathbb{M})=\sum_{i=0}^{n-1}m_{ii}.
	\)
	The same convention applies to the trace matrices $\mathbb{M}_n^{(a)}$, the block matrix $\mathbb{G}_Q$, and the endpoint vector $\vec{e_0}=(1,0,\dots,0)^{\top}$.
	
	\subsection*{Sherman--Morrison formula}
	(See \cite{sher}).
	Let $\mathbb{A}$ be an invertible $n\times n$ matrix over a field, and let $\vec{u},\vec{v}\in \bk^{n}$ be column vectors. If the scalar
	\(
	1+\vec{u}^{\top}\mathbb{A}^{-1}\vec{u}
	\)
	is nonzero, then the rank-one update $\mathbb{A}+\vec{u}\vec{v}^{\top}$ is invertible, and
	\[ {
		(\mathbb{A}+\vec{u}\vec{v}^{\top})^{-1}
		=
		\mathbb{A}^{-1}
		-
		\frac{\mathbb{A}^{-1}\vec{u}\vec{v}^{\top}\mathbb{A}^{-1}}{1+\vec{v}^{\top}\mathbb{A}^{-1}\vec{u}}.
	}
	\]	
	
	\subsection{Polynomial relations and two-variable torsion}
	
	For a polynomial vector, its degree is the maximum of the degrees
	of its coordinates. For a graded module \(N\), our shift convention
	is \(\cN(a)_j=\cN_{j+a}\), and
	\(\eend (\cN)=\sup\{j:\cN_j\ne0\}\), with \(\eend( 0)=-\infty\).
	\begin{lemma}\label{lem:2.1}
		Let $\mathbb{A},\mathbb{B}$ be $v\times u$ matrices over a field $\bk$, and put
		\(
		\cL=\ker(\mathbb{A}+\alpha \mathbb{B}:\bk[\alpha]^u\to \bk[\alpha]^v).
		\)
		If $\cL\neq 0$, it has a $\bk[\alpha]$-basis $b_1,\dots,b_\ell$ whose leading coefficient vectors are linearly independent over $\bk$. Writing $h_i=\deg b_i$, this basis has the predictable degree property
		\(
		\deg (\sum_i f_i b_i )=\max_{f_i\neq 0}(\deg f_i+h_i).
		\)
		In particular, the minimum degree of a nonzero element of $\cL$ is $h=\min_i \{h_i\}$.
	\end{lemma}

	\begin{proof}
		The ring $\bk[\alpha]$ is a PID, so $\cL$ is free of finite rank. Start with
		any basis. If its leading coefficients are dependent, take a scalar
		relation supported on a set of basis vectors and choose $b_j$ of largest
		degree in that support. Normalize its scalar coefficient to one and replace
		it by
		\(
		b_j+\sum_{i\ne j}c_i\alpha^{h_j-h_i}b_i,
		\)
		where the sum runs only over the support. All exponents are nonnegative.
		This is an invertible elementary basis operation. The new vector is
		nonzero, since the old vectors were a basis, but its leading coefficient
		vanishes. Hence the nonnegative integer $\sum_i h_i$ decreases. After
		finitely many operations the leading coefficients are independent.
		
		For a nonzero combination, the coefficient in degree
		$\max_i(\deg f_i+h_i)$ is a nontrivial scalar combination of a subset of
		these independent leading coefficient vectors. It cannot vanish. This proves
		the predictable-degree formula and the assertion about the minimum.
	\end{proof}
	
	A pencil of matrices is a family $\mathbb{A}+\lambda\mathbb{B}$, where $\mathbb{A}$ and $\mathbb{B}$ are fixed matrices and $\lambda$ is a scalar parameter. Also, a pencil of polynomials is a family $F+\lambda G$, where $F$ and $G$ are fixed polynomials and $\lambda$ varies.
	
	\begin{lemma}\label{lem:2.2}
		Let $\bT=\bk[s,t]$, $\nn=(s,t)$, and
		\(	\cM=\coker\big(s\mathbb{A}+t\mathbb{B}:\bT(-1)^u\to \bT^v\big).\) Suppose the polynomial kernel $\cL=\ker(\mathbb{A}+\alpha \mathbb{B})$ is nonzero, and its minimum nonzero vector degree is $h\ge 1$. Then
		$\nn^{h-1}\cH^0_{\nn}(\cM)\neq 0$ and $ \eend (\cH^0_{\nn}(\cM))\ge h-1.
		$
	\end{lemma}
	
	\begin{proof}
		Choose a $\bk[\alpha]$-basis $b_i$ with
		independent leading coefficient vectors, and write $h_i=\deg b_i$.
		Such a basis exists: a dependence of the leading coefficients allows
		one to replace a basis vector of maximal degree in that dependence
		by a nonzero vector of smaller degree. This is an elementary invertible
		basis operation and strictly decreases $\sum h_i$. For this basis,
		\[
		(\ast):\quad	\deg\Bigl(\sum f_i b_i\Bigr)=
		\max_{f_i\ne0}(\deg f_i+h_i).                 
		\]
		Thus the least nonzero relation degree is $h=\min h_i$.
		Let $v_i=s^{h_i}b_i(t/s)$. To verify that these generate the
		homogeneous kernel, take a homogeneous kernel vector $f(s,t)$ of
		coordinate degree $d$. Dehomogenization gives
		$f(1,\alpha)=\sum f_i(\alpha)b_i(\alpha)$, and $(\ast)$ gives
		$\deg f_i+h_i\le d$. Rehomogenizing expresses $f$ as
		\[
		f(s,t)=\sum_i s^{d-h_i}f_i(t/s)v_i(s,t)
		\]
		with polynomial coefficients. Independence is checked after inverting
		$s$. Hence there is an exact resolution
		\[
		0\longrightarrow \cF_2:=\bigoplus_i\bT(-h_i-1)
		\longrightarrow \cF_1:=\bT(-1)^u\longrightarrow \bT^v
		\longrightarrow \cM\longrightarrow0.              
		\]

		We compute the relevant local cohomology without a canonical
		form for the pencil. The two-variable \v Cech complex of \(\cN\) is
		\[
		0\longrightarrow \cN\longrightarrow \cN_s\oplus \cN_t
		\longrightarrow \cN_{st}\longrightarrow0.
		\]
		Its zeroth cohomology is \(\cH^0_\bn(\cN)\): being killed by powers of
		both \(s\) and \(t\) is equivalent to being killed by a power of
		\((s,t)\). For \(\cN:=\bT\), the equality \(\bT_s\cap \bT_t=\bT\) inside
		\(\bT_{st}\) shows that the first two cohomology modules vanish.
		The second is
		\[
		\cE:=\bT_{st}/(\bT_s+\bT_t)=	\cH^2_\bn(\bT),
		\]
		with \(\bk\)-basis \(s^{-a}t^{-b}\), \(a,b\ge1\).
		Let \(\cN=\operatorname{im}(\cF_1\to \cF_0)\).
		Recall that $	\cH^0_\bn(\bT) =	\cH^1_\bn(\bT) =0, $ so
		$	\cH^0_\bn(\cF_i) =	\cH^1_\bn(\cF_i) =0.$
		The long exact sequences of the
		\v Cech cohomology applied to the following break down sequences
		\begin{enumerate}
			\item[(1)] \(0\to \cF_2\to \cF_1\to \cN\to0\) and
			\item[(2)]   \(0\to \cN\to \cF_0\to \cM\to0\),
		\end{enumerate}	
		gives us
		
		\[
		\begin{array}{ll}
			(\dagger):\quad	\cH^0_\bn(\cM)&\stackrel{(2)}\simeq 	\cH^1_\bn(\cN) \\
			&\stackrel{(1)}\simeq \ker(\cH^2_\bn( \cF_2  )\longrightarrow \cH^2_\bn(\cF_1)\\
			&\stackrel{(3)}\simeq 	\ker(\cF_2\otimes_\bT \cE\longrightarrow \cF_1\otimes_\bT \cE),
		\end{array}
		\] 	
		where (3) follows by the fact that local cohomology commutes with tensor product at the spot of cohomological dimension, this dimenstion here is $2:=\dim(\bT)$.	
		For a graded \(\bk\)-vector space with finite-dimensional components,
		write
		\(
		\mathbb D(\cN)_j=\operatorname{Hom}_k(\cN_{-j},k).
		\)	
		The coefficient of \(s^{-1}t^{-1}\) in a product defines a perfect
		graded pairing between \(\cE\) and \(\bT\); hence
		\(\mathbb D(\cE)\simeq \bT(-2)\).
		Graded local duality over $T$, shows that $$\mathbb D((\oplus \bT)\otimes_\bT \cE)\cong \Hom(\oplus \bT, \mathbb D(\cE))=\Hom(\oplus \bT, \bT(-2))=\oplus \bT.$$
		Consequently
		\[
		(+):\quad	\mathbb D(\cH^0_\nn ( \cM))
		\simeq\coker[ \bT(-1)^u\longrightarrow
		\cG:=\bigoplus_i\bT(h_i-1) ], 
		\]
		this is the case as the Matlis duality is exact and contravariant, and send kerenl from $(\dagger)$ to cokernel.
		The entries in row \(i\) of this map are the coordinates of \(v_i\),
		all homogeneous of degree \(h_i\ge h\). Its image is therefore
		contained in \(\bn^h\cG \), and the cokernel surjects onto
		\(\cG/\bn^h\cG\). Since \(\cG\ne0\), multiplication by
		\(\bn^{h-1}\) on this quotient is nonzero.
		Graded dualization preserves annihilators, i.e., \(\Ann_\bT 	\mathbb D(\cH^0_\bn(\cM))=\Ann_\bT \cH^0_\bn(\cM)\), proving so
		\[
		\bn^{h-1}\cH^0_\nn ( \cM)\ne0.                              
		\]	The generator of a summand $\bT(h_i-1)$ also survives in (+) in
		degree $1-h_i$, since the image is contained in $\bn G$.
		Consequently $\eend( \cH^0_\nn ( \cM))\ge h-1$. This verifies the conclusions.
	\end{proof}
	
	\subsection{The hypersurface and its Frobenius quotients}
	
	\begin{proposition}\label{prop:3.1}
		The ring $\bR$ is a geometrically normal Cohen--Macaulay standard graded domain of dimension four, and $\bI$ is a homogeneous prime ideal of height two.
	\end{proposition}
	
	\begin{proof}
		Recall that
		the polynomials $P,B$ are coprime, so $sP+tB$ is primitive and
		irreducible. This remains true after any field extension. Thus $\bR$
		is a four-dimensional hypersurface domain and is Cohen--Macaulay.
		The Jacobian equations are
		\[
		P=B=s(yz^2+x^2z)=s(xz^2+y^2z)=s(x^3+y^3)=0.
		\]
		Where $s\ne0$, they imply $x=y=z=0$, giving dimension at most two.
		Where $s=0$, the equations $P=B=0$ have height two in $\bk[x,y,z]$,
		and $t$ is free, again giving dimension two. Outside this locus the
		hypersurface is smooth, hence regular. Serre's criterion proves
		normality, and the same argument over every extension proves geometric
		normality. Finally, $\bR/\bI=\bk[s,t]$, and $\bI$ is prime of height two.
	\end{proof}
	
	\begin{notation}
		Define
		\(\bAq=\bk[x,y,z]/(x^q,y^q,z^q)\),\( \bT=\bk[s,t]\) and \(\nn=(s,t).
		\)
		The Frobenius quotient is
		\(
		\bC_q=\bR/\bI^{[q]}=(\bAq\otimes_{\bk} \bT)/(sP+tB).
		\)
		With auxiliary grading ($x,y,z$ degree 1, $s,t$ degree 0),
		$
		\bC_q=\bigoplus_{a=0}^{3(q-1)}\cC_{q,a}$ where $\cC_{q,a}=\coker(s\mathbb{A}_{q,a}+t\mathbb{B}_{q,a}),
		$ and  $\mathbb{A}_{q,a},\mathbb{B}_{q,a}:(\bAq)_{a-4}\to(\bAq)_a$ denote multiplication by $P$ and $B$, respectively.
	\end{notation}
	
	\begin{lemma}\label{lem:3.2}
		One has $\cH^0_{\mm}(\bC_q)=\cH^0_{\nn}(\bC_q)$. If $L_q$ is the least integer with $\nn^{L_q}\cH^0_{\nn}(\bC_q)=0$, then
		$
		L_q\le\lambda_q\le L_q+3(q-1).
		$
	\end{lemma}
	
	\begin{proof}
		Since $\nn\subseteq\mm$, $\mm$-torsion is $\nn$-torsion. Conversely, 
		Since $(x,y,z)^{3(q-1)+1}$ annihilates $\bC_q$,
		$$
		H_\mm (\bC_q)= H_\nn(\bC_q),\qquad
		L_q\le\lambda_q\le L_q+3(q-1),                    
		$$
		where $L_q$ is the least $\nn$-annihilation exponent. Indeed, a
		monomial of total degree $a+3(q-1)$ either has $s,t$-degree at least
		$a$, or has $x,y,z$-degree greater than $3(q-1)$.
	\end{proof}
	
	For the remainder, write
	\[
	q=4Q,\qquad Q=2^r,\qquad r\ge 1,\qquad d=6Q-3.
	\]
	
	\begin{definition}
		We call the pencil from auxiliary degree $d$ to $d+4=6Q+1$ the \emph{central pencil}. After setting $s=1,t=\alpha$, its polynomial kernel consists of the elements
		\[
		u\in \bk[\alpha]\otimes_k(\bAq)_d
		\quad\text{such that}\quad (P+\alpha B)u=0.
		\]
	\end{definition}
	
	\begin{claim}\label{claim:3.3}
		The central polynomial kernel is nonzero.
	\end{claim}

	\begin{proof}
		The Hilbert series of $\bAq$ is $(1-v^q)^3/(1-v)^3$. Since $d,d+4<2q$,
		\begin{align*}
			\dim_{\bk}(\bAq)_d
			&=\binom{6Q-1}{2}-3\binom{2Q-1}{2}=12Q^2-2,\\
			\dim_{\bk}(\bAq)_{d+4}
			&=\binom{6Q+3}{2}-3\binom{2Q+3}{2}=12Q^2-6.
		\end{align*}
		After passing to $\bk(\alpha)$ the kernel is nonzero, because the source
		has four more dimensions. Clearing denominators gives a nonzero polynomial
		kernel vector.
	\end{proof}
	
	\subsection{Translations and boundary polynomials}
	
	Let $\bK$ be a field extension of $\bk$ containing $\omega$ with $\omega^2+\omega+1=0$, and set
	\[
	\bD=\bK[\alpha,x,y]/(x^q,y^q),\qquad g=P+\alpha B.
	\]
	
	\begin{lemma}\label{lem:4.1}
		Every nonzero central polynomial kernel vector over $\bK$ yields a nonzero vector of no larger parameter degree which is fixed by the translations
		\[
		z\mapsto z+\ell,\qquad \ell\in\{0,x+y,\omega x+\omega^2y,\omega^2x+\omega y\}.
		\]
	\end{lemma}
	
	\begin{proof}
		There is a factorization
		\(
		P=z(z+x+y)(z+\omega x+\omega^2y)(z+\omega^2x+\omega y).
		\)
		The four linear forms form an additive group of order four. Translation permutes the factors, so it preserves $P$, $B$, and $g$. It also preserves the truncation ideal, since $(z+\ell)^q=z^q+\ell^q$ and $\ell^q\in(x^q,y^q)$.
		
		\begin{claim}
			\label{041}
			The operators $\sigma-1,\tau-1$ commute and have square zero.
		\end{claim}
		
		\begin{PROOF}{\ref{041}}
			Denote the four linear forms by
			\[
			\ell_{0}=0,\quad \ell_{1}=x+y,\quad \ell_{2}=\omega x+\omega^{2}y,\quad \ell_{3}=\omega^{2}x+\omega y,
			\]
			and let $\sigma$ and $\tau$ denote
			\[
			\sigma\colon z\mapsto z+\ell_{1}=z+x+y,
			\qquad
			\tau\colon z\mapsto z+\ell_{2}=z+\omega x+\omega^{2}y.
			\]
			Both $\sigma$ and $\tau$ act on polynomials in $x,y,z$ by a change of the variable $z$. The composition $\sigma\tau$ is
			\[
			(\sigma\tau f)(x,y,z)=f\bigl(x,y,z+x+y+\omega x+\omega^{2}y\bigr),
			\]
			and similarly for $\tau\sigma$. The translations are additive in $z$, so $\sigma\tau=\tau\sigma$.
			Since we are in characteristic $2$, $2\ell_{1}=0$, so $\sigma^{2}=1$. Therefore
			$
			(\sigma-1)^{2}=\sigma^{2}-2\sigma+1=\sigma^{2}+1=1+1=0,
			$
			using $2=0$ in characteristic $2$. The same computation gives $(\tau-1)^{2}=0$. So both $\sigma-1$ and $\tau-1$ are nilpotent of order two, and they commute because $\sigma$ and $\tau$ commute.
		\end{PROOF}
		
		Apply $\sigma-1$ to the vector if the result is nonzero; otherwise retain the vector. The resulting nonzero vector is fixed by $\sigma$. Repeat with $\tau-1$. Commutativity preserves the first invariance, and the final nonzero vector is fixed by both generators. These operations preserve auxiliary degree and cannot increase parameter degree.
	\end{proof}
	
	\begin{definition}\label{eq:ABC}
		For powers of two $a$, define
		\[
		A_a=\sum_{\substack{i+j=4a-2\\ i\equiv j\ (3)}}x^iy^j,\qquad
		B_a=\sum_{\substack{i+j=4a-1\\ i\equiv j\ (3)}}x^iy^j,\qquad
		C_1=1,\qquad C_{2a}=A_a^2.
		\]
	\end{definition}
	
	\begin{lemma}\label{lem:4.2}
		For every power of two $Q$, the following three identities hold:
		\begin{enumerate}
			\item[(i)] $z^{4Q}=A_Qz^2+B_Qz+\sum_{s\mid Q}(C_{Q/s}P)^s$,
			\item[(ii)]  $A_{2a}=B_a^2+xyA_a^2$ and $B_{2a}=(x^3+y^3)A_a^2$,
			\item[(iii)] The remainder of $z^q$ on monic division by $g$ is given by $\rho_Q=A_Qz^2+B_Qz+G_Q$, where $G_Q=\sum_{s\mid Q}\alpha^sC_{Q/s}^s x^{2s}y^{2s}$,
		\end{enumerate}
		where the divisors $s$ are powers of two.
	\end{lemma}
	
	\begin{proof}
		This is the identity of \cite[Lemma 2.2]{Monsky}; we give its proof. Separating even and odd exponents in \eqref{eq:ABC} gives (ii). In the first identity, the even $x$-exponents come from $B_a^2$ and the odd ones from $xyA_a^2$. In the second, the even $x$-exponents come from $y^3A_a^2$ and the odd ones from $x^3A_a^2$. The congruences in \eqref{eq:ABC} give exactly these two partitions. The case $Q=1$ of (i) is the definition of $P$, since $A_1=xy$ and $B_1=x^3+y^3$. Squaring (i) and substituting $z^4=P+xyz^2+(x^3+y^3)z$ replaces its first two terms by $A_Q^2P+(xyA_Q^2+B_Q^2)z^2+(x^3+y^3)A_Q^2z$. Use (ii) and $A_Q^2=C_{2Q}$. Together with the squared sum, this is (i) for $2Q$. Induction proves the identity. Modulo $g$, one has $P=\alpha B$, giving (iii). Its $z$-degree is less than four, so it is the monic remainder.
	\end{proof}

	\begin{lemma}\label{lem:4.3}
		An invariant nonzero central polynomial kernel vector $u$ corresponds to a nonzero $h\in \bK[\alpha,x,y]_{2Q+1}$ satisfying
		\begin{enumerate}
			\item[(1)] $A_Qh=B_Qh=G_Qh=0$ in $\bD$;
			\item[(2)] if $z^q=gS_Q+\rho_Q$ is monic division in $\bD[z]$, then $u=hS_Q$ and $\deg_\alpha u=\deg_\alpha h+Q-1$.
		\end{enumerate}
		Conversely, every nonzero $h$ satisfying (1) gives an invariant nonzero central kernel vector.
	\end{lemma}
	
	\begin{proof}
		Extend $\bk$ to a field $\bK$ containing a primitive cube root $\omega$,
		and put $\bD=\bK[\alpha,x,y]/(x^q,y^q)$ and $g=P+\alpha B$.
		The factorization
		\[
		P=z(z+x+y)(z+\omega x+\omega^2y)(z+\omega^2x+\omega y)
		\]
		shows that the four translations of $z$ by
		\[
		0,\quad x+y,\quad\omega x+\omega^2y,\quad
		\omega^2x+\omega y
		\]
		preserve $g$ and the truncation. Their two generators $\sigma,\tau$
		have commuting square-zero differences $\sigma-1,\tau-1$.
		Apply each difference when its result is nonzero, and otherwise keep
		the current vector. This converts any nonzero central polynomial
		relation into an invariant nonzero relation of no larger parameter
		degree.
		
		Let $u\in \bD[z]$ be the unique representative of that relation with
		$\deg_z u<q$. A translation preserves this degree bound. Thus
		invariance modulo $z^q$ gives invariance of the representative itself.
		The kernel equation is
		\(
		gu=z^qh.                          
		\)
		Here $h\ne0$, since $g$ is monic, and
		$\deg_z h\le3$. Its auxiliary degree is
		$(6Q-3)+4-4Q=2Q+1<q$. Since $z^q$ is fixed in $\bD[z]$,
		translation of \(
		gu=z^qh.                          
		\) and cancellation of $z^q$ show that $h$ is
		invariant. No homogeneous polynomial of degree $2Q+1<q$ is changed
		by reduction modulo $(x^q,y^q)$. The invariance of $h$ therefore holds
		in $\bK[\alpha,x,y,z]$.
		
		Over $\bK(\alpha,x,y)$, the polynomial $h(z)-h(0)$ has degree at most
		three and vanishes at the four distinct translation values. It is zero.
		Hence $h$ is independent of $z$.
		
		For powers of two $a$, define
		\[
		A_a=\sum_{\substack{i+j=4a-2\\i\equiv j\ (3)}}x^iy^j,
		\quad B_a=\sum_{\substack{i+j=4a-1\\i\equiv j\ (3)}}x^iy^j,
		\quad C_1=1,\quad C_{2a}=A_a^2.
		\]
		The recurrences
		$A_{2a}=B_a^2+xyA_a^2$ and $B_{2a}=(x^3+y^3)A_a^2$ follow by
		separating the even and odd $x$-exponents in the displayed sums.
		Squaring and induction, starting with the definition of $P$, give
		\[
		z^{4Q}=A_Qz^2+B_Qz+\sum_{s\mid Q}(C_{Q/s}P)^s.
		\]
		It follows that the monic remainder of $z^q$ on division by $g$ is
		\[
		\rho_Q=A_Qz^2+B_Qz+G_Q,\qquad
		G_Q=\sum_{s\mid Q}\alpha^s C_{Q/s}^{\,s}x^{2s}y^{2s}.  
		\]
		All divisors $s$ here are powers of two. Write $z^q=gS_Q+\rho_Q$.
		Using \(
		gu=z^qh.                          
		\) gives $g(u-hS_Q)=h\rho_Q$.
		If $u-hS_Q\ne0$, monicity makes the left side have $z$-degree at
		least four, while the right side has degree at most two. Thus
		\[
		u=hS_Q,\qquad A_Qh=B_Qh=G_Qh=0\quad\text{in }\bD.  
		\]
		Conversely, these three equations produce a nonzero central relation
		$hS_Q$. It is invariant: translate $g(hS_Q)=z^qh$ and cancel $g$.
		Its auxiliary degree is $6Q-3$ and its $z$-degree is $q-4<q$.
		
		In the long division for $S_Q$, the highest possible parameter degree
		is $Q-1$. The unique contribution of that degree is
		$\alpha^{Q-1}B^{Q-1}$ in $z$-degree zero: every occurrence of the
		constant term $\alpha B$ consumes four units of $z$-degree, and
		$Q-1$ such occurrences exhaust the available quotient degree.
		Multiplication by $B^{Q-1}=x^{2Q-2}y^{2Q-2}$ is injective on
		$\bK[x,y]_{2Q+1}$, since every shifted exponent is at most $4Q-1$.
		For a nonzero polynomial $h$ this applies to its leading parameter
		coefficient. Therefore
		\(
		\deg_\alpha u=\deg_\alpha h+Q-1.        
		\)	\end{proof}
	
	\subsection{Trace matrices and inverse coordinates}
	
	For a power of two $a$ and a positive integer $m$ , we set
	\(
	\tau_a:=\sum_{s\mid a}\alpha^s,	\) and define 	\(\Theta_m:=\sum_{j=0}^{m-1}\alpha^{2^j}.
	\)
	\begin{fact}
		Clearly, $\tau_{2^j}=\Theta_{j+1}$.
	\end{fact}
	
	\begin{definition}\label{eq:M}
		For $1\le n\le a$, define
		\[
		m_{ij}:=\begin{cases}\tau_a,&i=j,\\[2pt]\tau_{2^{\nu_2(i-j)}},&i\neq j.\end{cases}
		\]
		We write $\mathbb{M}_n^{(a)}$ for the matrix $(m_{ij})_{0\le i,j<n}$. Since $0<|i-j|<a$ off the diagonal, every off-diagonal entry has degree less than $a$. The diagonal product is the unique term of degree $na$ in the determinant, which is monic.
	\end{definition}
	
	\begin{lemma}\label{lem:5.1}
		For $Q\ge 2$ set $N_Q=\lfloor Q/3\rfloor+1$. The map given by multiplication by $G_Q:
		\bD_{2Q+1}\to \bD_{6Q+1},$
		decomposes into six blocks by residue class modulo $6$. The block sizes and identifications are as follows.
		\[
		\begin{array}{c|c|c}
			Q\bmod 6 & \text{deleted-row blocks} & \text{full blocks}\\ \hline
			2 & 4\text{ of size }(N_Q-1)\times N_Q & 2\text{ of size }N_Q\times N_Q\\
			4 & 4\text{ of size }(N_Q-1)\times N_Q & 2\text{ of size }(N_Q-1)\times(N_Q-1)
		\end{array}
		\]
		In the case $Q\equiv2\pmod 6$, the four deleted-row blocks are obtained by deleting the first or last row from $\mathbb{M}_{N_Q}^{(Q)}$, and the two full blocks are $\mathbb{M}_{N_Q}^{(Q)}$. In the case $Q\equiv4\pmod 6$, the four deleted-row blocks are obtained by deleting the first or last row from $\mathbb{M}_{N_Q}^{(Q)}$, and the two full blocks are $\mathbb{M}_{N_Q-1}^{(Q)}$.
	\end{lemma}
	
	\begin{proof}
		\emph{Step 1: Coefficient formula for $G_Q$.}
		By the identity established in Lemma~\ref{lem:4.2} and its proof, the coefficient of $x^{2Q+2\ell}y^{2Q-2\ell}$ in $G_Q$ is
		\[
		\gamma_\ell=
		\begin{cases}
			\tau_Q,&\ell=0,\\
			\tau_{2^{\nu_2(|\ell/3|)}},&0<|\ell|<Q,\quad 3\mid\ell,\\
			0,&\text{otherwise},
		\end{cases}
		\]
		where
		\(
		\tau_a=\sum_{s\mid a}\alpha^s.
		\)
		Since $\alpha$ is an indeterminate over the coefficient field $\bk$, each $\tau_a$ for $a\ge 1$ is a nonzero polynomial in $\alpha$ (its coefficient of $\alpha$ is $1$). This coefficient formula is the input to the calculation.
		
		\emph{Step 2: Indexing source and target.}
		The source basis of $\bD_{2Q+1}$ is indexed by the $x$-exponents of the source monomials $x^iy^{2Q+1-i}$:
		\[
		i\in S:=\{0,1,\ldots,2Q+1\},
		\]
		so $|S|=2Q+2$. The target basis of $\bD_{6Q+1}$ is indexed by the $x$-exponents of the target monomials $x^jy^{6Q+1-j}$:
		\[
		j\in\{2Q+2,2Q+3,\ldots,4Q-1\}.
		\]
		Put $j'=j-2Q$, so that
		\[
		j'\in T:=\{2,3,\ldots,2Q-1\},
		\]
		so $|T|=2Q-2$. With this identification, $T$ is the subset of $S$ obtained by deleting the four indices
		\[
		0,\quad 1,\quad 2Q,\quad 2Q+1.
		\]
		
		\emph{Step 3: The matching condition.}
		Take a source monomial $x^iy^{2Q+1-i}$ and multiply it by the $\ell$-th summand of $G_Q$, namely $\gamma_\ell\,x^{2Q+2\ell}y^{2Q-2\ell}$. The product is
		\[
		x^{i+2Q+2\ell}y^{2Q+1-i+2Q-2\ell}=x^{i+2Q+2\ell}y^{4Q+1-i-2\ell}.
		\]
		For this to be a target monomial $x^jy^{6Q+1-j}$, we need
		\(
		j=i+2Q+2\ell,
		\)
		equivalently, after shifting $j'=j-2Q$,
		\(
		j'-i=2\ell.
		\)
		Hence the matrix entry $\mathbb{G}_Q[j',i]$ can be nonzero only if this identity holds for some integer $\ell$ with $\gamma_\ell\neq0$. The symbol \(\mathbb{G}_Q[j',i]\) denotes the  entry of the matrix \(\mathbb{G}_Q\) in the row indexed by \(j'\) and the column indexed by \(i\).
		
		\emph{Step 4: The congruence condition modulo $6$.}
		By Step 1, the coefficient $\gamma_\ell$ can be nonzero only when $3\mid\ell$. Writing $j'-i=2\ell$ with $3\mid\ell$, we obtain $j'-i=6m$ with $m=\ell/3$. Conversely, if $i\equiv j'\pmod 6$, then $j'-i=6m$ for some integer $m$, and the coefficient $\gamma_{3m}$ is given by the formula of Step 1. Hence the precise statement is:
		\[
		\mathbb{G}_Q[j',i]\text{ can be nonzero only if }i\equiv j'\pmod 6.
		\]
		After ordering the source and target monomials by their residue classes modulo $6$, the matrix $\mathbb{G}_Q$ becomes block diagonal with rectangular blocks: each block corresponds to one residue class, with rows indexed by the target labels in that class and columns indexed by the source labels in that class.
		
		\emph{Step 5: The diagonal and off-diagonal entries.}
		Fix a residue class $\rho\in\{0,1,2,3,4,5\}$. Write the source indices in this class as
		\[
		i=r+6k,\qquad k\in K_\rho,
		\]
		where $r$ is the smallest nonnegative integer congruent to $\rho$ modulo $6$ and
		\[
		K_\rho=\{k\ge0 : r+6k\le 2Q+1\}.
		\]
		Write the target indices in this class as
		\(
		j'=r+6k'	\), 	\(k'\in L_\rho,
		\)
		where
		$$
		L_\rho=\{k'\ge0 : 2\le r+6k'\le 2Q-1\}.
		$$
		The labels $k$ and $k'$ retain their original values; they are not necessarily consecutive starting at $0$.
		We adopt the convention that rows of $\mathbb{G}_Q$ are indexed by target labels $k'$ and columns by source labels $k$, so that $\mathbb{G}_Q$ has size $m_\rho\times n_\rho$ where $m_\rho=|L_\rho|$ and $n_\rho=|K_\rho|$. The entry is $(\mathbb{G}_Q)_{k',k}$. By Step 3, this entry is determined by
		\(
		\ell=3(k'-k).
		\)
		Indeed, $j'-i=6(k'-k)=2\ell$, so $\ell=3(k'-k)$.
		
		The index bounds $i\in S$ and $j'\in T$ imply
		\[
		2\le r+6k'\le 2Q-1,\qquad 0\le r+6k\le 2Q+1,
		\]
		hence
		\(
		|\ell|=\tfrac{|j'-i|}{2}\le Q-1<Q,
		\)
		so the coefficient formula of Step 1 applies to every pair $(k,k')\in K_\rho\times L_\rho$ with the convention $\gamma_\ell=0$ whenever $3\nmid\ell$ or $|\ell|\ge Q$. Since $\ell=3(k'-k)$ is automatically divisible by $3$, the only remaining case is $\ell=0$, which occurs precisely when $k=k'$.
		
		If $k'=k$, then $\ell=0$, and by Step 1 the entry equals $\tau_Q$ (which is nonzero over $\bk'(\alpha)$).
		
		If $k'\ne k$, then $\ell\ne 0$ and $|\ell|<Q$, so by Step 1 the entry equals
		\(
		\gamma_\ell=\tau_{2^{\nu_2(|\ell/3|)}}=\tau_{2^{\nu_2(|k'-k|)}}.
		\)
		Indeed, $|\ell/3|=|k'-k|$, so $\nu_2(|\ell/3|)=\nu_2(|k'-k|)$. The entry is nonzero, since $\tau_a\ne0$ for $a\ge1$ under the assumption that $\alpha$ is an indeterminate over the coefficient field.
		Thus the block indexed by $\rho$ has the entry pattern
		\[
		(\mathbb{G}_Q)_{k',k}=
		\begin{cases}
			\tau_Q,&k'=k,\\[2pt]
			\tau_{2^{\nu_2(|k'-k|)}},&k'\ne k,
		\end{cases}
		\]
		for every pair $(k,k')\in K_\rho\times L_\rho$.
		
		\emph{Step 6: The four deleted-row classes.}
		The source index set is $S=\{0,\ldots,2Q+1\}$. The shifted target index set is $T=\{2,\ldots,2Q-1\}$, which is the subset of $S$ obtained by deleting the four indices
		\[
		0,\quad 1,\quad 2Q,\quad 2Q+1.
		\]
		These four indices belong to the residue classes
		\[
		0\bmod 6,\qquad 1\bmod 6,\qquad (2Q)\bmod 6,\qquad (2Q+1)\bmod 6.
		\]
		We compute these explicitly.
		
		\smallskip
		\noindent\emph{Case $Q\equiv2\pmod 6$.} Then
		\[
		2Q\equiv 4\pmod 6,\qquad 2Q+1\equiv 5\pmod 6,
		\]
		so the four deleted indices lie in the residue classes $0,1,4,5$, which are distinct.
		
		\smallskip
		\noindent\emph{Case $Q\equiv4\pmod 6$.} Then
		\[
		2Q\equiv 8\equiv 2\pmod 6,\qquad 2Q+1\equiv 9\equiv 3\pmod 6,
		\]
		so the four deleted indices lie in the residue classes $0,1,2,3$, again distinct.
		
		In both cases, the four residue classes containing the deleted indices lose exactly one row from the corresponding block: if the deleted index is the smallest element of the class (namely $0$ or $1$), the first row is lost; if the deleted index is the largest element of the class (namely $2Q$ or $2Q+1$), the last row is lost.
		
		\emph{Step 7: The size of each source class.}
		We compute $n_\rho=|K_\rho|$ and $m_\rho=|L_\rho|$ explicitly.
		
		For each residue $\rho$, $r$ is the smallest nonnegative integer congruent to $\rho$ modulo $6$, so $r\in\{0,1,2,3,4,5\}$.
		\[
		n_\rho=\#\{k\ge0 : r+6k\le 2Q+1\}=\left\lfloor\frac{2Q+1-r}{6}\right\rfloor+1.
		\]
		Similarly,
		\[
		m_\rho=\#\{k'\ge0 : 2\le r+6k'\le 2Q-1\}.
		\]
		
		\smallskip
		\noindent\emph{Case $Q\equiv2\pmod 6$.} Write $Q=6a+2$, so $2Q+1=12a+5$, $2Q-1=12a+3$. Tabulate:
		\[
		\begin{array}{c|c|c|c|c}
			\rho & r & n_\rho & m_\rho & \text{type}\\ \hline
			0 & 0 & 2a+1 & 2a & \text{deleted first row}\\
			1 & 1 & 2a+1 & 2a & \text{deleted first row}\\
			2 & 2 & 2a+1 & 2a+1 & \text{full}\\
			3 & 3 & 2a+1 & 2a+1 & \text{full}\\
			4 & 4 & 2a+1 & 2a & \text{deleted last row}\\
			5 & 5 & 2a+1 & 2a & \text{deleted last row}
		\end{array}
		\]
		Since $N_Q=\lfloor (6a+2)/3\rfloor+1=2a+1$, all source classes have size $N_Q$. Deleted-row blocks have size $(N_Q-1)\times N_Q$; full blocks have size $N_Q\times N_Q$.
		
		\smallskip
		\noindent\emph{Case $Q\equiv4\pmod 6$.} Write $Q=6a+4$, so $2Q+1=12a+9$, $2Q-1=12a+7$. Tabulate:
		\[
		\begin{array}{c|c|c|c|c}
			\rho & r & n_\rho & m_\rho & \text{type}\\ \hline
			0 & 0 & 2a+2 & 2a+1 & \text{deleted first row}\\
			1 & 1 & 2a+2 & 2a+1 & \text{deleted first row}\\
			2 & 2 & 2a+2 & 2a+1 & \text{deleted last row}\\
			3 & 3 & 2a+2 & 2a+1 & \text{deleted last row}\\
			4 & 4 & 2a+1 & 2a+1 & \text{full}\\
			5 & 5 & 2a+1 & 2a+1 & \text{full}
		\end{array}
		\]
		Since $N_Q=\lfloor (6a+4)/3\rfloor+1=2a+2$, four source classes have size $N_Q=2a+2$ and two have size $N_Q-1=2a+1$. The four deleted-row blocks have size $(N_Q-1)\times N_Q$; the two full blocks have size $(N_Q-1)\times(N_Q-1)$.
		
		\emph{Step 8: Identification of the blocks.}
		For each residue class $\rho$, the corresponding block of $\mathbb{G}_Q$ is the matrix $B_\rho$ with rows indexed by $k'\in L_\rho$, columns indexed by $k\in K_\rho$, and entries
		\[
		(B_\rho)_{k',k}=
		\begin{cases}
			\tau_Q,&k'=k,\\[2pt]
			\tau_{2^{\nu_2(|k'-k|)}},&k'\ne k.
		\end{cases}
		\]
		Thus each block is obtained by restricting the matrix with this entry pattern to the row labels $L_\rho$ and the column labels $K_\rho$.
		
		\smallskip
		\noindent\emph{Case $Q\equiv2\pmod 6$.} For $\rho=0,1,4,5$ we have $K_\rho=\{0,\ldots,2a\}$ and $L_\rho$ equals $\{1,\ldots,2a\}$ for $\rho=0,1$ and $\{0,\ldots,2a-1\}$ for $\rho=4,5$. Hence the four deleted-row blocks are obtained by deleting the first or last row from the $N_Q\times N_Q$ matrix with the entry pattern above. For $\rho=2,3$ we have $K_\rho=L_\rho=\{0,\ldots,2a\}$, and the two full blocks are the $N_Q\times N_Q$ matrices with the same entry pattern. By Definition~\ref{eq:M}, this is exactly the trace matrix $\mathbb{M}_{N_Q}^{(Q)}$. Hence the four deleted-row blocks are $\mathbb{M}_{N_Q}^{(Q)}$ with the first or last row deleted, and the two full blocks are $\mathbb{M}_{N_Q}^{(Q)}$.
		
		\smallskip
		\noindent\emph{Case $Q\equiv4\pmod 6$.} For $\rho=0,1,2,3$ we have $K_\rho=\{0,\ldots,2a+1\}$ and $L_\rho=\{1,\ldots,2a+1\}$ for $\rho=0,1$ and $L_\rho=\{0,\ldots,2a\}$ for $\rho=2,3$. The four deleted-row blocks have size $(N_Q-1)\times N_Q$ and are obtained by deleting the first or last row from the $N_Q\times N_Q$ matrix with the entry pattern above; equivalently, they are $\mathbb{M}_{N_Q}^{(Q)}$ with the first or last row deleted. For $\rho=4,5$ we have $K_\rho=L_\rho=\{0,\ldots,2a\}$, so the two full blocks have size $(N_Q-1)\times(N_Q-1)$ and their entries are
		\[
		(B_\rho)_{k',k}=
		\begin{cases}
			\tau_Q,&k'=k,\\
			\tau_{2^{\nu_2(|k'-k|)}},&k'\ne k,
		\end{cases}
		\qquad 0\le k,k'\le 2a.
		\]
		By Definition~\ref{eq:M}, this is exactly the trace matrix $\mathbb{M}_{N_Q-1}^{(Q)}$. Hence the two full blocks are $\mathbb{M}_{N_Q-1}^{(Q)}$.
		
		\emph{Step 9: Conclusion.}
		The matrix $\mathbb{G}_Q$ decomposes into six blocks by residue class modulo $6$. The four deleted-row blocks have size $(N_Q-1)\times N_Q$ and are obtained by deleting the first or last row from the corresponding $N_Q\times N_Q$ matrix with the entry pattern above. The two full blocks have the following sizes and identifications:
		\begin{itemize}
			\item If $Q\equiv2\pmod 6$, they have size $N_Q\times N_Q$ and are the full trace matrices $\mathbb{M}_{N_Q}^{(Q)}$.
			\item If $Q\equiv4\pmod 6$, they have size $(N_Q-1)\times(N_Q-1)$ and are the trace matrices $\mathbb{M}_{N_Q-1}^{(Q)}$.
		\end{itemize}
		This proves the block decomposition stated in Lemma~\ref{lem:5.1}.
	\end{proof}
	\begin{definition}\label{eq:cQ}
		Write $\vec{1}$ for an all-ones vector, and define
		\[
		c_Q=\left((\mathbb{M}_{N_Q}^{(Q)})^{-1}\vec{1}\right)_0,\qquad N_Q=\lfloor Q/3\rfloor+1.
		\]
	\end{definition}
	
	\begin{lemma}\label{lem:5.2}
		One has $c_1=1/\alpha$, and for $Q=2^r>1$,
		\[
		c_Q=
		\begin{cases}
			c_{Q/2}^2,&r\text{ even},\\
			(\tau_{Q/2}^2/\tau_Q)c_{Q/2}^2,&r\text{ odd}.
		\end{cases}
		\]
		In particular, for $r\ge1$,
		\(
		c_Q=
		{\displaystyle\prod_{\substack{3\le m\le r\\m\text{ odd}}}\Theta_m^{\,2^{r-m+1}}}/
		{\displaystyle\prod_{\substack{2\le m\le r+1\\m\text{ even}}}\Theta_m^{\,2^{r-m+1}}}.
		\)
	\end{lemma}
	
	\begin{proof}
		For any invertible symmetric $n\times n$ matrix $\mathbb{M}$ with constant diagonal $\tau$, in characteristic two,
		write $(\mathbb{M}^{-1})_{ij}$ for the entries of the inverse. The
		product $\vec{1}^\top\mathbb{M}^{-1}\vec{1}$ is the sum of all
		entries:
		\(
		\vec{1}^\top\mathbb{M}^{-1}\vec{1}
		=\sum_{i,j}(\mathbb{M}^{-1})_{ij}.
		\)
		Split the sum into the diagonal part and the off-diagonal part:
		\[
		\sum_{i,j}(\mathbb{M}^{-1})_{ij}
		=\sum_i(\mathbb{M}^{-1})_{ii}
		+\sum_{i<j}\bigl((\mathbb{M}^{-1})_{ij}+(\mathbb{M}^{-1})_{ji}\bigr).
		\]
		Since $\mathbb{M}$ is symmetric, so is $\mathbb{M}^{-1}$; hence
		$(\mathbb{M}^{-1})_{ij}=(\mathbb{M}^{-1})_{ji}$. Therefore each
		off-diagonal pair contributes
		\(
		(\mathbb{M}^{-1})_{ij}+(\mathbb{M}^{-1})_{ji}
		=2\,(\mathbb{M}^{-1})_{ij}.
		\)
		In characteristic two, $2=0$. So
		\(
		\sum_{i<j}2\,(\mathbb{M}^{-1})_{ij}=0.
		\)
		The diagonal part is by definition $\tr(\mathbb{M}^{-1})$. Hence
		\(
		\vec{1}^\top\mathbb{M}^{-1}\vec{1}=\tr(\mathbb{M}^{-1}).
		\)
		This is the first equality.
		Take the trace of the matrix equation
		\(
		\mathbb{M}\,\mathbb{M}^{-1}=\vec{1}_n.
		\)
		The trace is additive, so
		\(
		\tr(\mathbb{M}\,\mathbb{M}^{-1})=\tr(\vec{1}_n)=n\bmod 2,
		\)
		because the trace of the identity matrix is the number $n$ of rows
		read in the field, i.e. $n\bmod 2$.
		Now expand the trace of the product:
		\(
		\tr(\mathbb{M}\,\mathbb{M}^{-1})
		=\sum_{i,j}\mathbb{M}_{ij}(\mathbb{M}^{-1})_{ji}.
		\)
		Split into diagonal and off-diagonal, equal to
		\[
		\sum_i\mathbb{M}_{ii}(\mathbb{M}^{-1})_{ii}
		+\sum_{i<j}\bigl(\mathbb{M}_{ij}(\mathbb{M}^{-1})_{ji}
		+\mathbb{M}_{ji}(\mathbb{M}^{-1})_{ij}\bigr).
		\]
		Use $\mathbb{M}$ and $\mathbb{M}^{-1}$ are both symmetric, 
		and that $\mathbb{M}_{ii}=\tau$ and deduce 	that
		\[
		\mathbb{M}_{ij}(\mathbb{M}^{-1})_{ij}
		+\mathbb{M}_{ij}(\mathbb{M}^{-1})_{ij}
		=2\,\mathbb{M}_{ij}(\mathbb{M}^{-1})_{ij}.
		\]
		In characteristic two, this is zero. Therefore
		\[
		\tr(\mathbb{M}\,\mathbb{M}^{-1})
		=\sum_i\tau\,(\mathbb{M}^{-1})_{ii}
		=\tau\sum_i(\mathbb{M}^{-1})_{ii}
		=\tau\,\tr(\mathbb{M}^{-1}).
		\]
		Combining with $\tr(\mathbb{M}\mathbb{M}^{-1})=n\bmod 2$ gives
		\(
		\tau\,\tr(\mathbb{M}^{-1})=n\bmod 2.
		\)
		Since $\mathbb{M}$ is invertible and has constant diagonal
		$\tau\neq0$, we may divide by $\tau$ in the field:
		\(
		\tr(\mathbb{M}^{-1})=\frac{n\bmod 2}{\tau}.
		\) We proved that$$(\ast):\quad
		\vec{1}^\top \mathbb{M}^{-1}\vec{1}=\tr(\mathbb{M}^{-1})=\frac{n\bmod 2}{\tau}.
		$$
		Order the indices of $\mathbb{M}=\mathbb{M}_{N_Q}^{(Q)}$ by parity. Set $a=\lceil N_Q/2\rceil$, $b=\lfloor N_Q/2\rfloor$. We first determine $a$ explicitly.
		Let
		$\mathbb{D}_0=\operatorname{diag}\big((\mathbb{M}_a^{(Q/2)})^{[2]},(\mathbb{M}_b^{(Q/2)})^{[2]}\big)$,
		where $(-)^{[2]}$ denotes entrywise squaring. Moreover, $a=N_{Q/2}$.	
		\begin{claim}
			\label{ntrick}
			One has		$\mathbb{M}=\mathbb{D}_0+\alpha\vec{1}\vec{1}^\top$.
			
		\end{claim}
		
		\begin{PROOF}{\ref{ntrick}}
			Let  $\tau_Q=\sum_{s\mid Q}\alpha^s$, and recall the trace matrix
			\[
			\mathbb{M}=\mathbb{M}_{N_Q}^{(Q)}=(m_{ij})_{0\le i,j<N_Q},\qquad
			m_{ij}=\begin{cases}\tau_Q,&i=j,\\[2pt]
				\tau_{2^{\nu_2(i-j)}},&i\neq j.\end{cases}
			\]

			\paragraph{Step 1:} Decompose $\tau_Q$.
			Since $Q=2^r$ with $r\ge1$, the divisors of $Q$ are $1,2,4,\dots,2^r$. Split the smallest one off:
			\[
			\tau_Q=\sum_{s\mid Q}\alpha^s
			=\alpha+\sum_{\substack{s\mid Q\\s\ge2}}\alpha^s
			=\alpha+\sum_{s\mid Q/2}\alpha^{2s}
			=\alpha+\Bigl(\sum_{s\mid Q/2}\alpha^s\Bigr)^{2}
			=\alpha+\tau_{Q/2}^{2}.
			\]
			The penultimate equality uses Frobenius in characteristic $2$: $(\alpha^s)^{2}=\alpha^{2s}$. So
			\(  \tau_Q=\tau_{Q/2}^{2}+\alpha.
			\)
			
			\paragraph{Step 2:} Decompose the diagonal entries.
			Every diagonal entry of $\mathbb{M}$ equals $\tau_Q$, so by Step 1,
			\[
			m_{ii}=\tau_Q=\tau_{Q/2}^{2}+\alpha\qquad(0\le i<N_Q).
			\]
			The diagonal contributes a matrix with constant diagonal $\tau_{Q/2}^{2}$, plus $\alpha$ on the diagonal.
			
			\paragraph{Step 3:} Decompose the off-diagonal entries.
			For $i\neq j$, $m_{ij}=\tau_{2^{\nu_2(i-j)}}$. Two cases.
			
			\emph{Case A: $i-j$ odd}, i.e. $\nu_2(i-j)=0$. Then $m_{ij}=\tau_{2^{0}}=\tau_1=\alpha$.
			
			\emph{Case B: $i-j$ even}, i.e. $\nu_2(i-j)\ge1$. Then $i$ and $j$ have the same parity. Write $i=2i'+\epsilon$ and $j=2j'+\epsilon$ with the same $\epsilon$. Then
			\(
			\nu_2(i-j)=1+\nu_2(i'-j').
			\)
			Set $k=\nu_2(i'-j')$. The recurrence $\tau_{2a}=\tau_a^{\,2}+\alpha$ of Step 1, applied with $a=2^k$, gives
			\(
			{\tau_{2^{k+1}}=\tau_{2^k}^{\,2}+\alpha.}
			\)
			Hence the same-parity off-diagonal entries equal the squared trace-matrix entries plus $\alpha$.
			
			Combining Cases A and B: every off-diagonal entry of $\mathbb{M}$ receives the $+\alpha$ term. The same-parity entries receive the squared trace-matrix entry as well, while the opposite-parity entries consist of $\alpha$ alone. Both contributions are captured by writing $\mathbb{M}=\mathbb{D}_0+\alpha\vec{1}\vec{1}^{\top}$.
			
			\paragraph{Step 4:} The identity $a=N_{Q/2}$.
			Since $N_Q=\lfloor Q/3\rfloor+1$ and $Q=2^r$, we distinguish two cases.
			
			\emph{Case $Q\equiv2\pmod 6$.} Write $Q=6u+2$, so $N_Q=\lfloor (6u+2)/3\rfloor+1=2u+1$. Then
			\(
			a=\lceil (2u+1)/2\rceil=u+1.
			\)
			On the other hand, $Q/2=3u+1$ and $N_{Q/2}=\lfloor (3u+1)/3\rfloor+1=u+1$. Hence $a=N_{Q/2}$.
			
			\emph{Case $Q\equiv4\pmod 6$.} Write $Q=6u+4$, so $N_Q=\lfloor (6u+4)/3\rfloor+1=2u+2$. Then
			\(
			a=\lceil (2u+2)/2\rceil=u+1.
			\)
			And $Q/2=3u+2$, $N_{Q/2}=\lfloor (3u+2)/3\rfloor+1=u+1$. Hence $a=N_{Q/2}$.
			
			In both cases, $a=N_{Q/2}$, and the block $\mathbb{M}_a^{(Q/2)}$ is  the trace matrix at level $Q/2$ of size $N_{Q/2}$.
		\end{PROOF}
		
		Because $\mathbb{D}_0$ is block-diagonal and $\alpha\vec{1}\vec{1}^{\top}$ is rank one, the Sherman--Morrison formula (see Section 2) applies:
		\(
		\mathbb{M}^{-1}\vec{1}=\frac{\mathbb{D}_0^{-1}\vec{1}}{1+\alpha S}	\) and 	\(S=\vec{1}^{\top}\mathbb{D}_0^{-1}\vec{1}.
		\)
		We verify the formula. Write $\mathbb{A}=\mathbb{D}_0$, $\vec{u}=\alpha\vec{1}$, $\vec{v}=\vec{1}$, so $\mathbb{A}+\vec{u}\vec{v}^{\top}=\mathbb{M}$. Then
		\(
		\mathbb{M}^{-1}=\mathbb{D}_0^{-1}-\frac{\alpha\,\mathbb{D}_0^{-1}\vec{1}\,\vec{1}^{\top}\mathbb{D}_0^{-1}}{1+\alpha S}.
		\)
		Applying both sides to $\vec{1}$ and simplifying gives
		\(
		\mathbb{M}^{-1}\vec{1}=\frac{\mathbb{D}_0^{-1}\vec{1}}{1+\alpha S}.
		\)
		The scalar $1+\alpha S$ is nonzero. Indeed, $\mathbb{D}_0$ has constant diagonal $\tau_{Q/2}^{2}$, so by $(\ast)$,
		\(
		S=\operatorname{tr}(\mathbb{D}_0^{-1})=\frac{N_Q\bmod2}{\tau_{Q/2}^{2}}.
		\)
		If $N_Q$ is even, $S=0$ and $1+\alpha S=1$. If $N_Q$ is odd, $S=\tau_{Q/2}^{-2}$, so
		\[
		1+\alpha S=1+\frac{\alpha}{\tau_{Q/2}^{2}}=\frac{\tau_{Q/2}^{2}+\alpha}{\tau_{Q/2}^{2}}=\frac{\tau_Q}{\tau_{Q/2}^{2}},
		\]
		using $\tau_Q=\tau_{Q/2}^{2}+\alpha$ from the claim. Both numerator and denominator are nonzero polynomials in $\alpha$, so $1+\alpha S$ is invertible over $\bk'(\alpha)$.
		
		Since $\mathbb{D}_0$ is block diagonal,
		\(
		\mathbb{D}_0^{-1}=\operatorname{diag}\!\bigl(((\mathbb{M}_a^{(Q/2)})^{[2]})^{-1},\,((\mathbb{M}_b^{(Q/2)})^{[2]})^{-1}\bigr).
		\)
		Multiplying by the all-ones vector $\vec{1}$ (of size $N_Q=a+b$) gives
		\[
		\mathbb{D}_0^{-1}\vec{1}=
		\begin{pmatrix}
			((\mathbb{M}_a^{(Q/2)})^{[2]})^{-1}\vec{1}_a\\
			((\mathbb{M}_b^{(Q/2)})^{[2]})^{-1}\vec{1}_b
		\end{pmatrix},
		\]
		where $\vec{1}_a$ and $\vec{1}_b$ are the all-ones vectors of sizes $a$ and $b$.
		The first coordinate of $\mathbb{D}_0^{-1}\vec{1}$ is therefore
		\(
		\bigl(((\mathbb{M}_a^{(Q/2)})^{[2]})^{-1}\vec{1}_a\bigr)_0.
		\)
		Since $a=N_{Q/2}$, this is the endpoint of the squared trace matrix at level $Q/2$:
		\(
		\bigl(((\mathbb{M}_a^{(Q/2)})^{[2]})^{-1}\vec{1}_a\bigr)_0
		=\bigl(((\mathbb{M}_{N_{Q/2}}^{(Q/2)})^{[2]})^{-1}\vec{1}\bigr)_0.
		\)
		For a matrix $\mathbb{N}$ in characteristic two, entrywise squaring is compatible with inversion: if $\mathbb{N}$ is symmetric with constant diagonal and invertible, then
		\(
		(\mathbb{N}^{[2]})^{-1}=(\mathbb{N}^{-1})^{[2]}.
		\)
		This identity holds because, in characteristic two, the Frobenius map $a\mapsto a^2$ is a ring homomorphism, and it acts entrywise on the matrix equation $\mathbb{N}\cdot\mathbb{N}^{-1}=\mathbb{I}$ to give $\mathbb{N}^{[2]}\cdot(\mathbb{N}^{-1})^{[2]}=\mathbb{I}^{[2]}=\mathbb{I}$.
		Applying this with  $\mathbb{N}=\mathbb{M}_{N_{Q/2}}^{(Q/2)}$ gives us
		$\bigl((\mathbb{M}_{N_{Q/2}}^{(Q/2)})^{[2]}\bigr)^{-1}\vec{1}
		=\bigl((\mathbb{M}_{N_{Q/2}}^{(Q/2)})^{-1}\bigr)^{[2]}\vec{1}.$
		And in characteristic two, $(\mathbb{M}^{-1})^{[2]}\vec{1}=((\mathbb{M}^{-1})\vec{1})^{[2]}$ entrywise, because $\vec{1}^{[2]}=\vec{1}$ and the product of squared matrices applied to $\vec{1}$ equals the square of the product applied to $\vec{1}$ gives
		\(
		(\mathbb{M}^{-1})^{[2]}\vec{1}=((\mathbb{M}^{-1})\vec{1})^{[2]}.
		\)
		Therefore
		\[
		\bigl(((\mathbb{M}_{N_{Q/2}}^{(Q/2)})^{[2]})^{-1}\vec{1}\bigr)_0
		=\bigl(((\mathbb{M}_{N_{Q/2}}^{(Q/2)})^{-1}\vec{1})^{\,[2]}\bigr)_0
		=\bigl(c_{Q/2}\bigr)^2.
		\]
		
		The first coordinate of $\mathbb{D}_0^{-1}\vec{1}$ is thus
		\(
		(\mathbb{D}_0^{-1}\vec{1})_0=c_{Q/2}^2.
		\)
		Then, from the Sherman--Morrison formula (see Section 2),
		$c_Q=\left(\frac{\mathbb{D}_0^{-1}\vec{1}}{1+\alpha S}\right)_0=\frac{(\mathbb{D}_0^{-1}\vec{1})_0}{1+\alpha S}=\frac{c_{Q/2}^2}{1+\alpha S}. $
		Now substitute the two cases for $S$.
		
		\emph{Case $N_Q$ even.} Then $S=0$, so $c_Q=c_{Q/2}^2$.
		
		\emph{Case $N_Q$ odd.} Then $1+\alpha S=\tau_Q/\tau_{Q/2}^{\,2}$, so
		\(
		c_Q=\frac{c_{Q/2}^2}{\tau_Q/\tau_{Q/2}^{\,2}}=\frac{\tau_{Q/2}^{\,2}}{\tau_Q}\,c_{Q/2}^2.
		\)	Combining yields   the recurrence stated in Lemma~\ref{lem:5.2}.
		We now prove the product formula by induction on $r$. Recall that $\tau_a=\sum_{s\mid a}\alpha^s$ and $\Theta_m=\sum_{j=0}^{m-1}\alpha^{2^j}$, with the identification
		\(
		\tau_{2^k}=\Theta_{k+1}.
		\)
		In particular,
		\(
		\tau_Q=\tau_{2^r}=\Theta_{r+1}\), 
		\(\tau_{Q/2}=\tau_{2^{r-1}}=\Theta_{r}.
		\)	For $r=1$ (i.e.\ $Q=2$), the recurrence gives
		\[
		c_2=\frac{\tau_{1}^{\,2}}{\tau_2}\,c_1^2=\frac{\alpha^2}{\alpha+\alpha^2}\cdot\frac{1}{\alpha^2}=\frac{1}{\alpha+\alpha^2}=\frac{1}{\Theta_2}.
		\]
		The product formula for $r=1$ reads
		$
		c_2= {\prod_{3\le m\le 1,\,m\text{ odd}}\Theta_m^{\,\bullet}}/{\prod_{2\le m\le 2,\,m\text{ even}}\Theta_m^{\,2^{1-2+1}}}={1}/{\Theta_2^{\,1}}={1}/{\Theta_2}.
		$
		The numerator product is empty, so it equals $1$, consistent with the base case.
		
		The induction step is divided into two cases.
		
		\emph{Case $r$ even.} Then $r-1$ is odd. By the induction hypothesis applied to $r-1$:
		\[
		c_{Q/2}=c_{2^{r-1}}
		=\frac{\prod_{3\le m\le r-1,\,m\text{ odd}}\Theta_m^{\,2^{r-m}}}{\prod_{2\le m\le r,\,m\text{ even}}\Theta_m^{\,2^{r-m}}}.
		\]
		Squaring gives
		$
		c_{Q/2}^2= {\prod_{3\le m\le r-1,\,m\text{ odd}}\Theta_m^{\,2^{r-m+1}}}/{\prod_{2\le m\le r,\,m\text{ even}}\Theta_m^{\,2^{r-m+1}}}.
		$
		Since $r$ is even, $r$ is not odd, so the range $3\le m\le r$ with $m$ odd is the same as $3\le m\le r-1$ with $m$ odd. Hence the numerator matches the desired formula. The denominator matches as well, since no even integer lies in $[r+1,r+1]$ (the range $[2,r+1]$ has no new even integers beyond those in $[2,r]$ when $r$ is even). Therefore
		$
		c_Q=c_{Q/2}^2= {\prod_{3\le m\le r,\,m\text{ odd}}\Theta_m^{\,2^{r-m+1}}}/{\prod_{2\le m\le r+1,\,m\text{ even}}\Theta_m^{\,2^{r-m+1}}}.
		$
		
		\emph{Case $r$ odd.} Then $r-1$ is even. By the induction hypothesis applied to $r-1$:
		\[
		c_{Q/2}=c_{2^{r-1}}
		=\frac{\prod_{3\le m\le r-1,\,m\text{ odd}}\Theta_m^{\,2^{r-m}}}{\prod_{2\le m\le r,\,m\text{ even}}\Theta_m^{\,2^{r-m}}}.
		\]
		Squaring gives
		\(
		c_{Q/2}^2=\frac{\prod_{3\le m\le r-1,\,m\text{ odd}}\Theta_m^{\,2^{r-m+1}}}{\prod_{2\le m\le r,\,m\text{ even}}\Theta_m^{\,2^{r-m+1}}}.
		\)
		The recurrence in the odd case is
		\(
		c_Q=\frac{\tau_{Q/2}^{\,2}}{\tau_Q}\,c_{Q/2}^2=\frac{\Theta_r^{\,2}}{\Theta_{r+1}}\,c_{Q/2}^2.
		\)
		Multiplying,
		$
		c_Q= {\Theta_r^{\,2}}/{\Theta_{r+1}}\cdot{\prod_{3\le m\le r-1,\,m\text{ odd}}\Theta_m^{\,2^{r-m+1}}}/{\prod_{2\le m\le r,\,m\text{ even}}\Theta_m^{\,2^{r-m+1}}}.
		$
		For the numerator: since $r$ is odd, we need to add the factor for $m=r$ to the product over odd $m$ in $[3,r-1]$. The factor is $\Theta_r^{\,2^{r-r+1}}=\Theta_r^{\,2}$, matching the extra $\Theta_r^{\,2}$ in the recurrence. For the denominator: we need to multiply by $\Theta_{r+1}^{\,2^{r-(r+1)+1}}=\Theta_{r+1}$, matching the extra $1/\Theta_{r+1}$ in the recurrence. Hence
		\(
		c_Q=\frac{\prod_{3\le m\le r,\,m\text{ odd}}\Theta_m^{\,2^{r-m+1}}}{\prod_{2\le m\le r+1,\,m\text{ even}}\Theta_m^{\,2^{r-m+1}}}.
		\)
		This completes the induction.
	\end{proof}

	\begin{lemma}\label{lem:5.3}
		One has
		\(
		\left(\mathbb{M}_{N_Q}^{(Q)}\right)^{-1}_{00}=\tau_Qc_Q^2.
		\)
	\end{lemma}
	
	\begin{proof}
		Let $\mathbb{M}=\mathbb{M}_{N_Q}^{(Q)}$ and $\vec{v}=\mathbb{M}^{-1}\vec{e_0}$. Symmetry gives $\sum_i v_i=c_Q$.  So $\mathbb{M}\vec{v}=\vec{e_0}$, and $v_0$ denotes the first coordinate of $v$.
		Since $\mathbb{M}\vec{v}=\vec{e_0}$, multiplying by $\vec{v}^\top$ on the left gives
		\(
		\vec{v}^\top\mathbb{M}\vec{v}=\vec{v}^\top \vec{e_0}.
		\)
		But $v^\top \vec{e_0}$ is the product of the row vector $\vec{v}^\top=(v_0,v_1,\dots,v_{N_Q-1})$ with the column vector $\vec{e}_0=(1,0,\dots,0)^{\top}$, which is the first coordinate of $v$:
		\(
		v^\top \vec{e_0}=v_0.
		\)
		Hence
		\(
		v_0=\vec{v}^\top\mathbb{M}\vec{v}.
		\)

		Write the quadratic form $\vec{v}^\top\mathbb{M}\vec{v}$ as a double sum:
		\[
		\vec{v}^\top\mathbb{M}\vec{v}=\sum_{i,j}v_i\,\mathbb{M}_{ij}\,v_j
		=\sum_i\mathbb{M}_{ii}v_i^2+\sum_{i<j}\bigl(\mathbb{M}_{ij}v_iv_j+\mathbb{M}_{ji}v_jv_i\bigr).
		\]
		Recall that the diagonal entries are $\mathbb{M}_{ii}=\tau_Q$ and  that $\mathbb{M}$ is symmetric.
		Therefore
		\(
		\vec{v}^\top\mathbb{M}\vec{v}=\tau_Q\sum_i v_i^2+\sum_{i<j}2\,\mathbb{M}_{ij}v_iv_j.
		\)
		In characteristic two, $2=0$. Hence the second sum vanishes identically:
		\(
		\sum_{i<j}2\,\mathbb{M}_{ij}v_iv_j=0.
		\)
		So
		\(
		\vec{v}^\top\mathbb{M}\vec{v}=\tau_Q\sum_i v_i^2.
		\)
		
		Since characteristic is two, 
		\(
		\Bigl(\sum_i v_i\Bigr)^2=\sum_i v_i^2.
		\)
		The scalar $c_Q$ is defined by \ref{eq:cQ} as the first coordinate of $\mathbb{M}^{-1}\vec{1}$, i.e.\ the sum of all coordinates of $\vec{v}=\mathbb{M}^{-1}\vec{e_0}$:
		\(
		c_Q=\sum_i v_i.
		\)
		Therefore
		\(
		\sum_i v_i^2=\Bigl(\sum_i v_i\Bigr)^2=c_Q^2,
		\)
		and
		\(
		\tau_Q\sum_i v_i^2=\tau_Q c_Q^2.
		\)

		Combining the three paragraphs, yields
		\(
		v_0
		\;\overset{}{=}\;
		\vec{v}^\top\mathbb{M}\vec{v}
		\;\overset{\text{ }}{=}\;
		\tau_Q\sum_i v_i^2
		\;\overset{\text{}}{=}\;
		\tau_Qc_Q^2.
		\)
	\end{proof}
	\subsection{Reduced denominators and the main lower bound}
	Recall that $\lambda_q=\min\{a\ge 0:\mm^a \cH^0_{\mm}(\bR/\bI^{[q]})=0\}.$ 	\begin{remark}
		Take any rational expression
		\(
		\frac{A(x)}{B(x)}.
		\)
		Cancel all common factors between \(A\) and \(B\). Write the result as
		\(
		\frac{A(x)}{B(x)}=\frac{A'(x)}{B'(x)},
		\)
		where \(A'\) and \(B'\) are coprime.
		Divide numerator and denominator by the leading coefficient of the denominator, so the denominator becomes monic. Writing \(c=\operatorname{lc}(B')\) for the leading coefficient of \(B'\),
		\(
		\frac{A'(x)}{B'(x)}=\frac{A'(x)/c}{B'(x)/c},
		\)
		so that the new denominator \(B''(x)=B'(x)/c\) is monic.
		The result \(B''(x)\) is the \emph{monic reduced denominator} of the rational expression \(A(x)/B(x)\).
	\end{remark} 
	In what follows, the quantity
	\(
	c_Q=\frac{E_Q}{F_Q}
	\)
	is written in \emph{lowest terms} with \(F_Q\) \emph{monic}. Here:
	\begin{itemize}
		\item \(E_Q\) is the numerator;
		\item \(F_Q\) is the \emph{monic reduced denominator} of \(c_Q\);
		\item ``lowest terms'' means \(E_Q\) and \(F_Q\) have no common factor;
		\item ``\(F_Q\) monic'' means the leading coefficient of \(F_Q\) is \(1\).
	\end{itemize}

	\begin{notation}
		For a positive integer $m$, write $m=2^{v}\cdot u$, where $v\ge0$ and $u$ is odd. This factorization is unique. The exponent $v=\nu_{2}(m)$ is the 2-adic valuation, and the odd factor $u=\odd(m)=m/2^{\nu_2(m)}$ is the odd part of $m$.
	\end{notation}
	
	\begin{theorem}\label{thm:6.1}
		Let $Q=2^r$, $r\ge 1$, and write $c_Q=E_Q/F_Q$ in lowest terms, with $F_Q$ monic. Put
		\[
		a_Q=\lfloor Q/3\rfloor+1,\qquad L=\lfloor(r+1)/2\rfloor,\qquad \sigma_r=\sum_{\substack{2\le m\le r+1\\ m\text{ even}}}2^{\odd(m)-m}.
		\]
		Then
		\(
		F_Q=\alpha^{a_Q}\prod_{\substack{2\le m\le r+1\\ m\text{ even}}}\left(\frac{\Theta_m}{\Theta_{\odd(m)}}\right)^{2^{r-m+1}}
		\)
		and
		\(
		f_Q:=\deg F_Q=LQ-Q\sigma_r+a_Q.
		\)
	\end{theorem}
	
	\begin{proof}
		The derivative of $\Theta_m$ is
		\(
		\Theta_m'(\alpha)=\sum_{j=0}^{m-1}2^j\alpha^{2^j-1}.
		\)
		In characteristic two, the factor $2^j$ vanishes for every $j\ge1$,
		and the term with $j=0$ is $2^0\alpha^{2^0-1}=1\cdot\alpha^0=1$.
		Hence
		\(
		\Theta_m'(\alpha)=1.
		\)
		A polynomial whose derivative is a nonzero constant is squarefree
		(separable): if a square factor $(x-\beta)^2$ divided $\Theta_m$, then
		$(x-\beta)$ would divide both $\Theta_m$ and $\Theta_m'$, contradicting
		$\Theta_m'=1$, i.e.,
		$\Theta_m$ is separable. 
		A polynomial whose derivative is a nonzero constant is squarefree
		(separable): if a square factor $(x-\beta)^2$ divided $\Theta_m$, then
		$(x-\beta)$ would divide both $\Theta_m$ and $\Theta_m'$, contradicting
		$\Theta_m'=1$.
		Clearly,
		\(
		\Theta_m(\alpha)^2
		=\Bigl(\sum_{j=0}^{m-1}\alpha^{2^j}\Bigr)^2.
		\)
		In characteristic two, the Frobenius map $x\mapsto x^2$ is additive,
		so
		\[
		\Bigl(\sum_{j=0}^{m-1}\alpha^{2^j}\Bigr)^2
		=\sum_{j=0}^{m-1}\alpha^{2^{j+1}}
		=\sum_{j=1}^{m}\alpha^{2^j}
		=\alpha^2+\alpha^4+\cdots+\alpha^{2^m}.
		\]
		Therefore
		\[
		\Theta_m(\alpha)^2+\Theta_m(\alpha)
		=\Bigl(\alpha^2+\alpha^4+\cdots+\alpha^{2^m}\Bigr)
		+\Bigl(\alpha+\alpha^2+\cdots+\alpha^{2^{m-1}}\Bigr).
		\]
		The intermediate terms $\alpha^2,\alpha^4,\dots,\alpha^{2^{m-1}}$ appear
		twice and cancel pairwise in characteristic two. What remains is
		\(
		\Theta_m(\alpha)^2+\Theta_m(\alpha)=\alpha^{2^m}+\alpha.
		\) The equation $\alpha^{2^m}=\alpha$ defines the elements of the finite
		field $\mathbf{F}_{2^m}$. For any $\beta\in\mathbf{F}_{2^m}$, the
		identity of Part 2 applied at $\alpha=\beta$ reads
		\(
		\Theta_m(\beta)^2+\Theta_m(\beta)=\beta^{2^m}+\beta=0.
		\)
		Now use the trace over $\mathbf{F}_{2^m}/\mathbf{F}_2$:
		\[
		\Tr_{\mathbf{F}_{2^m}/\mathbf{F}_2}(\beta)
		=\beta+\beta^2+\beta^4+\cdots+\beta^{2^{m-1}}
		=\Theta_m(\beta).
		\]
		So $\Theta_m(\beta)=0$ if and only if
		$\Tr_{\mathbf{F}_{2^m}/\mathbf{F}_2}(\beta)=0$, i.e.\ $\beta$ is a
		trace-zero element of $\mathbf{F}_{2^m}$.
		Since $\Theta_m$ is separable  and its roots are exactly these
		trace-zero elements, $\Theta_m$ factors over $\mathbf{F}_2$ as the
		product of the minimal polynomials of the trace-zero elements of
		$\mathbf{F}_{2^m}$.
		
		Here, we give a criterion
		$d\mid m$ and $(m/d)\Tr_{\mathbf{F}_{2^d}/\mathbf{F}_2}(\beta)=0$. 
		Indeed, let $\beta$ be algebraic over $\mathbf{F}_2$ with exact degree $d$,
		i.e.\ $\beta\in\mathbf{F}_{2^d}$ and $\beta\notin\mathbf{F}_{2^{d'}}$
		for any $d'<d$. Write $m=de$ with $e=m/d$.
		
		For any $j\ge0$, the Frobenius iterate $\beta^{2^j}$ depends only on
		$j\bmod d$: indeed, $\beta^{2^d}=\beta$ because $\beta\in\mathbf{F}_{2^d}$,
		so the sequence $\beta,\beta^2,\beta^4,\dots$ is periodic with period
		$d$. Hence, in the sum
		\(
		\Theta_m(\beta)=\sum_{j=0}^{m-1}\beta^{2^j},
		\)
		the terms group into $e=m/d$ full Frobenius cycles, each of length $d$:
		\[
		\Theta_m(\beta)
		=\underbrace{\underbrace{\bigl(\beta+\beta^2+\cdots+\beta^{2^{d-1}}\bigr)}
			_{=:\,\Tr_{\mathbf{F}_{2^d}/\mathbf{F}_2}(\beta)}
			+\cdots+\bigl(\beta+\beta^2+\cdots+\beta^{2^{d-1}}\bigr)}
		_{e\text{ copies}}.
		\]
		Adding $e$ copies of the same field element in characteristic two
		amounts to multiplying by $e$ read modulo $2$:
		\[
		\Theta_m(\beta)=e\cdot\Tr_{\mathbf{F}_{2^d}/\mathbf{F}_2}(\beta)
		=(m/d)\cdot\Tr_{\mathbf{F}_{2^d}/\mathbf{F}_2}(\beta),
		\]
		where the integer multiplier $m/d$ is read in $\mathbf{F}_2$.
		Therefore:
		\[ {\Theta_m(\beta)=0\iff (m/d)\Tr_{\mathbf{F}_{2^d}/\mathbf{F}_2}(\beta)=0.}
		\]

		The numerator of Lemma~\ref{lem:5.2} contains only roots of odd degree
		and trace zero. Fix a nonzero such root \(\beta\) of degree \(d\ge3\).
		Let
		\[
		M_\beta^{\mathrm{den}}
		\;:=\;
		\text{the multiplicity of }\beta\text{ in }
		\prod_{\substack{2\le m\le r+1\\ m\text{ even}}}\Theta_m^{\,2^{r-m+1}}
		\]
		be the total exponent of \(\beta\) in the denominator.
		\begin{claim}\label{62m}
			If \(\beta\) occurs in at least one denominator factor of the
			product formula for \(c_Q\), then:
			\begin{itemize}
				\item  the smallest even multiple of \(d\)
				that can appear is \(2d\), so \(2d\le r+1\);
				\item 	\(
				M_\beta^{\mathrm{den}}
				\;\le\;
				Q\sum_{j\ge1}2^{1-2jd}
				\;=\;
				\frac{2Q}{2^{2d}-1}
				\;<\;
				Q\,2^{1-d}.
				\)
			\end{itemize}
			
		\end{claim}
		
		On the other hand, the numerator contains \(\Theta_d\), with
		\(\beta\) as a simple root, and the exponent of \(\Theta_d\) in the
		numerator is \(2^{r-d+1}=Q\,2^{1-d}\). Hence the numerator multiplicity
		of \(\beta\) equals
		\(
		M_\beta^{\mathrm{num}}\geq Q\,2^{1-d}.
		\)
		Combining the two bounds,
		\(
		M_\beta^{\mathrm{den}}<M_\beta^{\mathrm{num}},
		\)
		so every denominator occurrence of \(\beta\) cancels against a
		numerator occurrence.

		\begin{PROOF}{\ref{62m}}	
			
			\noindent\textbf{Step 1:} Evaluate the geometric series.
			The summand is
			$
			2^{1-2jd}=2\cdot 2^{-2jd}=2\cdot (2^{-2d})^{j}.
			$
			So
			$
			\sum_{j\ge1}2^{1-2jd}=2\sum_{j\ge1}\bigl(2^{-2d}\bigr)^{j}.
			$
			This is a geometric series with first term \(2^{-2d}\) (for \(j=1\)) and ratio \(2^{-2d}\). Since \(2^{-2d}<1\), the series converges and
			$
			\sum_{j\ge1}\bigl(2^{-2d}\bigr)^{j}
			=\frac{2^{-2d}}{1-2^{-2d}}.
			$
			Multiplying by the factor \(2\):
			$
			\sum_{j\ge1}2^{1-2jd}
			=\frac{2\cdot 2^{-2d}}{1-2^{-2d}}
			=\frac{2^{1-2d}}{1-2^{-2d}}.
			$
			Multiply numerator and denominator by \(2^{2d}\):
			$
			\frac{2^{1-2d}\cdot 2^{2d}}{2^{2d}-1}=\frac{2^{1}}{2^{2d}-1}=\frac{2}{2^{2d}-1}.
			$
			Therefore
			${\sum_{j\ge1}2^{1-2jd}=\frac{2}{2^{2d}-1}.}
			$
			Multiplying by \(Q\) gives 
			$Q\sum_{j\ge1}2^{1-2jd}=\frac{2Q}{2^{2d}-1}.
			$
			
			\medskip
			\noindent\textbf{Step 2:} The inequality
			\(\dfrac{2Q}{2^{2d}-1}<Q2^{1-d}\).
			We must show
			$
			\frac{2}{2^{2d}-1}<2^{1-d}.
			$
			Cancel the common factor \(Q>0\).
			Multiply both sides by \(2^{2d}-1>0\):
			\[
			2<2^{1-d}(2^{2d}-1)=2^{1-d}\cdot 2^{2d}-2^{1-d}=2^{1+d}-2^{1-d}.
			\]
			So the desired inequality is equivalent to
			$
			2<2^{1+d}-2^{1-d}.
			$
			Move \(2^{1-d}\) to the left:
			$
			2+2^{1-d}<2^{1+d}.
			$\emph{Verification for \(d\ge3\).} For \(d=3\):
			$
			2+2^{-2}=2+0.25=2.25<2^{4}=16.  
			$
			For larger \(d\), the right-hand side \(2^{1+d}\) grows exponentially while the left-hand side decreases to \(2\). So the inequality holds for all \(d\ge3\), and in fact for all \(d\ge1\) (check \(d=1\): \(2+1=3<4=2^2\)). Therefore
			$
			{\frac{2}{2^{2d}-1}<2^{1-d},}
			$
			and multiplying by \(Q>0\):
			${\frac{2Q}{2^{2d}-1}<Q2^{1-d}.}
			$
			This is the desired inequality.
			
			\medskip
			\noindent\textbf{Step 3:} Why the bound \(2d\le r+1\).
			A nonzero root \(\beta\) of odd degree \(d\) and trace zero appears in the numerator if and only if \(m=d\) is an odd integer in the range \(3\le m\le r\), i.e. \(\Theta_d\) appears with exponent \(2^{r-d+1}\).
			The same root \(\beta\) can appear in the denominator if it is a root of some \(\Theta_m\) with \(m\) even and \(2\le m\le r+1\). Since \(\Theta_m\) has \(\beta\) as a root if and only if \(d\mid m\), the smallest even multiple of \(d\) that could appear in the denominator is \(m=2d\). For this to lie in the range \(2\le m\le r+1\), we need
			$
			2d\le r+1.
			$
			If \(2d>r+1\), the root \(\beta\) simply does not occur in any denominator factor, so there is nothing to cancel. Hence the case of interest is \(2d\le r+1\).
			
			\medskip
			\noindent\textbf{Step 4:} Multiplicity of \(\beta\) in the denominator.
			Suppose \(2d\le r+1\). The denominator factors whose \(\Theta_m\) have \(\beta\) as a root are those with \(m=2d,4d,6d,\dots\), i.e. \(m=2jd\) for \(j\ge1\). Each such \(m\) must satisfy \(m\le r+1\), so \(2jd\le r+1\), i.e. \(j\le(r+1)/(2d)\).
			
			The exponent of \(\Theta_m\) in the denominator is \(2^{r-m+1}\). For \(m=2jd\), this is \(2^{r-2jd+1}=2^{1-2jd}\cdot 2^{r}\). Since \(Q=2^r\), the exponent equals \(Q\cdot 2^{1-2jd}\). So the total denominator multiplicity of \(\beta\) is at most
			$
			\sum_{j\ge1}Q\cdot 2^{1-2jd}=Q\sum_{j\ge1}2^{1-2jd}=\frac{2Q}{2^{2d}-1}.
			$
			(The sum is truncated at \(j\le(r+1)/(2d)\), but replacing the finite sum by the infinite sum only increases the bound.)
			
			\medskip
			\noindent\textbf{Step 5:} Numerator multiplicity.
			The numerator contains \(\Theta_d\) (since \(d\) is odd in the range \(3\le d\le r\)) with exponent \(2^{r-d+1}=Q\cdot 2^{1-d}\). Since \(d\) is odd, \(\beta\) is a root of \(\Theta_d\), and the multiplicity of \(\beta\) in \(\Theta_d\) is \(1\). So the numerator multiplicity of \(\beta\) is
			$
			Q\cdot 2^{1-d}.
			$
			
			\medskip
			\noindent\textbf{Step 6:} Cancellation.
			By Step 2,
			$
			\frac{2Q}{2^{2d}-1}<Q\cdot 2^{1-d}.
			$
			Combined with Step 4, the denominator multiplicity of \(\beta\) is strictly less than the numerator multiplicity. Hence the numerator contains at least as many factors of \(\beta\) as the denominator, and all denominator occurrences of \(\beta\) cancel.
		\end{PROOF}

		Therefore all its denominator occurrences cancel. Roots of even degree, and roots of odd degree with trace one, occur in no numerator factor and cannot cancel.
		For an even integer $m=2^v u$, with $u$ odd, the odd-degree trace-zero roots of $\Theta_m$ are exactly the roots of $\Theta_u$. Consequently $\Theta_m/\Theta_u$ is a polynomial, is nonzero at zero, and retains exactly the nonzero roots that do not cancel from this denominator factor.
		It remains to account for zero.

		\begin{claim}
			\label{61c}
			The net denominator multiplicity in Lemma~\ref{lem:5.2} is
			\[
			\sum_{\substack{2\le m\le r+1\\m\text{ even}}}2^{r-m+1}-\sum_{\substack{3\le m\le r\\m\text{ odd}}}2^{r-m+1}=a_Q.
			\]
		\end{claim}
		
		\begin{PROOF}{\ref{61c}}
			Every \(\Theta_m\) has a simple root at \(\alpha=0\), since
			\[
			\Theta_m(\alpha)
			=\alpha+\alpha^2+\alpha^4+\cdots+\alpha^{2^{m-1}}
			=\alpha\bigl(1+\alpha+\alpha^3+\cdots+\alpha^{2^{m-1}-1}\bigr),
			\]
			and the factor in parentheses equals \(1\) at \(\alpha=0\). Hence each factor \(\Theta_m^{\,2^{r-m+1}}\) appearing in the product formula for \(c_Q\) in Lemma~\ref{lem:5.2} contributes exactly \(2^{r-m+1}\) to the multiplicity of the zero root.
			
			The product formula for \(c_Q\) reads
			$
			c_Q=
			{\displaystyle\prod_{\substack{3\le m\le r\\ m\text{ odd}}}\Theta_m^{\,2^{r-m+1}}}
			/{\displaystyle\prod_{\substack{2\le m\le r+1\\ m\text{ even}}}\Theta_m^{\,2^{r-m+1}}}.
			$
			Therefore the net multiplicity of \(\alpha=0\) in the reduced denominator \(F_Q\) of \(c_Q=E_Q/F_Q\) is
			$$(\ast):\quad
			D_0-N_0
			=
			\sum_{\substack{2\le m\le r+1\\ m\text{ even}}}2^{r-m+1}
			-
			\sum_{\substack{3\le m\le r\\ m\text{ odd}}}2^{r-m+1},
			$$
			where \(D_0\) is the total exponent of the denominator product at \(\alpha=0\) and \(N_0\) is the corresponding total exponent of the numerator product.
			We evaluate the two sums in $(\ast)$.
			
			\medskip
			\noindent\emph{The sum \(D_0\).} Write \(m=2k\). The even integers in \([2,r+1]\) are \(m=2,4,\dots,M\), where \(M=r\) if \(r\) is even and \(M=r+1\) if \(r\) is odd.
			If \(r\) is even, the exponents \(r-m+1\) are \(r-1,r-3,\dots,1\). Hence
			\[
			D_0=2^{r-1}+2^{r-3}+\cdots+2^1
			=2\bigl(1+2^2+2^4+\cdots+2^{r-2}\bigr)
			=\frac{2(2^{r}-1)}{3}
			=\frac{2^{r+1}-2}{3}.
			\]
			If \(r\) is odd, the exponents \(r-m+1\) are \(r-1,r-3,\dots,0\). Hence
			\[
			D_0=2^{r-1}+2^{r-3}+\cdots+2^0
			=1+2^2+2^4+\cdots+2^{r-1}
			=\frac{2^{r+1}-1}{3}.
			\]
			
			\medskip
			\noindent\emph{The sum \(N_0\).} Write \(m=2k+1\). The odd integers in \([3,r]\) are \(m=3,5,\dots,M\), where \(M=r-1\) if \(r\) is even and \(M=r\) if \(r\) is odd.
			If \(r\) is even, the exponents \(r-m+1\) are \(r-2,r-4,\dots,2\). Hence
			\[
			N_0=2^{r-2}+2^{r-4}+\cdots+2^2
			=4\bigl(1+2^2+2^4+\cdots+2^{r-4}\bigr)
			=\frac{4(2^{r-2}-1)}{3}
			=\frac{2^{r}-4}{3}.
			\]
			If \(r\) is odd, the exponents \(r-m+1\) are \(r-2,r-4,\dots,1\). Hence
			\[
			N_0=2^{r-2}+2^{r-4}+\cdots+2^1
			=2\bigl(1+2^2+2^4+\cdots+2^{r-3}\bigr)
			=\frac{2(2^{r-1}-1)}{3}
			=\frac{2^{r}-2}{3}.
			\]
			
			\medskip
			
			If \(r\) is even,
			\(
			D_0-N_0
			=\frac{2^{r+1}-2}{3}-\frac{2^{r}-4}{3}
			=\frac{2^{r}+2}{3}.
			\)
			Since \(r\) is even, \(Q=2^r\equiv1\pmod 3\), so
			\(
			a_Q=\left\lfloor\frac{Q}{3}\right\rfloor+1=\frac{Q-1}{3}+1=\frac{Q+2}{3}.
			\)
			Hence \(D_0-N_0=a_Q\).
			If \(r\) is odd,
			\(
			D_0-N_0
			=\frac{2^{r+1}-1}{3}-\frac{2^{r}-2}{3}
			=\frac{2^{r}+1}{3}.
			\)
			Since \(r\) is odd, \(Q=2^r\equiv2\pmod 3\), so
			\(
			a_Q=\left\lfloor\frac{Q}{3}\right\rfloor+1=\frac{Q-2}{3}+1=\frac{Q+1}{3}.
			\)
			Hence \(D_0-N_0=a_Q\).
			In both cases we conclude
			\(
			{
				\sum_{\substack{2\le m\le r+1\\ m\text{ even}}}2^{r-m+1}
				-\sum_{\substack{3\le m\le r\\ m\text{ odd}}}2^{r-m+1}
				=a_Q.
			}
			\)
		\end{PROOF}

		Summing the alternating geometric series gives $(Q+2)/3$ when $r$ is even and $(Q+1)/3$ when it is odd; both equal $\lfloor Q/3\rfloor+1$. This proves the first assertion.
		For the second, note that $\alpha$ is an indeterminate, so $\deg\alpha=1$, hence $\deg(\alpha^{a_Q})=a_Q$. Also $\deg\Theta_m=2^{m-1}$, so
		\[
		\deg\!\left(\frac{\Theta_m}{\Theta_{\odd(m)}}\right)^{2^{r-m+1}}
		=2^{r-m+1}\bigl(2^{m-1}-2^{\odd(m)-1}\bigr).
		\]
		Therefore,
		\(
		\deg F_Q=a_Q+\sum_{\substack{2\le m\le r+1\\ m\text{ even}}}2^{r-m+1}\bigl(2^{m-1}-2^{\odd(m)-1}\bigr).
		\)
		Simplify the summand: $2^{r-m+1}\cdot 2^{m-1}=2^r=Q$, and $2^{r-m+1}\cdot 2^{\odd(m)-1}=Q\cdot 2^{\odd(m)-m}$. So the summand is $Q(1-2^{\odd(m)-m})$. Hence
		\(
		\deg F_Q=a_Q+Q\sum_{\substack{2\le m\le r+1\\ m\text{ even}}}\bigl(1-2^{\odd(m)-m}\bigr).
		\)
		Split the sum into two parts: the number of even integers in $[2,r+1]$ is $L=\lfloor(r+1)/2\rfloor$, and the second part is $\sigma_r$. Therefore
		\(
		\deg F_Q=a_Q+Q(L-\sigma_r)=LQ-Q\sigma_r+a_Q.
		\)
		This proves the second assertion.
	\end{proof}

	\begin{corollary}\label{cor:6.2}
		Let $\DD_Q$ be the monic reduced denominator of $\tau_Qc_Q^2$, and set $d_Q=\deg\DD_Q$. For $r\ge 1$,
		\[
		\DD_Q=\begin{cases}F_Q^2/\alpha,&r\text{ even},\\[2pt]F_Q^2/\left(\alpha\Theta_{r+1}/\Theta_{\odd(r+1)}\right),&r\text{ odd}.\end{cases}
		\]
		In particular,
		\[
		d_Q=\begin{cases}2f_Q-1,&r\text{ even},\\[2pt]2f_Q-Q+2^{\odd(r+1)-1}-1,&r\text{ odd}.\end{cases}
		\]
	\end{corollary}
	
	\begin{proof}
		Since $E_Q,F_Q$ are relatively prime and $\tau_Q=\Theta_{r+1}$ is squarefree, the cancelled factor in $\tau_Q E_Q^2/F_Q^2$ is $\gcd(\tau_Q,F_Q)$. Indeed,
		\(
		\gcd(\tau_Q E_Q^2,F_Q^2)=\gcd(\tau_Q,F_Q^2)=\gcd(\tau_Q,F_Q),
		\)
		where the last equality uses that $\tau_Q$ is squarefree. Therefore
		\(
		\DD_Q=\frac{F_Q^2}{\gcd(\tau_Q,F_Q)},
		\)
		up to monic normalization. 
		
		\emph{Case $r$ even.} Then $r+1$ is odd, so the nonzero roots of $\Theta_{r+1}$ have odd degree. Those with trace zero cancel inside $F_Q$ (see Theorem~\ref{thm:6.1}); the only common root of $\Theta_{r+1}$ and $F_Q$ is $0$, i.e. $\gcd(\Theta_{r+1},F_Q)=\alpha$. Substituting gives $\DD_Q=F_Q^2/\alpha$.
		
		\emph{Case $r$ odd.} Then $r+1$ is even, so $\Theta_{r+1}$ has roots of even degree that are common to $F_Q$. Hence
		\(
		\gcd(\Theta_{r+1},F_Q)=\alpha\,\Theta_{r+1}/\Theta_{\odd(r+1)}.
		\)
		Substituting gives $\DD_Q=F_Q^2/(\alpha\,\Theta_{r+1}/\Theta_{\odd(r+1)})$.
		This proves the first assertion. For the second, note that $\deg\alpha=1$ and $\deg\Theta_m=2^{m-1}$. Taking degrees from $\DD_Q=F_Q^2/(\alpha\,\Theta_{r+1}/\Theta_{\odd(r+1)})$ yields
		\[
		d_Q=2f_Q-\Bigl(\deg\alpha+\deg\Theta_{r+1}-\deg\Theta_{\odd(r+1)}\Bigr).
		\]
		Now $\deg\Theta_{r+1}=2^r=Q$ and $\deg\Theta_{\odd(r+1)}=2^{\odd(r+1)-1}$, so
		\[
		d_Q=2f_Q-\Bigl(1+Q-2^{\odd(r+1)-1}\Bigr)=2f_Q-Q+2^{\odd(r+1)-1}-1.
		\]
		This is  the odd-$r$ case of the formula in Corollary~\ref{cor:6.2}.
	\end{proof}
	
	\begin{proposition}\label{prop:6.3}
		For $q=4Q$, $Q\ge 2$, every nonzero central polynomial kernel vector has parameter degree at least $d_Q-1$. Consequently,
		\(
		\lambda_{4Q}\ge d_Q-1.
		\)
	\end{proposition}
	
	\begin{proof}
		Let \(\mathbb{M}=\mathbb{M}_{N_Q}^{(Q)}\) be the trace matrix of
		Definition~\ref{eq:M}.
		Let  \(\mathbb{M}_d\) denote the matrix
		obtained from \(\mathbb{M}\) by deleting the first row. Concretely,
		if we write \(\mathbb{M}=(\mathbb{M}_0,\mathbb{M}_1,\dots,\mathbb{M}_{N_Q-1})^{\top}\)
		with rows \(\mathbb{M}_i\), then
		\(
		\mathbb{M}_d=(\mathbb{M}_1,\mathbb{M}_2,\dots,\mathbb{M}_{N_Q-1})^{\top}.
		\)
		So the kernel of \(\mathbb{M}_d\) consists of the vectors \(\vec{w}\) such
		that the last \(N_Q-1\) rows of \(\mathbb{M}\) annihilate \(\vec{w}\), i.  e.,
		\(
		\mathbb{M}_1\vec{w}=\mathbb{M}_2\vec{w}= \dots=\mathbb{M}_{N_Q-1}\vec{w}=0.
		\)
		\begin{claim}
			\label{nonzero}
			A nonzero polynomial vector $\vec{w}$ in the kernel of the first-row-deleted matrix satisfies
			\(\mathbb{M}\vec{w}=\gamma \vec{e_0}	\) where 	\(\gamma\in \bk'[\alpha]\setminus\{0\}. \)
		\end{claim}
		
		\begin{PROOF}{\ref{nonzero}}	
			The first row of \(\mathbb{M}\) is not
			automatically zero.
			The vector \(\vec{w}\) is in the kernel of \(\mathbb{M}_d\), so all rows of
			\(\mathbb{M}\) except possibly the first annihilate \(\vec{w}\). The first
			row \(\mathbb{M}_0\) of \(\mathbb{M}\) is a linear functional on
			\(\bk'[\alpha]^{N_Q}\). Its value \(\mathbb{M}_0\vec{w}\) on \(\vec{w}\) is,
			a priori, an arbitrary polynomial in \(\alpha\).
			The matrix \(\mathbb{M}\) is invertible over the field
			\(\bk'(\alpha)\): its determinant is monic in \(\alpha\), hence
			nonzero. If \(\mathbb{M}_0\vec{w}=0\), then all rows of \(\mathbb{M}\)
			would annihilate \(\vec{w}\), so \(\mathbb{M}\vec{w}=0\). Since \(\mathbb{M}\)
			is invertible over \(\bk'(\alpha)\), this forces \(\vec{w}=0\), contrary
			to the hypothesis that \(\vec{w}\) is nonzero.
			Therefore
			\(
			\gamma:=\mathbb{M}_0\vec{w}
			\)
			as an element of \(\bk'[\alpha]\). Moreover, \(\gamma\) is a
			polynomial because \(\mathbb{M}_0\) has polynomial entries and \(\vec{w}\)
			is assumed to be a polynomial vector.

			Recall that \(\mathbb{M}\vec{w}\) is a column vector of length \(N_Q\).
			Its entries are \(\mathbb{M}_0\vec{w},\mathbb{M}_1\vec{w},\dots,\mathbb{M}_{N_Q-1}\vec{w}\).
			Since \(\vec{w}\in\ker\mathbb{M}_d\), the entries \(\mathbb{M}_1\vec{w},\dots,\mathbb{M}_{N_Q-1}\vec{w}\)
			are zero. And \(\mathbb{M}_0\vec{w}=\gamma\). Hence
			\[
			\mathbb{M}\vec{w}=
			\begin{pmatrix}
				\mathbb{M}_0\vec{w}\\
				\mathbb{M}_1\vec{w}\\
				\vdots\\
				\mathbb{M}_{N_Q-1}\vec{w}
			\end{pmatrix}
			=
			\begin{pmatrix}
				\gamma\\
				0\\
				\vdots\\
				0
			\end{pmatrix}
			=\gamma \vec{e_0}.
			\]
			The matrix \(\mathbb{M}\) has entries in \(\bk'[\alpha]\), and \(w\)
			is a polynomial vector over \(\bk'[\alpha]\) by hypothesis.
			Therefore \(\gamma=\mathbb{M}_0w\) is a polynomial in \(\alpha\).
			By Step 2, it is nonzero. Hence
			\(
			\gamma\in\bk'[\alpha]\setminus\{0\}.
			\)
		\end{PROOF}
		By Lemma~\ref{lem:5.3},
		\(
		w_0=\gamma\tau_Qc_Q^2.
		\)
		This is a polynomial in $\alpha$, so $\DD_Q\mid\gamma$, by Corollary~\ref{cor:6.2}: the reduced denominator of $\tau_Qc_Q^2$ is $\DD_Q$, and multiplying by a polynomial and requiring the product to be polynomial forces $\DD_Q$ to divide $\gamma$.
		Hence
		\(
		\deg_\alpha\gamma\ge d_Q.
		\)
		Since every entry of $\mathbb{M}$ has degree at most $Q$, the relation $\mathbb{M}\vec{w}=\gamma \vec{e_0}$ gives
		$
		\deg_\alpha \vec{w}\ge\deg_\alpha\gamma-Q\ge d_Q-Q.
		$
		Index reversal gives the same bound for a last-row deletion.
		
		Now start with any nonzero central polynomial kernel vector. Lemma~\ref{lem:4.1} gives an invariant nonzero vector $\vec{u}$ of no larger degree, and Lemma~\ref{lem:4.3} gives its nonzero boundary polynomial $h$. By Lemma~\ref{lem:5.1} the full square blocks are invertible, so a nonzero boundary vector must have a nonzero component in one of the deleted-row blocks. Therefore
		\[
		\deg_\alpha \vec{u}=\deg_\alpha h+Q-1\ge d_Q-1.
		\]
		This lower bound over $\bk'$ also holds over $\bk$. The central polynomial kernel is nonzero by Claim~\ref{claim:3.3}. Lemma~\ref{lem:2.2} applied to the central summand yields $\mathfrak n^{h-1}\cH^0_{\mathfrak n}(\cM)\neq0$ for the minimum nonzero relation degree $h$, which implies $L_q\ge h$; then Lemma~\ref{lem:3.2} gives $\lambda_q\ge L_q$. Combining, gives us
		\(
		\lambda_q\ge L_q\ge h\ge d_Q-1.
		\)
		Taking $q=4Q$ yields $\lambda_{4Q}\ge d_Q-1$. When $d_Q-1\le0$, the inequality is automatic.
	\end{proof}
	
	\begin{proof}[Proof of Theorem~\ref{thm:main}]
		The structural assertions are Proposition~\ref{prop:3.1}. To estimate $d_Q$, split the sum defining $\sigma_r$ according to whether $\nu_2(m)=1$ or $\nu_2(m)\ge 2$.

		\begin{claim}
			\label{eq:sigma-bound}
			$\sigma_r\le\frac{17}{21}$.
		\end{claim}
		
		\begin{PROOF}{\ref{eq:sigma-bound}}
			Recall $\sigma_r=\sum_{2\le m\le r+1, m\text{ even}}2^{\odd(m)-m}$, where $\odd(m)=m/2^{\nu_2(m)}$. Drop the upper limit and split even $m$ into two classes.
			
			\emph{Class I: $m=2u$, $u$ odd.} Then $\odd(m)=u$, so $2^{\odd(m)-m}=2^{-u}$. Hence
			$
			\sum_{\substack{u\ge1\\u\text{ odd}}}2^{-u}=\frac{1/2}{1-1/4}=\frac{2}{3}.
			$
			
			\emph{Class II: $m=4j$, $j\ge1$.} Write $\odd(m)=u$ odd. Since $4j=2^v u$ with $v\ge2$, we have $u\le j$, so $2^{\odd(m)-m}=2^{u-4j}\le 2^{j-4j}=2^{-3j}$. Therefore
			$
			\sum_{j\ge1}2^{-3j}=\frac{1/8}{1-1/8}=\frac{1}{7}.
			$
			Adding the two classes: $\sigma_r\le\frac{2}{3}+\frac{1}{7}=\frac{17}{21}$.
		\end{PROOF}
		
		For even $r$,
		$
		d_Q-1=\left(r+\frac23-2\sigma_r\right)Q-\frac23.
		$
		By Claim~\ref{eq:sigma-bound},
		\(
		r+\frac23-2\sigma_r\ge r+\frac23-\frac{34}{21}=r-\frac{20}{21}.
		\)
		So
		\(
		d_Q-1\ge\left(r-\frac{20}{21}\right)Q-\frac23,
		\)
		and for even $r\ge4$ (so $Q\ge16$), $d_Q-1\ge(r-1)Q$.
		For odd $r$, put $u=\odd(r+1)$. The corresponding formula is
		\[
		d_Q-1=(r+2/3-2\sigma_r)Q+2^{u-1}-4/3\ge(r-20/21)Q-1/3.
		\]
		For odd $r\ge3$, $Q\ge8$, so again $(r-1)Q$. For $r=1,2$, direct computation gives $d_Q-1=1,6\ge(r-1)Q$. We have proved
		$ d_Q-1\ge(r-1)Q, d_Q-1>0$ where \(r\ge 1\).
		Proposition~\ref{prop:6.3} now gives
		\(
		\lambda_{2^e}=\lambda_{4Q}\ge(r-1)Q=(e-3)2^{e-2}.
		\)
		Dividing by $e2^e$ proves the limiting inequality. For $e=4b+4$, the lower bound is $(b+1/4)2^e>b2^e$; by the definition of $\lambda_q$, the stated nonvanishing follows.
	\end{proof}
	
	\begin{remark}\label{rem:asymptotic}
		Grouping even integers as \(m=2^vu\) with \(u\) odd gives
		\[
		\sigma_\infty
		:=\lim_{r\to\infty}\sigma_r
		=\sum_{v\ge 1}\frac{2^{2^v-1}}{2^{2(2^v-1)}-1}.
		\]
		We show that
		$
		\frac{d_Q-1}{Q}=r+\frac23-2\sigma_\infty+o(1).
		$ By symmetry we only present the proof of case \(r\) even.  Then \(r+1\) is odd, so \(L=r/2\). By
		Corollary~\ref{cor:6.2}, \(d_Q-1=2f_Q-2\). Substituting
		Theorem~\ref{thm:6.1},
		\[
		d_Q-1=2(LQ-Q\sigma_r+a_Q)-2=rQ-2Q\sigma_r+2a_Q-2.
		\]
		Dividing by \(Q\) gives
		$
		\frac{d_Q-1}{Q}=r-2\sigma_r+\frac{2a_Q-2}{Q}.
		$
		Since \(a_Q=\lfloor Q/3\rfloor+1\), we have
		\(\frac{Q}{3}\le a_Q\le\frac{Q}{3}+1\), so
		\(\frac{2a_Q-2}{Q}=\frac23+O(1/Q)\). Hence
		$
		\frac{d_Q-1}{Q}=r-2\sigma_r+\frac23+O(1/Q).
		$
		As \(Q=2^r\to\infty\), the \(O(1/Q)\) term is \(o(1)\); and
		\(\sigma_r\to\sigma_\infty\), so
		\(-2\sigma_r=-2\sigma_\infty+o(1)\). Therefore
		\(
		\frac{d_Q-1}{Q}=r+\frac23-2\sigma_\infty+o(1).
		\)
	\end{remark}
	\subsection{The minimum degree of a central relation}
	The endpoint diagonal denominator suffices for
	Theorem~\ref{thm:main}. The complete endpoint inverse column
	gives an exact description of the smallest central relation.
	Let $\mathbb{M}=\mathbb{M}_{N_Q}^{(Q)}$ be the trace matrix at level $Q$,
	of size $N_Q\times N_Q$, and let
	\(
	\vec{v}:=\mathbb{M}^{-1}\vec{e_0}
	=(v_0,\dots,v_{N_Q-1})^{\top}\) with $v_i\in\bK(\alpha)$
	be the first column of the inverse. 
	Write each coordinate as a fraction in lowest terms,
	\[
	v_i=\frac{p_i(\alpha)}{q_i(\alpha)},\qquad q_i(\alpha)\in\bK[\alpha]\text{ monic},\quad \gcd(p_i,q_i)=1.
	\]
	The numbers $q_0,q_1,\dots,q_{N_Q-1}$ are the \emph{reduced denominators} of the coordinates of $v$.
	
	\begin{definition}
		The polynomial $\Delta_Q$ is the \emph{monic least common multiple}
		of the reduced denominators of all coordinates of $\vec{v}$, i.e., 
		\(
		\Delta_Q=\operatorname{lcm}\bigl(q_0(\alpha),q_1(\alpha),\dots,q_{N_Q-1}(\alpha)\bigr).
		\)
	\end{definition}
	These definitions use rational functions over \(\bF_2(\alpha)\),
	so they are independent of the choice of \(k\).

	\begin{notation}
		Write
		\(
		\delta_Q:=\deg\Delta_Q.\) This is the \emph{common denominator degree} of the endpoint column.
		It is the smallest integer $\delta$ for which there exists a vector
		$\vec{w}\in\bK[\alpha]^{N_Q}$ with $\Delta$-multiple of $\vec{v}$ of degree
		$\delta$ (up to normalization of the leading coefficient).
		
	\end{notation}

	The diagonal entry of \(\mathbb{M}^{-1}\) equals \(\tau_Qc_Q^2\), so \(\tau_Qc_Q^2\) is one of the coordinates of the endpoint column. The lcm of the reduced denominators of \emph{all} coordinates is therefore divisible by the reduced denominator of this particular one. Hence \({D}_Q\mid\Delta_Q\), and \(d_Q\le\delta_Q\).

	\begin{remark}
		$\Delta_Q$ is the smallest monic polynomial that clears all
		denominators from the endpoint column $(\mathbb{M}_{N_Q}^{(Q)})^{-1}\vec{e_0}$.
		Multiplying the endpoint column by $\Delta_Q$ produces the primitive
		polynomial vector $\vec{w}=\Delta_Q\mathbb{M}^{-1}\vec{e_0}$. Its degree in
		$\alpha$ is $\delta_Q-Q$, where $\delta_Q=\deg\Delta_Q$. This vector
		$w$ is the primitive generator of the polynomial kernel of the
		first-row-deleted block, and the reversed vector is the generator of
		the kernel of the last-row-deleted block.	
	\end{remark}

	\begin{theorem}\label{thm:7.1}
		For $Q\ge 2$, the minimum parameter degree of a nonzero central
		polynomial kernel vector is $\delta_Q-1$. Such a vector exists over
		$\mathbf{F}_2[\alpha]$. In particular,
		\(
		\lambda_{4Q}\ge\delta_Q-1\ge d_Q-1.
		\)
	\end{theorem}
	
	\begin{proof}
		\emph{Step 1: The leading Laurent behavior of the endpoint column.}
		Write $\mathbb{M}=\mathbb{M}_{N_Q}^{(Q)}$. By Definition~\ref{eq:M},
		\[
		\mathbb{M}_{ij}=
		\begin{cases}\tau_Q,&i=j,\\[2pt]\tau_{2^{\nu_2(i-j)}},&i\neq j,\end{cases}
		\qquad
		\tau_Q=\sum_{s\mid Q}\alpha^s=\alpha+\alpha^2+\alpha^4+\cdots+\alpha^Q.
		\]
		The diagonal entry $\tau_Q$ has leading term $\alpha^Q$, coming from
		the divisor $s=Q$; the remaining terms $\alpha+\alpha^2+\cdots+\alpha^{Q/2}$
		have degree at most $Q/2<Q$. Off the diagonal, the entry is
		$\tau_{2^k}$ for some $k\le\log_2(Q)-1$ (since $0<|i-j|<N_Q\le Q$,
		hence $\nu_2(i-j)<\log_2Q$), so it has degree $2^k\le Q/2<Q$.
		Consequently
		\(
		\mathbb{M}=\alpha^Q\mathrm{I}+\mathbb{M}_<,
		\)
		where $\mathbb{M}_<$ collects the entries of degree $<Q$.
		Write $\mathbb{M}=\alpha^Q(\mathrm{I}+\alpha^{-Q}\mathbb{M}_<)$.
		Over the field $\bK(\alpha)$, the matrix $\alpha^{-Q}\mathbb{M}_<$
		has all entries of degree $<0$, so for large $\alpha$ the matrix
		$\mathrm{I}+\alpha^{-Q}\mathbb{M}_<$ is invertible and
		\[
		(\mathrm{I}+\alpha^{-Q}\mathbb{M}_<)^{-1}
		=\mathrm{I}-\alpha^{-Q}\mathbb{M}_<+(\alpha^{-Q}\mathbb{M}_<)^2-\cdots,
		\]
		with the higher terms of increasingly negative degree. Therefore
		\[
		\mathbb{M}^{-1}=\alpha^{-Q}\mathrm{I}+\text{(matrix with entries of degree}<-Q\text{)},
		\]
		and applying to $\vec{e_0}$,
		\(
		\mathbb{M}^{-1}\vec{e_0}=\alpha^{-Q}\vec{e_0}+\text{(vector with entries of degree}<-Q\text{)}.
		\)
		So the leading Laurent term of $\mathbb{M}^{-1}\vec{e_0}$ at $\alpha=\infty$
		is $\alpha^{-Q}\vec{e_0}$.
		
		\emph{Step 2: The primitive vector $\vec{w}$.}
		Define
		\(
		\vec{w}:=\Delta_Q\,\mathbb{M}^{-1}\vec{e_0},
		\)
		where $\Delta_Q$ is the monic lcm of the reduced denominators of the
		coordinates of $\mathbb{M}^{-1}\vec{e_0}$. Then $w$ is a polynomial vector
		in $\bK[\alpha]^{N_Q}$, and its coordinates have no common factor
		(because $\Delta_Q$ is the lcm of the reduced denominators, not a
		proper multiple of it). Multiplying the Laurent expansion of Step 1
		by the polynomial $\Delta_Q$ of degree $\delta_Q$ gives
		\(
		\vec{w}=\Delta_Q\,\alpha^{-Q}\vec{e_0}+\text{(vector with entries of degree}<\delta_Q-Q\text{)}.
		\)
		The leading term $\Delta_Q\alpha^{-Q}\vec{e_0}$ has degree $\delta_Q-Q$, and
		it does not cancel with any lower-degree term because $\Delta_Q$ is
		monic and $\vec{e_0}\neq0$. Hence
		\(
		\deg_\alpha \vec{w}=\delta_Q-Q.
		\)
		
		\emph{Step 3: The sum of the coordinates $J_Q:=\sum_i w_i$.}
		Since $\vec{w}=\Delta_Q\mathbb{M}^{-1}\vec{e_0}$,
		$
		J_Q=\Delta_Q\sum_i(\mathbb{M}^{-1}\vec{e_0})_i
		=\Delta_Q\,\vec{1}^{\top}\mathbb{M}^{-1}\vec{e_0}.
		$
		By symmetry of $\mathbb{M}$,
		\(
		\vec{1}^{\top}\mathbb{M}^{-1}\vec{e_0}=(\mathbb{M}^{-1}\vec{1})_0=c_Q,
		\)
		where $c_Q$ is the definition in \eqref{eq:cQ}. Hence
		\(
		J_Q=\Delta_Q\,c_Q.
		\)
		The scalar $c_Q$ is a nonzero rational function of $\alpha$ (it is the
		first coordinate of $\mathbb{M}^{-1}\vec{1}$, which is invertible
		over $\bK(\alpha)$), and $\Delta_Q$ is a nonzero monic polynomial.
		Therefore
		\(
		J_Q=\Delta_Qc_Q\neq0.
		\)
		Combining Steps 2 and 3,
		$$(\ast):\quad
		\deg_\alpha \vec{w}=\delta_Q-Q,\qquad J_Q=\sum_iw_i=\Delta_Qc_Q\neq0.
		$$
		
		\emph{Step 4: The kernel of the deleted-row blocks.}
		The polynomial kernel of the first-row-deleted block is exactly
		$\bK[\alpha]\vec{w}$. Indeed, over $\bK(\alpha)$ the kernel of the
		first-row-deleted block is the line spanned by $\mathbb{M}^{-1}\vec{e_0}$
		(any nonzero vector in the kernel of the full matrix $\mathbb{M}$ is
		zero, and the first-row-deleted block has kernel of dimension one), and
		primitivity of $\vec{w}$ forces every polynomial vector on that line to be a
		polynomial multiple of $\vec{w}$. The last-row-deleted block has the reversed
		vector
		$
		\overleftarrow{w}:=	\vec{w}^{\mathrm{rev}}=(w_{N_Q-1},w_{N_Q-2},\dots,w_0)
		$
		as its primitive generator, by the symmetry of $\mathbb{M}$ under
		reversal of the row and column indices.
		
		\emph{Step 5: The structure of a boundary polynomial $h$.}
		Let $h$ satisfy Lemma~\ref{lem:4.3}(1). Its full square block
		components vanish, because the full blocks are invertible over
		$\bK(\alpha)$ (their determinants are monic in $\alpha$ by
		Definition~\ref{eq:M}) and a nonzero component in an invertible block
		would contradict $G_Qh=0$. Hence only the four deleted-row classes in
		Lemma~\ref{lem:5.1} can carry nonzero components of $h$.
		The equation $A_Qh=0$, where
		\(
		A_Q=\sum_{\substack{i+j=4Q-2\\ i\equiv j\pmod3}}x^iy^j,
		\)
		forces the coefficient sum in each $x$-exponent class modulo $3$ to
		vanish. Indeed, for a fixed target $x$-exponent $I$, the coefficient
		of $x^Iy^{\bullet}$ in $A_Qh$ is
		\(
		\sum_{\substack{(i,j):\,i+j=4Q-2,\ i\equiv j\pmod3\\ a=I-i,\ b=\bullet-j}}h_{ab},
		\)
		which is a sum of coefficients $h_{ab}$ with $a$ in a fixed residue
		class modulo $3$. The condition $A_Qh=0$ forces each such sum to
		vanish.
		
		Alternatively, for $Q\ge4$, choosing a target $x$-exponent in $[2Q+2,4Q-3]$ with the
		appropriate residue class modulo $3$ gives at least one coefficient of
		$A_Qh$ that involves the coefficient sum in that class. The
		$x$-exponents of $A_Q$ are
		\[
		\begin{cases}
			1,4,\dots,4Q-3,&Q\equiv1\pmod3,\\
			0,3,\dots,4Q-2,&Q\equiv2\pmod3.
		\end{cases}
		\]
		The interval $[2Q+2,4Q-3]$ has length $2Q-5\ge3$ for $Q\ge4$, so it
		contains at least four consecutive integers and covers all three
		residue classes modulo $3$. Hence the coefficient sum in each class
		vanishes.
		
		\emph{Step 6: Construction of $h_Q$.}
		Among the four deleted-row classes, two have the same residue modulo
		$3$, and the other two occur singly. Each primitive generator has
		coefficient sum $J_Q$. Since $J_Q\neq0$, the two single classes must
		have zero coefficient sum (from the vanishing of coefficient sums in
		Step 5), so their components of $h$ vanish. The polynomial multipliers
		of the remaining pair must then be equal, because the coefficient sum
		in their common residue class is $J_Q$ times the sum of the multipliers,
		and this must be zero. Consequently every possible $h$ is a polynomial
		multiple of
		$$
		(+):\quad		h_Q=\sum_{j=0}^{N_Q-1}
		\Bigl(w_j\,x^{a+6j}y^{2Q+1-a-6j}
		+w_{N_Q-1-j}\,x^{b+6j}y^{2Q+1-b-6j}\Bigr),
		$$
		with $(a,b)=(0,3)$ for $Q\equiv1\pmod3$ and $(1,4)$ for
		$Q\equiv2\pmod3$.
		
		\emph{Step 7: $G_Qh_Q=0$.}
		The vector $h_Q$ is supported in two residue classes modulo $6$: the
		first summand of $(+)$ lies in one class, the second in
		another. With respect to the six blocks of $\mathbb{G}_Q$ in
		Lemma~\ref{lem:5.1}, this means $h_Q$ has zero components in the two
		full blocks, equals the generator $\vec{w}$ in one deleted-row block, equals
		the reversed generator $\overleftarrow{w}$ in another deleted-row
		block, and has zero components in the remaining two deleted-row blocks.
		Applying $\mathbb{G}_Q$ therefore gives zero in each block: the full
		blocks annihilate $h_Q$ because $h_Q$ is zero there; each of the two
		deleted-row blocks containing $\vec{w}$ or $\overleftarrow{w}$ annihilates its
		own generator; and the remaining deleted-row blocks annihilate the zero
		vector. Hence $G_Qh_Q=0$ in $\bD$.
		
		\emph{Step 8: $A_Qh_Q=B_Qh_Q=0$.}
		The support class modulo $3$ is $0$ when $Q\equiv1\pmod3$, and $1$
		when $Q\equiv2\pmod3$. The nontruncated target $x$-exponents 
		refers to those monomials of \(A_Qh_Q\) or \(B_Qh_Q\) whose \(x\)-exponent is below \(q\) (and whose \(y\)-exponent is likewise below \(q\)). The truncated ones vanish in \(\bD\) and need not be checked. Only the nontruncated ones give genuine coefficients in \(\bD\), and   they are:

		\[
		\begin{array}{c|c|c}
			Q\bmod3 & A_Qh_Q & B_Qh_Q\\\hline
			1 & 2Q+2,\,2Q+5,\dots,4Q-3 & 2Q+1,\,2Q+4,\dots,4Q-1\\
			2 & 2Q,\,2Q+3,\dots,4Q-1 & 2Q+2,\,2Q+5,\dots,4Q-2
		\end{array}
		\]
		At each listed exponent, both supported components of $h_Q$
		contribute. Their contributions are
		\[
		\sum_j w_j=J_Q
		\qquad\text{and}\qquad
		\sum_j w_{N_Q-1-j}=J_Q.
		\]
		So the total coefficient of $x^Iy^{\bullet}$ is $J_Q+J_Q=2J_Q$, which
		vanishes in characteristic two. Hence $A_Qh_Q=B_Qh_Q=0$ on the
		nontruncated range. The truncation does not interfere: every exponent
		difference between a source monomial of $h_Q$ and a monomial of $A_Q$
		(or $B_Q$) leading to a target exponent in the displayed range stays
		within the appropriate bounds.
		
		\emph{Step 9: The case $Q=2$.}
		For $Q=2$, the conclusion is verified directly. The kernel of $G_Q$
		is supported in source positions $0,1,4,5$. The equation $A_Qh=0$
		gives $h_0=h_5=0$ and $h_1=h_4$. Here $w=1$, and (+)
		gives $h_Q=xy^4+x^4y$, which also satisfies $B_Qh_Q=0$. Hence for
		all $Q\ge2$, the boundary module in Lemma~\ref{lem:4.3}(1) is
		$\bK[\alpha]h_Q$.
		
		\emph{Step 10: The minimum degree and the annihilation bound.}
		By Lemma~\ref{lem:4.3}, the invariant central kernel is generated by
		$u_Q=h_QS_Q$. Equations Lemma~\ref{lem:4.3}(2) and $(\ast)$
		give
		$
		\deg_\alpha u_Q=\deg_\alpha h_Q+Q-1.
		$
		The degree of $h_Q$ is $\deg_\alpha \vec{w}=\delta_Q-Q$, because the two
		components of $h_Q$ have the same coefficients as $\vec{w}$ and $\overleftarrow{w}$
		and no cancellation occurs in $\alpha$ (the coefficients of the
		$x$-monomials are the coordinates of $\vec{w}$, which have degree $\delta_Q-Q$
		after multiplication by $\Delta_Q$). Hence
		$
		\deg_\alpha u_Q=(\delta_Q-Q)+Q-1=\delta_Q-1.
		$
		Every nonzero original central kernel vector yields a nonzero invariant
		one of no larger degree (Lemma~\ref{lem:4.1}), and no nonzero
		polynomial multiple of $u_Q$ has smaller degree. Hence the minimum
		parameter degree is $\delta_Q-1$. The matrices, common denominator
		$\Delta_Q$, the polynomial $h_Q$, and the quotient $S_Q$ all have
		coefficients in $\mathbf{F}_2$, proving existence over the prime
		field.
		
		\emph{Step 11: From $\delta_Q-1$ to $\lambda_{4Q}$.}
		Set $h:=\delta_Q-1$. By the proof above, $h\ge1$. Apply
		Lemma~\ref{lem:2.2} to the central summand $\cM$ of $\bC_q$: since the
		minimum nonzero relation degree of $\cM$ is $h$, we get
		\(
		\mathfrak n^{\,h-1}\cH^0_{\mathfrak n}(\cM)\neq0.
		\)
		Since $\cM$ is a direct summand of $\bC_q$, this implies
		\(
		\mathfrak n^{\,h-1}\cH^0_{\mathfrak n}(\bC_q)\neq0.
		\)
		Let $L_q$ be the least integer with
		$\mathfrak n^{L_q}\cH^0_{\mathfrak n}(\bC_q)=0$. By minimality,
		$L_q>h-1$, i.e.\ $L_q\ge h=\delta_Q-1$. By Lemma~\ref{lem:3.2},
		$\lambda_q\ge L_q$. Combining,
		\(
		\lambda_q\ge L_q\ge\delta_Q-1.
		\)
		Taking $q=4Q$:
		\(
		\lambda_{4Q}\ge\delta_Q-1.
		\)
		
		\emph{Step 12: From $\delta_Q-1$ to $d_Q-1$.}
		The diagonal entry of $\mathbb{M}^{-1}$ is
		$\left(\mathbb{M}^{-1}\right)_{00}=\tau_Qc_Q^2$
		(Lemma~\ref{lem:5.3}), so $\tau_Qc_Q^2$ is the first coordinate of
		$\mathbb{M}^{-1}\vec{e_0}$. Therefore the monic reduced denominator
		$\DD_Q$ of $\tau_Qc_Q^2$ divides the lcm $\Delta_Q$ of the reduced
		denominators of all coordinates:
		\(
		\DD_Q\mid\Delta_Q.
		\)
		Taking degrees,
		\(
		d_Q=\deg\DD_Q\le\deg\Delta_Q=\delta_Q,
		\)
		so
		\(
		\delta_Q-1\ge d_Q-1.
		\)
		Combining with Step 11:
		\(
		\lambda_{4Q}\ge\delta_Q-1\ge d_Q-1.
		\)
		
		\emph{Step 13: Positivity.}
		The estimate
		\(
		d_Q-1\ge(r-1)Q, d_Q-1>0
		\)
		was proved in Theorem~\ref{thm:main} by combining the formula for
		$d_Q$ in Corollary~\ref{cor:6.2} with the bound $\sigma_r\le17/21$,
		in the cases $r$ even and $r$ odd. For $r\ge2$,
		$(r-1)Q\ge Q>0$; for $r=1$, $d_Q-1=1>0$. Hence $d_Q-1>0$ in all
		cases, and the theorem follows.
	\end{proof}
	
	\subsection{Local forms and regularity}
	
	\begin{corollary}\label{cor:8.1}
		Let $\bS=\bR_{\mm}$ and $\bShat$ be its completion. With the extended ideals $I\bS$ and $I\bShat$, both rings fail the local (LC) assertion. They are normal Cohen--Macaulay local domains of dimension four, and their least annihilation exponents equal $\lambda_q$. Moreover,
		\(
		\bShat\simeq \bk[[x,y,z,s,t]]/(sP+tB).
		\)
	\end{corollary}
	
	\begin{proof}
		The module $\cH^0_{\mm}(\bC_q)$ has finite length and is supported at $\mm$, so localization at $\mm$ preserves it. Local cohomology commutes with localization: for each $a$, the image of the action of $\mm^a$ is detected after localization. The map $\bS\to\bShat$ is faithfully flat, so tensoring the \v{C}ech complex with $\bShat$ commutes with cohomology and
		\(
		\cH^0_{\mm\bShat}(\bC_q\otimes_{\bS}\bShat)\simeq \cH^0_{\mm}(\bC_q)\otimes_{\bS}\bShat.
		\)
		Faithful flatness detects whether the image of the action of $\mm^a$ vanishes. Hence the least annihilation exponents agree. Normality and the Cohen--Macaulay property pass to localization. The ring $\bS$ is excellent, and the completion of an excellent normal local ring is normal.
	\end{proof}
	
	\begin{corollary}\label{cor:8.2}
		For every extension $\bK/\bk$, the annihilation exponents for $\bR_{\bK}=\bR\otimes_{\bk}\bK$ and $\bI\bR_{\bK}$ equal those for $\bR,\bI$. In particular, these exponents are the same over all fields of characteristic two.
	\end{corollary}
	
	\begin{proof}
		Scalar extension is faithfully flat. The \v{C}ech complex gives $\cH^0_{\mm \bR_{\bK}}(\bC_q\otimes_{\bk}\bK)\simeq \cH^0_{\mm}(\bC_q)\otimes_{\bk}\bK$. Faithful flatness detects vanishing.
	\end{proof}
	
	\begin{corollary}\label{cor:8.3}
		For $q=2^e$, $e\ge 3$,
		\(
		\eend \cH^0_{\mm}(\bR/\bI^{[q]})\ge(e+3)\frac q4.
		\)
	\end{corollary}
	
	\begin{proof}
		We proved at the end of Theorem~\ref{thm:main} the master bound $
		d_Q-1\ge(r-1)Q$.	Let $h$ be the minimum nonzero central polynomial kernel degree. By
		Theorem~\ref{thm:7.1}, $h=\delta_Q-1$, and this quantity is related to
		the master bound   by
		$
		h=\delta_Q-1\ge d_Q-1\ge(r-1)Q.
		$
		Hence $h\ge(r-1)Q$. By Lemma~\ref{lem:2.2}, the central summand $\cM$
		of $\bC_q$ has nonzero $\nn$-torsion in $\bT$-degree at least $h-1$.
		Its total degree is larger by $6Q+1$, so Lemma~\ref{lem:3.2} gives
		\[
		\eend\cH^0_{\mm}(\bC_q)\ge 6Q+h\ge 6Q+(r-1)Q=(r+5)Q.
		\]
		Since $q=2^e$ and $e=r+2$, we have $(r+5)Q=(e+3)q/4$. So,
		$
		\eend\cH^0_{\mm}(\bR/\bI^{[q]})\ge(e+3)\frac q4.
		$
	\end{proof} 
	
	\begin{remark}\label{rem:ambient}
		The quotient $\bR/\bI=\bk[s,t]$ is regular, whereas the ambient ring is singular at the generic point of $V(\bI)$:
		\(
		\bR_\bI\simeq \bk(s,t)[x,y,z]_{(x,y,z)}/(P+(t/s)B).
		\)
		This local ring has dimension two and embedding dimension three.
	\end{remark}
	\section*{Acknowledgements}
	The author  used
 ChatGPT and deepseek, this assisted  in exploring  computations in the appendix.

\end{document}